\documentclass{amsart}
\usepackage{amssymb, amsmath, amssymb, enumerate, verbatim}
\usepackage{epsfig,graphicx}
\usepackage{hyperref}
\hypersetup{colorlinks=true,citecolor=black,linkcolor=black,anchorcolor=black,filecolor=black,menucolor=black,urlcolor=black,pdftitle={Margulis Lemma},pdfauthor={V. Kapovitch, B. Wilking},pdfdisplaydoctitle}
\newtheorem{mainthm}{Theorem}[]\newtheorem{maincor}[mainthm]{Corollary}
\newtheorem{thm}{Theorem}[section]
\newtheorem*{problem}{Problem}

\newtheorem*{thm*}{Theorem} 
\newtheorem{cor}[thm]{Corollary}
\newtheorem{lem}[thm]{Lemma}

\newtheorem{prop}[thm]{Proposition}
\newtheorem{sublem}[thm]{Sublemma}
\theoremstyle{definition}\newtheorem{stp}{Step}\newtheorem{ste}{Step}
\newtheorem{step}{Step}\newtheorem{example}{Example}
\newtheorem{defin}[thm]{Definition}
\newtheorem*{rems*}{Remarks}
\theoremstyle{remark}
\newtheorem{rem}[thm]{Remark}

\newcommand{\N}{\mathbb{N}} 
\newcommand{\Q}{\mathbb{Q}}
\newcommand{\R}{\mathbb{R}}
\newcommand{\Z}{\mathbb{Z}}\newcommand{\pe}{\alpha}
\newcommand{\tg}{\tilde{g}}\newcommand{\gN}{\mathsf{N}} 
\newcommand{\Sph}{\mathbb{S}}

\newcommand{\Mx}{\mathop{\rm Mx}}\newcommand{\Mxr}{\mathrm {Mx} _r}
\newcommand{\Mxrh}{\mathrm {Mx} _\rho}
\newcommand{\Mxtrh}{\mathrm{Mx} _{2\rho}}

\newcommand{\conv}{^{_{^{^{G-H}}}}{\hspace*{-1.7em}\longrightarrow} }

\newcommand{\gA}{\mathsf{A}}
\newcommand{\tc}{\tilde{c}}

\newcommand{\gF}{\mathsf{F}}
\newcommand{\G}{\mathsf{G}}

\newcommand{\hG}{\hat{\mathsf{G}}}
 \newcommand{\gS}{\mathsf{S}}
\newcommand{\gH}{\mathsf{H}}

\newcommand{\ty}{\tilde{y}}
\newcommand{\tdelta}{\tilde{\delta}}
\newcommand{\fN}{\mathbb{N}}

\DeclareMathOperator{\GL}{GL}
\DeclareMathOperator{\Ric}{Ric}

\DeclareMathOperator{\dt}{dt}

\DeclareMathOperator{\Or}{O}\DeclareMathOperator{\Iso}{Iso}

\newcommand{\gT}{\mathsf{T}}
\newcommand{\hgN}{\hat{\mathsf{N}}}
\newcommand{\gL}{\mathsf{L}}
\newcommand{\gK}{\mathsf{K}}
\DeclareMathOperator{\vol}{vol}

\DeclareMathOperator{\pr}{pr}

\DeclareMathOperator{\id}{id} 
\DeclareMathOperator{\Tor}{Tor}
\DeclareMathOperator{\im}{im}
 \DeclareMathOperator{\diam}{diam}
\DeclareMathOperator{\rank}{rank}

\newcommand{\ml}{\langle}                     % Riemannian metric (left )
\newcommand{\mr}{\rangle}                     % Riemannian metric (right)

\newcommand{\tM}{\tilde{M}}

\newcommand{\tf}{\tilde{f}}
\newcommand{\eps}{\varepsilon}\newcommand{\beps}{\bar{\varepsilon}}
\newcommand{\hg}{\hat{g}}\newcommand{\hgL}{\hat{\mathsf{L}}}
\newcommand{\tphi}{\tilde{\phi}}
\newcommand{\bsigma}{\bar{\sigma}}

\newcommand{\tiota}{\tilde{\iota}}

\newcommand{\ta}{\tilde{a}}
\newcommand{\tb}{\tilde{b}}\newcommand{\teps}{\widetilde{\varepsilon}}
\newcommand{\tK}{\widetilde{K}}\newcommand{\tD}{\widetilde{D}}
\newcommand{\ene}{\end{equation} }
\newcommand{\Out}{\mathrm{Out}}
\newcommand{\ba}{\begin{eqnarray}}
\newcommand{\ea}{\end{eqnarray}}
\newcommand{\ban}{\begin{eqnarray*}}
\newcommand{\ean}{\end{eqnarray*}}

\newcommand{\lG}{\mathsf{G}}

\newcommand{\hGamma}{\hat{\Gamma}}
\newcommand{\Ll}{\mathfrak{l}}

\newcommand{\tp}{\tilde{p}}

\newcommand{\tq}{\tilde{q}}
\newcommand{\Aut}{\mbox{\rm Aut}}
\newcommand{\Hess}{\mathrm{Hess}}

\newcommand{\fint}{{-}\hspace*{-1.05em}\int}
\newcommand{\sfint}{{-}\hspace*{-0.90em}\int}
\def\co{\colon\thinspace}

\def\G{\Gamma}
\def\b{\beta}

\def\a{\alpha}
\def\g{\gamma}\def\tp{\tilde{p}}
\newcommand{\tx}{\tilde{x}}
\newcommand{\tX}{\tilde{X}}

\newcommand{\hM}{\hat{M}}\newcommand{\hY}{\hat{Y}}
\newcommand{\hp}{\hat{p}}
\newcommand{\tN}{\tilde{N}}

\newcommand{\bC}{\bar{C}}

\newcommand{\bGamma}{\bar{\Gamma}}

\def\GH{\textrm{G-H}}

\newtheorem{remark}[thm]{Remark}
\newenvironment{rmk}{\begin{remark}\rm}{\end{remark}}
\newtheorem{Fact}[thm]{Fact}

\newtheorem{Nothing}[thm]{$\!\!\!$}

\makeindex
\begin{document}
%\dedicatory{}
%\subjclass{}
%\keywords{diameter rigidity, positive curvature,
%indices of closed geodesics, $P_l$-manifolds}
%\thanks{}
\abovedisplayskip=6pt plus3pt minus3pt \belowdisplayskip=6pt
plus3pt minus3pt
\title[Structure of fundamental groups]{Structure of fundamental groups of manifolds with Ricci curvature bounded below}

\rm
\thanks{\it The first author was supported in part by a Discovery grant from NSERC}
\author{ Vitali Kapovitch and Burkhard Wilking}
%\date{}
\maketitle

%\begin{abstract}
%\end{abstract}
\setcounter{page}{1} \setcounter{tocdepth}{0}
%\tableofcontents

%%%%%%%%%%%%%%%%%%%%%%%%%%%%%%%%%%%%%%%%%%%%%%%%%%%%%%%%%%%%%%%%%%%%%%%%%%%
%%%%%%%%%%%%%%%%%%%%%%%%%%%%%%%%%%%%%%%%%%%%%%%%%%%%%%%%%%%%%%%%%%%%%%%%%%%
%%%%%%%%%%%%%%%%%%%%%%%%% Introduction %%%%%%%%%%%%%%%%%%%%%%%%%%%%%%%%%%%%
%%%%%%%%%%%%%%%%%%%%%%%%%%%%%%%%%%%%%%%%%%%%%%%%%%%%%%%%%%%%%%%%%%%%%%%%%%%
%%%%%%%%%%%%%%%%%%%%%%%%%%%%%%%%%%%%%%%%%%%%%%%%%%%%%%%%%%%%%%%%%%%%%%%%%%%
%\section{Introduction}

The main result of this paper is the following theorem which settles  a conjecture of Gromov. 
\begin{mainthm}[Generalized Margulis Lemma]\label{thm: intro margulis}\label{intro: margulis}
In each dimension $n$ there  are positive constants $C(n)$ and $\eps(n)$ 
such that the following holds for any complete $n$-dimensional Riemannian manifold $(M,g)$ 
with $\Ric>-(n-1)$ on a metric ball $B_1(p)\subset M$.
The image of the natural homomorphism 
\[ \pi_1\bigl(B_{\eps}(p),p\bigr)\rightarrow \pi_1\bigl(B_1(p),p\bigr)\] 
contains a nilpotent subgroup $\gN$ of index $\le C(n)$. 
Moreover, $\gN$ has a nilpotent basis of length at most $n$.
\end{mainthm}

We call a generator system $b_1,\ldots,b_n$ of a group 
$\gN$ a nilpotent basis if 
the commutator $[b_i,b_j]$ is contained in the subgroup $\ml b_1,\ldots,b_{i-1}\mr $
for  $1\le i<j\le n$.
Having a nilpotent basis of length $n$ implies in particular $\rank(\gN)\le n$. 
We will also show that equality in this inequality can only occur 
if $M$ is homeomorphic to an infranilmanifold, see Corollary~\ref{cor: rank n}.

In the case of a sectional curvature bound 
the theorem is due to Kapovitch, Petrunin and Tuschmann~\cite{KPT},
based on an earlier version which was proved by Fukaya and Yamaguchi~\cite{FY}.

In the case of Ricci curvature a weaker form of the theorem was 
stated by Cheeger and Colding ~\cite{CC}. The main difference is that 
-- similar to Fukaya and Yamaguchi's theorem --
no uniform bound on the index of the subgroup is provided. 
Cheeger and Colding  never wrote up the details of the proof,
and relied on some claims in \cite{FY}
stating that their results would carry over from 
lower sectional curvature bounds to lower Ricci curvature bounds if 
certain structure results would be obtained.
But compare also
Remark~\ref{rem: checol case} below.

\begin{maincor}\label{intro: almost nonneg} Let $(M,g)$ be a compact manifold with $\Ric>-(n-1)$ 
and $\diam(M)\le \eps(n)$ then $\pi_1(M)$ contains a nilpotent subgroup
$\gN$ of index $\le C(n)$. Moreover, $\gN$ has a nilpotent basis of
length $\le n$. 
\end{maincor}

One of the tools used to prove these results is

\begin{mainthm}\label{intro: finite generation}
Given $n$ and $D$ there exists $C$ such that for any $n$-manifold 
with $\Ric\ge -(n-1)$ and $\diam(M,g)\le D$, the fundamental group
$\pi_1(M)$ can be generated by at most $C$ elements. 
\end{mainthm}
This estimate was previously proven by Gromov \cite{G1} under the stronger assumption of a lower sectional curvature bound $K\ge -1$.
Recall that a conjecture of Milnor states that the fundamental group 
of an open manifold with nonnegative Ricci curvature 
is finitely generated. Although Theorem~\ref{intro: finite generation}
is far from a solution to that problem, the  
Margulis Lemma immediately implies the following.

\begin{maincor}\label{cor: open} Let $(M,g)$ be an open $n$-manifold with 
nonnegative Ricci curvature. Then 
$\pi_1(M)$ contains a nilpotent subgroup $\gN$ 
of index $\le C(n)$ such that any finitely 
generated subgroup of $\gN $ has a nilpotent basis of length 
$\le n$.
\end{maincor}

Of course, by work of Milnor~\cite{Mil1} and Gromov~\cite{G6},  the corollary is well known 
(without uniform bound on the index) in the case of finitely 
generated fundamental groups.
We should mention that by work of Wei \cite{Wei} (and an extension in \cite{Wil1})
every finitely generated
virtually  nilpotent 
group appears as fundamental group of an open manifold with positive Ricci curvature 
in some dimension. Corollary~\ref{cor: open} also implies 
that the first $\Z_p$-Betti number of $M$ is finite for any prime $p$, 
which was previously only known for rational coefficients.

{The proofs of our results} are based on the structure results of Cheeger 
and Colding for limit spaces of manifolds with lower Ricci curvature bounds~\cite{Col3, Col4, CC, CC1, CC2, CC3}. 
This is also true for the proof of following new tool
which  can be considered as  a Ricci curvature replacement of Yamaguchi's 
fibration theorem as well as a replacement of the gradient
flow of semi-concave functions used by Kapovitch, Petrunin and Tuschmann. 

\begin{mainthm}[Rescaling theorem]\label{intro: rescale} Suppose a sequence of Riemannian  
$n$-manifolds
$(M_i,p_i)$ with $\Ric\ge -\tfrac{1}{i}$ converges in the pointed Gromov--Hausdorff topology 
to the Euclidean space $(\R^k,0)$ with $k<n$. 
Then after passing to a subsequence 
there is a subset $G_1(p_i)\subset B_1(p_i)$, a rescaling sequence $\lambda_i\to \infty$ 
and a compact metric space $K\neq \{pt\}$ such that
\begin{enumerate}
\item[$\bullet$]  $\vol(G_1(p_i))\ge(1-\tfrac{1}{i})\vol(B_1(p_i))$.
\item[$\bullet$] For all $x_i\in G_1(p_i)$ 
the isometry type of the limit of any convergent subsequence of 
$(\lambda_i M_i,x_i)$ is given by $K\times \R^k$.
\item[$\bullet$]
For all $x_i, y_i \in G_1(p_i)$ we can find a diffeomorphism 
$f_i\colon M_i\rightarrow M_i$ such that 
$f_i$ subconverges in the weakly measured Gromov--Hausdorff sense 
to an isometry of the Gromov--Hausdorff limits
\[
f_\infty\colon \lim_{\GH,i\to\infty} (\lambda_iM_i,x_i)\longrightarrow \lim_{\GH,i\to\infty}(\lambda_iM_i,y_i).
\]
\end{enumerate}
\end{mainthm}

We will prove a technical generalization of this theorem
in section~\ref{sec: rescaling}. 
It will be important for the proof of the Margulis Lemma
that the diffeomorphisms $f_i$ are in a suitable 
sense close to isometries on all scales.
The precise concept is called zooming in property and is defined in 
section~\ref{sec: zoom}. There we also recall 
the concept of weakly measured Gromov--Hausdorff convergence of maps 
(Lemma~\ref{lem: iso}).

The diffeomorphisms $f_i$ will be composed out of gradient flows 
of harmonic functions arising in the analysis of Cheeger and Colding.
Their $L^2$-estimates on the Hessian's of these functions 
play a crucial role.

Like in Fukaya and Yamaguchi's paper 
the idea of the proof of the Margulis Lemma 
is to consider a contradicting sequence $(M_i,g_i)$.  By a fundamental observation of Gromov, the set of complete $n$-manifolds with $\Ric\ge -(n-1)$ is precompact in Gromov--Hausdorff topology. Therefore one can assume that the contradicting sequence converges to a (possibly very singular) space $X$.
One then uses various rescalings and normal coverings of $(M_i,g_i)$ 
in order to find contradicting sequences 
converging  to higher and higher dimensional spaces. 
One of the  differences to Fukaya and Yamaguchi's approach  is that 
if we pass to a normal covering we endow the cover with generators
of the deck transformation group in order not to lose information. 
Since after rescaling the displacements of these deck transformations 
converge to infinity, this approach seems bound to failure. 
However, using the rescaling theorem we are able to 
alter the deck transformations by composing them with a sequence of diffeomorphisms 
which are isotopic to the identity and which have the zooming in property. 
With these alterations we are able to keep much more information 
on the action by conjugation of long homotopy classes on very short 
ones. Since the altered deck transformations still converge
in a weakly measured sense to isometries, this allows one to eventually 
rule out the existence of a contradicting sequence.

For the proof of the Margulis Lemma we do not need any 
sophisticated structure 
results on the isometry group of the limit space 
except for the relatively elementary Gap Lemma~\ref{lem: gap}
guaranteeing that generic orbits of the limit group 
are locally path connected.
However, for the following application of the Margulis Lemma 
a recent structure result of Colding and Naber~\cite{CoNa} is key. 
They showed that the isometry group of a limit space
with lower Ricci curvature bound has no small subgroups
and thus, by the small subgroup theorem of Gleason, Montgomery and  Zippin ~\cite{MoZi}, is a Lie group.

\begin{mainthm}[Compact Version of the Margulis Lemma]\label{normal margulis}
Given $n$ and $D$ there are  positive constants 
$\eps_0$  and $C$ such that the following holds: 
If $(M,g)$ is a compact  $n$-manifold $M$ with $\Ric>-(n-1)
$ and $\diam(M)\le D$, then  there is $\eps\ge \eps_0$ and a normal subgroup $\gN\lhd \pi_1(M)$ 
such that for all $p\in M$:
\begin{enumerate}
\item[$\bullet$] the image of 
$\pi_1(B_{\eps/1000}(p),p)\rightarrow \pi_1(M,p)$ contains $\gN$,
\item[$\bullet$] the index of $\gN$ in the image of $\pi_1(B_{\eps}(p),p)\rightarrow \pi_1(M,p)$
is $\le C$ and
\item[$\bullet$] $\gN$ is a nilpotent group which has a nilpotent basis of length $\le n$.
\end{enumerate}
\end{mainthm}

If $D$, $n$, $\eps_0$ and $C$ are given, there is an effective way
to limit the number of possibilities for the quotient group 
$\pi_1(M,p)/\gN$, see Lemma~\ref{lem: cw}. 
In fact we have the following finiteness result. We recall 
that the torsion elements $\Tor(\gN)$ of a nilpotent group $\gN$ form a subgroup.

\begin{mainthm}\label{intro: pi1 structure}
\begin{enumerate}
\item[a)] For each $D>0$ and each dimension $n$ there 
are finitely many groups $\gF_1,\ldots, \gF_k$
such that the following holds: If $M$ is a compact $n$-manifold with
$\Ric>-(n-1)$ and $\diam(M)\le D$, then there is
a nilpotent normal subgroup $\gN\lhd \pi_1(M)$ 
with a nilpotent basis of length $\le n-1$ and $\rank(\gN)\le  n-2$
 such that $\pi_1(M)/\gN \cong \gF_i$ for suitable $i$. 
\item[b)] In addition to a) one can choose 
a finite collection of irreducible rational  representations 
$\rho_i^j\colon \gF_i\rightarrow \GL(n_i^j,\Q)$ 
($j=1,\ldots,\mu_i, i=1,\ldots,k$) such that 
for a suitable choice of the isomorphism $\pi_1(M)/\gN\cong \gF_i$ the following holds: There is a chain
of subgroups $\mathrm{Tor}(\gN)=\gN_0\lhd\cdots \lhd \gN_{h_0}=\gN$ 
which are all normal in $\pi_1(M)$ such that $[\gN,\gN_h]\subset \gN_{h-1}$
and $\gN_h/\gN_{h-1}$ is free abelian.  
Moroever, the action of $\pi_1(M)$ on $\gN$ by conjugation 
induces an action of $\gF_i$ on $ \gN_h/\gN_{h-1}$ 
and the induced  representation 
$\rho\colon \gF_i\rightarrow \GL\bigl((\gN_h/\gN_{h-1})\otimes_\Z\Q\bigr)$ 
is isomorphic to $\rho_i^j$ for a suitable $j=j(h)$, $h=1,\ldots,h_0$.
\end{enumerate}
\end{mainthm}

Part a) of Theorem ~\ref{intro: pi1 structure} generalizes a result of Anderson~\cite{And90} who proved finiteness of fundamental groups under the additional assumption of a uniform lower bound on volume.

In the preprint  \cite{Wi11} a partial converse of Theorem~\ref{intro: pi1 structure} 
is proved. In particular it is shown there 
that  for a finite collection of finitely presented groups and any finite collection 
of rational representations one can find $D$ such that each group $\Gamma$ satisfying
the algebraic restrictions described in
a) and b) of the above theorem
with respect to this data contains a finite nilpotent normal subgroup $\gH$ 
such that $\Gamma/\gH$ can be realized as a fundamental group of a $n+2$-dimensional 
manifold with $\diam(M)\le D$ and sectional curvature $|K|\le 1$.

We can also extend the diameter ratio theorem of Fukaya and Yamaguchi
\cite{FY}. 

\begin{mainthm}[Diameter Ratio Theorem]\label{thm: diamratio} \label{thm: diam} 
For $n$ and $D$ there is a $\tD$
such that any compact manifold $M$ with $\Ric\ge -(n-1)$ 
and  $\diam(M)= D$ satisfies: 
If $\pi_1(M)$ is finite, then the diameter of the universal cover $\tM$ of $M$
is bounded above by $\tD$.
\end{mainthm}

In the case of nonnegative Ricci curvature the theorem 
says that the ratio $\diam(\tM)/\diam(M)$ is bounded above.
Fukaya and Yamaguchi's theorem covers the case that $M$ has 
almost nonnegative sectional curvature.
The proof of Theorem~\ref{thm: diam} has some similarities 
to  parts of the proof of Gromov's polynomial growth theorem \cite{G6}.

Part of the paper was written up while the second named author was a Visiting Miller 
Professor at the University of California at Berkeley. He 
would like to thank the Miller institute for support and hospitality.
%
%
%
%%%%%%%%%%%%%%%%%%%%%%%%%%%%%%%%%%%%%%%%%%%%%%%%%%%%%%%%%%%%%%%%%%%%%%%%%%%%%%%%%%%%%%%%%%
% %%%%%%%%%%%%%%%%%%%%%%%%%%%%%%%%%%%%%%%%%%%%%%%%%%%%%%%%%%%%%%%%%%%%%%%%%%%%%%%%%%%%%%%%
%%%%%%%%%%%%%%%%%%%%%%%%%%%%%%%%%%%%%%%%%%%%%%%%%%%%%%%%%%%%%%%%%%%%%%%%%%%%%%%%%%%%%%%%%%%%
%%%%%%%%%%%%%%%%%%%%%%%%%%%%%%%%%%%%%%%%%%%%%%%%%%%%%%%%%%%%%%%%%%%%%%%%%%%%%%%%%%%%%%%%%%%%
%
\section*{Organization of the Paper}
\begin{enumerate}
 \item[{ \ref{sec: pre}.}] Prerequisites\hfill\pageref{sec: pre}
 \item[{ \ref{sec: finite gen}.}] Finite generation of fundamental groups\hfill\pageref{sec: finite gen}
 \item[\ref{sec: zoom}.] Maps which are on all scales close to isometries\hfill\pageref{sec: zoom}
 \item[\ref{sec: rough}.] A rough idea of the proof of the Margulis Lemma\hfill\pageref{sec: rough}
 \item[\ref{sec: rescaling}.] The Rescaling Theorem\hfill\pageref{sec: rescaling}
 \item[\ref{sec: induction}.] The Induction Theorem for C-Nilpotency\hfill\pageref{sec: induction}
 \item[\ref{sec: margulis}.] Margulis Lemma\hfill\pageref{sec: margulis}
 \item[\ref{sec: correction}.] Almost nonnegatively curved manifolds with $b_1(M,\Z_p)=n$\hfill\pageref{sec: correction}
 \item[\ref{sec: finiteness}.] Finiteness Results\hfill\pageref{sec: finiteness}
 \item[\ref{sec: diam}.] The Diameter Ratio Theorem\hfill\pageref{sec: diam}
\end{enumerate}

We start in section~\ref{sec: pre} with prerequisites.
We have included a subsection on notational conventions.
Next, in section~\ref{sec: finite gen}   we will prove  Theorem~\ref{intro: finite generation}.

Section~\ref{sec: zoom} is somewhat technical.  We define 
the zooming in property, prove the needed properties, and 
provide two somewhat similar construction methods.  
This section serves mainly as a preparation for the proof
of the rescaling theorem. 

We have added a short section~\ref{sec: rough} 
in which we sketch a rough idea of the proof of the Margulis Lemma.

In section~\ref{sec: rescaling} the refined  rescaling theorem 
(Theorem~\ref{thm: rescale}) is stated and proven.
In Section~\ref{sec: induction} we put things together 
and provide a proof of the Induction Theorem, which 
has Corollary~\ref{intro: almost nonneg} as its immediate consequence. 

The Margulis Lemma follows from the Induction Theorem 
in two steps which are somewhat similar to Fukaya and Yamaguchi's approach. 
Nevertheless, we included all details in section~\ref{sec: margulis}. 
We also show that the nilpotent group $\gN$ in the Margulis Lemma 
can only have rank $n$ if the underlying manifold is homeomorphic to 
 an infranilmanifold.

Section~\ref{sec: correction} uses lower sectional curvature bounds. 
We give counterexamples to a theorem of Fukaya and Yamaguchi ~\cite{FY}
stating that almost nonnegatively curved 
$n$-manifolds with first $\Z_p$-Betti number equal to $n$ have to be tori, provided 
$p$ is sufficiently big.
We show that instead  these manifolds  are nilmanifolds 
and that every nilmanifold covers another nilmanifold
with maximal first $\Z_p$ Betti number.

Section~\ref{sec: finiteness} contains the proofs of
Theorem~\ref{normal margulis} and Theorem~\ref{intro: pi1 structure}. 
Finally, Theorem~\ref{thm: diam} is proved in section~\ref{sec: diam}.

%%%%%%%%%%%%%%%%%%%%%%%%%%%%%%%%%%%%%%%%%%%%%%%%%%%%%%%%%%%%%%%%%%%%%%%%%%%%%%%%%%%%%%%%%%%%%%%%
%%%%%%%%%%%%%%%%%%%%%%%%%%%%%%%%%%%%%%%%%%%%%%%%%%%%%%%%%%%%%%%%%%%%%%%%%%%%%%%%%%%%%%%%%%%%%%%%%%
%%%%%%%%%%%%%%%%%%%%%%%%%%%%%%%%%%%%%%%%%%%%%%%%%%%%%%%%%%%%%%%%%%%%%%%%%%%%%%%%%%%%%%%%%%%%%%%%%
%%%%%%%%%%%%%%%%%%%%%%%%%%%%%%%%%%%%%%%%%%%%%%%%%%%%%%%%%%%%%%%%%%%%%%%%%%%%%%%%%%%%%%%%%%%%%%%%%

\section{Prerequisites}\label{sec: pre}

\subsection{Short basis.} We will use the following construction due to Gromov \cite{Gro1}.

Given a manifold $(\tM,\tp)$  and  a group $\Gamma$ acting properly discontinuously
and isometrically 
on $\tM$ one can define a short basis of the action of $\Gamma$ at $\tp$ as follows:

For $\gamma\in \Gamma$  we will refer to 
$|\gamma|= d(\tp,\gamma(\tp))$ as the norm or the length of $\gamma$.
Choose $\gamma_1\in \Gamma$ 
with  the minimal norm in $\Gamma$.
Next choose $\gamma_2$ to have 
minimal norm in $\Gamma\backslash \langle\gamma_1\rangle$. 
On the $n$-th step choose 
$\gamma_n$ to have minimal norm in 
$\Gamma\backslash \langle\gamma_1,\gamma_2,...,\gamma_{n-1}\rangle$.
The sequence $\{\gamma_1,\gamma_2,...\}$ 
is called a {\it short basis} of $\Gamma$ at $\tp$. 
In general, the number of elements of a short basis can be finite or infinite.
In the special case of the action of the fundamental group $\pi_1(M,p)$  on the universal cover  $\tM$ of $M$
one speaks of the short basis of  $\pi_1(M,p)$. 
For any $i>j$ we have $|\gamma_i|\le |\gamma_j^{-1}\gamma_i|$.

While a short basis at $\tp$ is not unique 
its length spectrum $\{|\gamma_1|,|\gamma_2|,\ldots \}\subset\ \R$ is unique. 
That is the main reason why the non-uniqueness will not matter in any of the proofs in this article and will largely be suppressed.

If $\tM/\Gamma$ is a closed manifold, then  the short basis is 
finite and $|\gamma_i|\le 2\diam(\tM/\Gamma)$.

\begin{comment}
\begin{rmk}\label{rem:short}
Suppose $\g_1,\ldots,\g_n,\ldots$ is a short basis of $\G$
and put $\G_i=\ml \g_1,\ldots,\g_i\mr$. 
Then for any $g\in \G\setminus \G_i$ the following holds:
for any finite sequence $p=x_0,\ldots,x_h=gp \in \G\star p$ 
there is one $l$ with $d(x_l,x_{l+1})\ge |\gamma_{i+1}|$.
This property clearly persists under equivariant Gromov--Hausdorff convergence.
\end{rmk}
\end{comment}

\subsection{Ricci curvature}
Throughout this paper
we will use the notation $\sfint$ to denote the average integral.
We will use the following results of Cheeger and Colding.

\begin{thm}[Splitting Theorem, \cite{CC}]\label{alm-split}
Let $(M^n_i,p_i)\conv (X,p)$ with  $\Ric(M_i)\ge -\frac{1}{i}$.
Suppose $X$ has a line. Then $X$ splits isometrically as $X\cong Y\times \R$.
\end{thm}
We will also need the following corollary of the stability theorem.

\begin{thm}\label{noncol}\cite{CC1}
Let $(M^n_i,p_i)\conv (\R^n,0)$  with $\Ric_{M_i}>-1$. Then
 $B_R(p_i)$ is contractible in $B_{R+1}(p_i)$ 
for all $i\ge i_0(R)$.
\end{thm}

Even more important for us is the following theorem 
which is closely linked with the proof of the splitting theorem for limit spaces.

\begin{thm}\label{hess-ineq}\cite{CC2} Suppose $(M^n_i,p_i)\conv (\R^k,0)$ with $\Ric_{M_i}\ge -1/i$. Then there exist harmonic functions $b^i_1,\ldots,b^i_k\co B_2(p_i)\to \R$ such that

\begin{equation}\label{e:lip}
|\nabla b^i_j|\le C(n) \mbox{ for all $i$ and $j$ and}
\end{equation}
\begin{equation}\label{e:sum}
\hspace*{1em}\fint_{B_1(p_i)} \sum_{j,l} \bigl|<\nabla b^i_j, \nabla b^i_l>-\delta_{j,l}\bigr|+\sum_j \|\Hess_{b^i_j}\|^2\,\,d\mu_i\longrightarrow 0\hspace*{1em} \mbox{ as } i\to\infty.
\end{equation}

Moreover, the maps $\Phi^i=(b^i_1,\ldots, b^i_k)\co M_i\to \R^k$ provide $\eps_i$-Gromov--Hausdorff approximations between
$B_1(p_i)$ and $B_1(0)$ with $\eps_i\to 0$.

\end{thm}
The functions $b^i_j$ in the above theorem are constructed as follows.   Approximate Busemann functions $f_j$ in $\R^k$  given by
$f_j=d(\cdot, N_ie_j))-N_i$  are lifted  to $M_i$ using Hausdorff approximations to corresponding functions $f^i_j$. Here  $e_j$ is the j-th coordinate vector in the standard  basis of $\R^k$ and $N_i\to\infty$ sufficiently slowly so that
$d_{G-H}(B_{N_i}(p_i), B_{N_i}(0))\to 0$ as $i\to\infty$.
The functions $b^i_j$ are obtained by solving the Dirichlet problem on $B_2(p_i)$  with $b^i_j|_{\partial B_2(p^i)}=f^i_j|_{\partial B_2(p^i)}$.

We will need a weak type 1-1 estimate for manifolds with lower Ricci curvature bounds which is a well-known consequence of the doubling inequality \cite[p. 12]{Stein}.

\begin{lem}[Weak type 1-1 inequality]\label{lem:mf-est}\label{lem: weak11}
Suppose $(M^n,g)$ has $\Ric\ge -(n-1)$ and let $f\co M\to\R$  be a nonnegative function.  Define $\Mx_{\rho} f(p):=\sup_{r\le 
\rho}\sfint_{B_r(p)}f$ for $\rho\in (0,4]$ and put $\Mx f(p)=\Mx_{2} f(p)$.
Then the following holds
\begin{enumerate}[(a)]
\item If $f\in L^\a(M)$ with $\a\ge 1$ then $\Mx_\rho f$ is finite almost everywhere.
\item If $f\in L^1(M)$ then $\vol\bigl\{x\mid \Mx_\rho f(x)>c\bigr\}\le \frac{C(n)}{c} \int_M f$ for any $c>0$.
\item If $f\in L^\a(M)$ with $\a>1$ then $\Mx_\rho f\in L^\a(M)$ and $||\Mx_\rho f||_\a\le C(n,\a)||f||_\a$.
\end{enumerate}
\end{lem}
We will also need the following inequality (for any $\a> 1$) which is an immediate consequence of the definition of $\Mx$ and H\"older inequality.
\begin{equation}\label{eq:mf-holder}
[\Mxrh f(x)]^\a\le \Mxrh (f^\a)(x).
\end{equation}
Indeed,
for any $r\le \rho$ we have that $\Mxrh (f^\a)(x)\ge \sfint_{B_r(x)}f^\a\ge \bigl(\sfint_{B_r(x)} f\bigr)^\alpha$
and the inequality follows by  taking the supremum over all $r\le \rho$.

As a consequence of Bishop Gromov one has 
\begin{eqnarray}\label{mx r rho}
\Mxrh f (y) &\le& C(n)\Mxtrh f(x)+ \Mxr f(y)\hspace*{1em}\mbox{ for $d(x,y)\le r\le \rho\le 2$ }
\end{eqnarray}
In fact, if the left hand side is not bounded above by the second summand 
of the right hand side then
$ \Mxrh f(y)= \sfint_{B_{r_1}(y)} f\,d \mu $ for some 
$r_1\in [r,\rho]$ and 
\[\Mxrh f(y)\le \tfrac{\vol(B_{2r_1}(x))}{\vol(B_{r_1}(y))}\fint_{B_{2r_1}(x)} f(p)\,d\mu(p)
\le C(n) \Mxtrh f(x).\]

Next we claim that if $f\in L^\a(M)$ with $\a>1$, then for any $x\in M$ we have the following {\it pointwise} estimate
\begin{equation}\label{eq:mmf}
\Mxrh((\Mxrh f)^\a)(x)\le C_2(n,\a)\Mxtrh (f^\a)(x)\hspace*{1em} \mbox{ for $\rho\in (0,2]$}.
\end{equation}
In order to show this 
we may assume that there is an $r\in (0,\rho]$ 
with 
\begin{eqnarray*}
\Mxrh((\Mxrh f)^\a)(x)\hspace*{-0.7em}&=& \fint_{B_r(x)}(\Mxrh f)^\a(y)\,d\mu(y)\\
&\overset{\eqref{mx r rho}}{\le}& \fint_{B_r(x)}\bigl(\Mxr f(y) +C(n)\Mxtrh f(x)\bigr)^\a\,d\mu(y)\\
&\le & 2^\a C(n)^\a(\Mxtrh f)^\a(x)+ 2^\a\fint_{B_r(x)} \bigl(\Mxr f(y)\bigr)^\a\, d\mu(y)\\
&
{\le}& 2^\a C(n)^\a(\Mxtrh f)^\a(x)
+\tfrac{2^\a C(n,\alpha)^\a}{\vol(B_{r}(x))}
\int_{B_{2r}(x)}f^\a(y)\,d\mu(y)\\
&\overset{\eqref{eq:mf-holder}}{\le} & C_2(n,\alpha)(\Mxtrh f^\a)(x), 
\end{eqnarray*}
where in order to deduce the third inequality from Lemma~\ref{lem: weak11} c)
we used  that the integrals on 
either side only depend on values of $f$ in $B_{2r}(x)$.

By combining \eqref{eq:mf-holder} and \eqref{eq:mmf} we obtain
\begin{equation}\label{est: mxmx}
\Mxrh[\Mxrh (f)](x)\le \Bigl(\Mxrh[(\Mxrh(f))^\a](x)\Bigr)^{1/\a}\le \Bigl(C_2(n,\alpha)\Mxtrh[f^\a](x)\Bigr)^{1/\a}
\end{equation}
for $\a>1$ and $\rho\in (0,1]$. 
Finally applying this to the function $g=f^{\frac{\b}{\a}}$ 
for $\a,\,\b>1$ gives
\begin{equation}\label{eq:mmf1}
\Mxrh[(\Mxrh (f^{\b/\a}))^\a](x)\le C_2(n,\a)\Mxtrh (f^\b)(x).
\end{equation}

We will also make use of the so-called  segment inequality of Cheeger and Colding which says that in average
the integral of a nonnegative function along all geodesics 
in a ball can be estimated by the $L^1$ norm of the function.

\begin{thm}[Segment inequality, \cite{CC}] \label{seg-ineq} Given $n$ and $r_0$ 
there exists $\tau=\tau(n,r_0)$ such that the following holds.
Let $\Ric(M^n)\ge -(n-1)$ and let $g\co M\to \R^+$ be a nonnegative function. 
Then for $r\le r_0$

\[
\fint_{B_r(p)\times B_r(p)}\int_{0}^{d(z_1,z_2)}g(\gamma_{z_1,z_2}(t))\,dt\,d\mu(z_1)d\mu(z_2)\le \tau\cdot r\cdot \fint_{B_{2r}(p)}g(q)\,d\mu(q),
\]
where $\gamma_{z_1,z_2}$ denotes a minimal geodesic from $z_1$ to $z_2$.
\end{thm}

Finally we will  need the following observation
\begin{lem}[Covering Lemma]\label{lem:av}\label{lem: covering}\label{lem: cover}
There exists a constant $C(n)$ such that the following holds. Suppose $(M^n,g)$ has $\Ric_g\ge -(n-1)$. Let $f\co M\to\R$ be a nonnegative function,
$p\in M$, $\pi\co \tM\to M$  the universal cover of $M$, $\tp\in\tM$ 
a lift of $p$,
 and let $\tf=f\circ\pi$. Then
\[
\fint_{B_1(\tp)}\tf\le C(n) \fint_{B_1(p)}f.
\]
\end{lem}
\begin{proof}
We choose a measurable section $j\co B_1(p)\to B_1(\tp)$, i.e.  $\pi(j(x))=x$ for any $x\in B_1(p)$. Let $T= j(B_1(p))$. 
Then we obviously have that $\diam(T)\le 2$ and 
\[
\fint_{B_1(p)}f= \fint_{T}\tf.
\]

Let $S$  be the union of $ g(T)$ over all $g\in\pi_1(M)$ such that $g(T)\cap B_1(\tp)\ne\emptyset$.
It is obvious from the triangle inequality that $S\subset B_3(\tp)$. It is also clear that
\[
\fint_{B_1(p)}f= \fint_{S}\tf
\]
and $B_1(\tp)\subset S$. Lastly, notice that by Bishop--Gromov relative volume comparison $\vol B_3(\tp)\le C(n)\vol B_1(\tp)$ for some universal
 constant $C(n)$. Since   $B_1(\tp)\subset S\subset B_3(\tp)$, 
we also have $\vol S\le C(n)\vol B_1(\tp)$ and hence
\[
\fint_{B_1(\tp)}\tf\le C(n) \fint_S\tf=C(n)  \fint_{B_1(p)}f.
\]
\end{proof}

It is easy to see that the above proof generalizes to an arbitrary cover of $M$.

\subsection{Equivariant Gromov--Hausdorff convergence}
Let $\Gamma_i$ be a   closed subgroup of the isometry group 
of a metric space $X_i$ and $\tp_i\in X_i$, $\,i\in \fN$.
 If $(X_i,\tp_i)\conv (Y,\tp_\infty)$, then after passing to a subsequence one can 
find a closed subgroup $\lG\subset \Iso(Y)$ such that  
$(X_i,\Gamma_i,\tp_i)\to (Y,\lG,\tp_\infty)$
in the equivariant Gromov--Hausdorff topology. 
For  details  we refer the reader to \cite{FY}.

\begin{defin}\label{def: open sub} A sequence of subgroups $\Upsilon_i\subset \Gamma_i$ 
is called uniformly open,  if there is some $\eps>0$ 
such that $\Upsilon_i$ contains the set 
\[
\bigl\{g\in \Gamma_i\mid d(g q,q)< \eps\,\mbox{ for all $q\in B_{1/\eps}(\tp_i)$}\bigr\}\mbox{ for all $i$.}
\]
$\Gamma_i$ is called uniformly discrete if $\{e\}\subset \Gamma_i$ is uniformly open. 
We say the sequence $\Upsilon_i$ is boundedly generated if there is some $R$ such 
that $\Upsilon_i$ is generated by 
\[
\{g\in \Upsilon_i\mid d(g \tp_i,\tp_i)< R\}.
\]
\end{defin}

\begin{lem}\label{lem: open sub} Let $\Upsilon_i^j\subset \Gamma_i$ be uniformly open 
with  $\Upsilon_i^j\to \Upsilon_\infty^j\subset \lG$, $j=1,2$.  
\begin{enumerate}
            \item[a)] $\Upsilon_i^1\cap \Upsilon_i^2$ is uniformly open 
and converges to $\Upsilon_\infty^1\cap \Upsilon_\infty^2$. 
\item[b)] If $g_i\in \Gamma_i$ converges 
to $g_\infty\in \lG$ 
then $g_i\Upsilon_i^1g_i^{-1}$ is uniformly open 
and converges to $g_\infty \Upsilon_\infty^1g_\infty^{-1}$.
\item[c)] Suppose  in addition that $\Upsilon_i^j$ is boundedly generated, $j=1,2$.
If  $\Upsilon_\infty^1\cap \Upsilon_\infty^2$ 
has finite index $H$ in $\Upsilon_{\infty}^1$, then 
$\Upsilon_i^1\cap \Upsilon_i^2$ has index $H$ 
in $\Upsilon_{i}^1$ for all large $i$.
\end{enumerate}
\end{lem}

The proof is an easy exercise.

\begin{example}[$\Z_l$ actions converging to a $\Z^2$-action.]
Consider a $2$-torus $\gT^2_k$ given by a Riemannian product 
$\gS^1\times \gS^1$ where each of the factors has length $k$.
Put $l=k^2+1$ and let
$\Z/l\Z$ act on $\gT^2_k$ { with the generator acting by 
$\bigl(e^{{\frac{2\pi i}{k^2+1}}},e^{{\frac{k2\pi i}{k^2+1}}}\bigr)$.}
In the equivariant Gromov--Hausdorff topology the 
action will converge (for $k\to\infty$)
to the standard $\Z^2$ action on $\R^2$.

{One can, of course, define a similar sequence of  actions on $\gS^3\times \gS^3$ 
but in this case the actions will not be uniformly cocompact.}
\end{example}

\begin{rem}\begin{enumerate}
           \item [a)] The example shows that it is difficult to relate  an open subgroup in the limit to corresponding subgroups in the sequence. 
\item[b)] If $g_1^i,\ldots,g_{h_i}^i$ is
 a generator system 
of some subgroup $\Gamma'_i\subseteq \Gamma_i$, then
one can consider --
after passing to a subsequence -- a different limit construction to get 
a subgroup $\lG'\subset \lG$: 
For some fixed $R$ consider all words $w$ in $g_1^i,\ldots,g_{h_i}^i$ 
with the property $w\star \tp_i\in B_R(\tp_i)$.
This defines a subset in the Cayley graph of $\Gamma'_i$. 
We enlarge the subset by adding to each  vertex all the neighboring edges and
consider now all elements in $\Gamma_i'$ represented by words
in the identity component of the enlarged subset.
Passing to the limit gives a closed subset $S_R\subset \lG$ of the limit group 
and one can now take the limit of $S_R$ for $R\to\infty$. 

This sort of word limit group can be smaller than the regular limit group 
as the above example shows.
Although it depends  on the choice of the generator system,  
there are occasions where this limit behaves 
more natural than the usual. 
However, this idea is used only indirectly in the paper.
          \end{enumerate}
\end{rem}

\subsection{Notations and conventions}
%Throughout the paper we will use the following conventions.
\begin{enumerate}
\item[$\bullet$] 
As already mentioned  $\sfint_Sf\,d\mu$ stands
 for $\tfrac{1}{\vol(S)}\int_S f\,d\mu$. We always 
use the Riemannian measure in integrals. It will often be suppressed when
 variables of integration are clear.
\item[$\bullet$]
$\Mxrh(f)(x)$ denotes the $\rho-$maximum function of $f$ evaluated at $x$,
 see Lemma~\ref{lem: weak11} for the definition and
$\Mx(f)(x):=\Mx_2(f)(x)$.
\item[$\bullet$]
A map $\sigma\colon X\rightarrow Y$ between locally compact complete inner metric spaces is called a submetry 
if $\sigma(B_r(p))=B_r(\sigma(p))$ for all $r>0$ and $p\in X$.
Notice that then a geodesic in $Y$
can be lifted 'horizontally' to a geodesic in $X$.
\item[$\bullet$] If $S\subset \gF$ is a subset of a group, $\ml S\mr$ 
denotes the subgroup generated by $S$. 
\item[$\bullet$]
If $g_1$ and $g_2$ are elements in a group, 
$[g_1,g_2]:=g_1g_2g_1^{-1}g_2^{-1}$ denotes the commutator. For subgroups 
$\gF_1,\gF_2$ we put $[\gF_1,\gF_2]:=\ml\{[f_1,f_2]\mid f_i\in\gF_i\}\mr$.
\item[$\bullet$] 
Generators  $b_1,\ldots,b_n\in \gF$ of a group are called a nilpotent basis
if $[b_i,b_j]\in \ml b_1,\ldots,b_{i-1}\mr $ for $i<j$.
\item[$\bullet$] $\gN\lhd \gF$ means: $\gN$ is a normal subgroup of $\gF$. 
\item[$\bullet$]
For a Riemannian manifold $(M,g)$ and $\lambda>0$ we let $\lambda M$ denote the Riemannian manifold 
$(M,\lambda^2 g)$. 
\item[$\bullet$]
For a limit space $Y$ 
a tangent cone $C_pY$ at $p$ is some Gromov--Hausdorff limit 
of $ (\lambda_i Y,p)$ for some $\lambda_i\to \infty$.
Tangent cones are not always metric cones and are  not necessarily unique.
\item[$\bullet$] We will sometimes use 
the concept of measured Gromov--Hausdorff convergence. 
Recall that for any sequence $(M_i,g_i,p_i)\conv (Y,p_\infty)$
with lower Ricci curvature bound the normalized Riemannian measures $\tfrac{d\mu_{g_i}}{\vol(B_1(p_i))}$ 
subconverge to a limit measure on $Y$, see \cite{CC1}.
\item[$\bullet$]
For a limit space $Y$ of manifolds with lower Ricci curvature bound 
a regular point $p\in Y$ is a point  all of whose 
tangent cones at $p$  are given by  $\R^{k_p}$. By a result of Cheeger and Colding ~\cite{CC1} these 
points have full measure (with respect to any limit measure on $Y$) and are thus dense.
\item[$\bullet$]
For two sequences of pointed metric spaces $(X_i,p_i)$, $(Y_i,q_i)$ 
and maps $f_i\colon X_i\rightarrow Y_i$
we use the notation 
$f_i\colon [X_i,p_i]\rightarrow [Y_i,q_i]$ 
to indicate that $f_i$ subconverges in 
the weakly measured sense (cf. Lemma~\ref{lem: iso})
to a map from the pointed Gromov--Hausdorff limit of $(X_i,p_i)$ 
to the pointed Gromov--Hausdorff limit of $(Y_i,q_i)$. 
It usually does not mean that $f_i(p_i)=q_i$.
However, if  this is the case we write
$f_i\colon (X_i,p_i)\rightarrow (Y_i,q_i)$. 
\item[$\bullet$] If a group $\lG$ acts on a metric space 
we say $g\in \lG$ displaces $p\in X$ by $r$ if $d(p,gp)=r$. We will denote the orbit of $p$ by $\lG\star p$.
\item[$\bullet$] Gromov's short 
generator system (or short basis for short)
of a fundamental group  is defined at the beginning of 
section~\ref{sec: pre}.
\end{enumerate}

%%%%%%%%%%%%%%%%%%%%%%%%%%%%%%%%%%%%%%%%%%%%%%%%%%%%%%%%%%%%%%%%%%%%%%%%%%%%%%%%%%%%%
%%%%%%%%%%%%%%%%%%%%%%%%%%%%%%%%%%%%%%%%%%%%%%%%%%%%%%%%%%%%%%%%%%%%%%%%%%%%%%%%%%%%%
%%%%%%%%%%%%%%%%%%%%%%%%%%%%%%%%%%%%%%%%%%%%%%%%%%%%%%%%%%%%%%%%%%%%%%%%%%%%%%%%%%%%%%
%%%%%%%%%%%%%%%%%%%%%%%%%%%%%%%%%%%%%%%%%%%%%%%%%%%%%%%%%%%%%%%%%%%%%%%%%%%%%%%%%%%%%%

\section{Finite generation of fundamental groups }\label{sec: finite gen}
\begin{lem}[Product Lemma]\label{prodlem}
Let $M_i$ be a sequence of manifolds with $\Ric_{M_i}>-\eps_i\to 0$ satisfying
\begin{enumerate}
\item[$\bullet$]
$\overline{B_{r_i}(p_i)}$ is  compact for all $i$ with $r_i\to\infty$  and $p_i\in M_i$, 
\item[$\bullet$]  for every $i$  and $j=1,\ldots,k$ there are harmonic functions $b_j^i\colon B_{r_i}(p_i)\rightarrow \R$ which are   $L$-Lipschitz and fulfill
\[
\fint_{B_{R}(p_i)}\sum_{j,l=1}^k |<\nabla b^i_j, \nabla
b^i_l>-\delta_{jl}|+\sum_{j=1}^k \|\Hess_{b^i_j}\|^2\,d\mu_i\rightarrow 0 \mbox{ for all $R>0$.}
\]
\end{enumerate}
Then  $(B_{r_i}(p_i),p_i)$ subconverges in the pointed Gromov--Hausdorff topology 
to a metric product  $(\R^k\times X,p_\infty)$ for some metric space $X$. 
Moreover, $(b_1^i,\ldots,b_k^i)$ converges to the projection onto the Euclidean factor.
\end{lem}

The lemma remains true if one just has a uniform lower Ricci curvature bound and
one can also prove a local version of the lemma if $r_i=R$ is a fixed number. 
However, the above  version suffices for our purposes.

\begin{proof}
The main problem is to prove this in the case of $k=1$. 
Put $b_i=b_1^i$. 
After passing to a subsequence we may assume that 
$(B_{r_i}(p_i),p_i)$ converges to some limit space $(Y,p_\infty)$. 
We also may assume that $b_i$ converges to an $L$-Lipschitz map 
$b_\infty\colon Y\rightarrow \R$.

\begin{ste} $b_\infty$  is  $1$-Lipschitz.
\end{ste}

This is essentially immediate from the segment inequality. 
Let $x,y\in Y$ be arbitrary.  Choose $R$ so large that 
$x,y\in B_{R/4}(p_\infty)$ and
let $x_i,y_i\in B_{R/2}(p_i)$ be sequences converging to $x$ and $y$.

For a fixed $\delta<<R$ consider 
all minimal geodesic from points in $B_\delta(x_i)$
to points in $B_{\delta}(y_i)$. 
For each geodesic $\gamma_{pq}$ consider $\int_0^1||\nabla b_i|-1|(\gamma_{pq}(t))\,dt$. 
Since  by Bishop--Gromov the set $B_\delta(x_i)\times B_\delta(x_i)$ 
fills up a fixed portion of the volume of $B_R(p_i)^2$,
 the segment inequality (Theorem~\ref{seg-ineq})
 implies that 
there is a minimal geodesic 
$\gamma_{i}$ 
from a point in $B_\delta(x_i)$ to a point in $B_{\delta}(y_i)$ 
with  $\int_0^1 ||\nabla b_i|-1|(\gamma_{i}(t))\,dt=h_i\to 0$. 
Thus \[
|b_i(x_i)-b_i(y_i)|\le 2L\delta+ L(\gamma_i)(1+h_i)\\
\]
and
$
|b_\infty(x)-b_\infty(y)|\le  2(L+1)\delta +d(x,y).
$
The claim follows as
$\delta$ was arbitrary.

\begin{ste} $b_\infty\colon Y\rightarrow \R$ is a submetry. 
\end{ste}

As the gradient flow $\phi_t^i$ of $b_i$ is measure preserving
and $||\nabla b_i|-1|\le ||\nabla b_i|^2-1|$, 
\begin{eqnarray*}
 \fint_{B_R(p_i)}\!\int_0^{t_0}||\nabla b_i|-1|(\phi_t^i(q))\,dt\,d\mu_i(q)\!\!
&=&\!\! \int_0^{t_0}\!\!\fint_{\phi_{t}^i(B_{R}(p_i)) } ||\nabla b_i|-1|(q)\,d\mu_i(q)\,dt\\
&\le&\!\!C
\int_0^{t_0}\!\!\fint_{B_{R+t_0L}(p_i) }\!\!\! ||\nabla b_i|-1|(q)\,d\mu_i(q)\,dt\\
&\to& 0,
\end{eqnarray*}
where the inequality holds with some constant $C$ satisfying
$\tfrac{\vol(B_{R+t_0L}(p_i))}{\vol(B_R(p_i))}\le C$.
Thus for fixed $R$ and $t_0$ there is $\delta_i\to 0$
such that 
most gradient curves $c_i$ of $b_i$ 
defined on $[0,t_0]$ and starting in $B_R(p)$ 
satisfy $\int_{0}^{t_0}||\nabla b_i|-1|(c_i(t))\,dt\le \delta_i$.
 For these curves we have $|L(c_i)-t_0|\le \delta_i$ 
and $b_i(c_i(t_0))-b_i(c_i(0))\ge t_0-\delta_i$. 
These curves subconverge to unit speed curves $c$ in the limit with $b_{\infty}(c(t_0))-b_{\infty}(c(0))=t_0$. Since $b_\infty$ 
is $1$-Lipschitz, { all such limit curves} must be geodesics. 
In the limit these curves go through every point and thus $b_\infty$ 
is a submetry.

Notice that the proof of Step 2 also 
shows that 
for the submetry $b_\infty$ lines can be lifted to $Y$.
By the splitting theorem of Cheeger and Colding (Theorem \ref{alm-split}) 
the result follows for $k=1$ and the generalization to arbitrary $k$
is an easy exercise.
\end{proof}

\begin{lem}\label{lem: prep gap} Let $(Y_i,\tp_i)$ 
be an inner metric space space endowed with an action 
of a closed subgroup $\lG_i$ of its isometry group, $i\in \N\cup \{\infty\}$.
Suppose $(Y_i,\lG_i,\tp_i)\to (Y_\infty,\lG_\infty,\tp_\infty)$
 in the equivariant Gromov--Hausdorff topology.
Let $\lG_i(r)$ denote the subgroup 
generated by those elements that displace $\tp_i$ by at most  
$r$, $i\in \N\cup \{\infty\}$.
Suppose there are $0\le a<b$ with
$\lG_{\infty}(r)=\lG_{\infty}(\tfrac{a+b}{2})$ for all $r\in 
(a,b)$. 

Then there is some sequence $\eps_i\to 0$ such that 
$\lG_{i}(r)=\lG_i(\tfrac{a+b}{2})$ for all 
$r\in (a+\eps_i, b-\eps_i)$.
\end{lem}

\begin{proof}
Suppose on the contrary 
we can find $g_i\in \lG_i(r_2)\setminus \lG_i(r_1)$ 
for fixed $r_1<r_2\in (a,b)$.
Without loss of generality $d(\tp_i,g_i\tp_i)\le r_2$.

Since $g_i\not\in \lG_i(r_1)$ 
it follows  that for any finite sequence of orbit points 
$\tp_{i}=x_1,\ldots,x_h=g_i \tp_i\in \lG_i\star \tp_i$  
there is one $j\in\{1,\ldots,h\}$ with $d(x_j,x_{j+1})\ge r_1$.
Clearly this property  carries 
over to the limit and implies
that $g_\infty\in \lG_\infty(r_2)$ is not contained in
$\lG((r_1+a)/2)$ -- a contradiction.
\end{proof}

\begin{lem}\label{lem:short}
Suppose $(M^n_i,q_i)$ converges to $(\R^k\times K,q_\infty)$ where $\Ric_{M_i}\ge -1/i$ and $K$ is compact.  
Assume the action of $\pi_1(M_i)$ on the universal cover $(\tM_i,\tq_i)$
converges to a 
limit action of a group $\lG$ on some limit space $(Y,\tq_\infty)$. 

Then $\lG(r)=\lG(r')$ for all $r,r'> 2\diam(K)$.
\end{lem}

\begin{proof} Since
$Y/\lG$ is isometric to $\R^k\times K$, it follows that
there is a submetry $\sigma\colon Y\rightarrow \R^k$. 
Hence lines in $\R^k$ can be lifted to lines 
in $Y$ and it is immediate from the splitting theorem (Theorem~\ref{alm-split}) 
that this submetry 
has to be linear, that is, for any geodesic $c$ in $Y$ 
the curve $\sigma\circ c$ is affine linear. 
We get 
a splitting $Y=\R^k\times Z$ such that $\lG$ acts trivially on $\R^k$ 
and on $Z$ with compact quotient $K$.
We may think of $\tq_\infty$ as a point in $Z$.
For $g\in \lG$ consider a mid point $x\in Z$ 
of $\tq_\infty$ and $g \tq_\infty$. 
Because  $Z/\lG=K$ we can find 
$g_2\in \lG$ with $d(g_2\tq_\infty,x)\le \diam(K)$. 
Clearly 
\begin{eqnarray*} d(\tq_\infty,g_2\tq_\infty)&\le& \tfrac{1}{2}d(\tq_\infty,g\tq_\infty)+\diam(K)\\
d(\tq_\infty,g_2^{-1}g\tq_\infty)&=&d(g_2\tq_\infty, g\tq_\infty)\le \tfrac{1}{2}d(\tq_\infty,g\tq_\infty)+\diam(K).
\end{eqnarray*}

This proves $\lG(r)\subset \lG\bigl(r/2+\diam(K)\bigr)$ and the lemma follows.
\end{proof}

\begin{lem}[Gap Lemma]\label{lem: gap}
Suppose we have a sequence of manifolds
$(M_i,p_i)$ with $\Ric_{M_i}\ge -(n-1)$
converging to some limit space $(X,p_\infty)$ and
suppose that the limit point $p_\infty$ is regular. 
Then there is a sequence 
$\eps_i\to 0$ and a number $\delta>0$ such that the following holds. 
If $\gamma_1,\ldots,\gamma_{l_i}$ 
is a short basis of $\pi_1(M_i,p_i)$ 
then either $|\gamma_j|\ge \delta$ or 
$|\gamma_j|<\eps_i$.

Moreover, if 
the action of $\pi_1(M_i)$ on the universal cover $(\tM_i,\tp_i)$
converges to an action of the limit group 
$\lG$ on $(Y,\tp_\infty)$, 
then the orbit $\lG\star \tp_\infty$ is locally path connected. Here 
$\tp_i$ denotes a lift of $p_i$.
\end{lem}

\begin{proof}
That the orbit is locally path connected is a consequence 
of the following\\[2ex]
{\bf Claim.} There is a $\delta>0$ such that 
for all points  $x\in\lG\star \tp_{\infty}$ 
with $d(\tp_\infty,x)<\delta$ we can find $y\in \lG\star \tp_{\infty}$ 
with $\max\{d(\tp_\infty,y), d(y,x)\}\le 0.51\cdot d(\tp_\infty,x)$.\\[2ex]
To prove the claim we argue by contradiction 
and assume that for some $\delta_h>0$ converging to $0$ 
we can find $g_h\in \lG$ 
with $d(g_h\tp_\infty,\tp_\infty)=\delta_h$ such that
for all 
$a\in \lG$ with $d(a\tp_\infty,\tp_{\infty})\le 0.51\delta_h$ 
we have $d(a\tp_\infty,g_h\tp_{\infty})> 0.51\delta_h$.

After passing to a subsequence we can assume that
$(\tfrac{1}{\delta_h}Y,\tp_\infty)$ converges to a tangent cone 
$(C_{p_\infty}Y,o)$ and that
the action of $\lG$ on $\tfrac{1}{\delta_h}Y$ converges 
to an action of  some group $\gK$ on $C_{p_\infty}Y$.
 Let  $g_\infty\in \gK$ be the limit of $g_h$. 
Clearly $d(o,g_\infty o)=1$ and for all 
$a\in \gK$ with $d(o,a o)< 0.51$ we have 
$d(go,a o)\ge 0.51$. In particular the orbit $\gK\star o$ is not convex.

Because  $X=Y/\lG$ the quotient 
$C_{\tp_\infty}Y/\gK$ is isometric to 
$\lim_{h\to \infty} (\tfrac{1}{\eps_h}X,p_\infty)$ 
which by assumption is isometric to some Euclidean space
$\R^k$.

As in the proof of Lemma~\ref{lem:short} it follows
that $C_{\tp_\infty}Y$ is isometric to $\R^k\times Z$ 
and the orbits of the action of $\gK$ are given by $v\times Z$ 
where $v\in \R^k$. In particular, the 
$\gK$ orbits are convex -- a contradiction. \\[2ex]
Clearly, the claim implies that  
$\lG\star\tp_\infty\cap B_{r}(\tp_\infty)$
is  path connected for $r<\delta$.
In fact, it is now easy to construct 
a H\"older continuous path from $\tp_\infty$ 
to any point in $\lG\star\tp_\infty\cap B_{r}(\tp_\infty)$.
Of course, the claim also implies 
$\lG(\eps)=\lG(\delta)$ for all $\eps \in (0,\delta]$.
Therefore the first part of
Lemma~\ref{lem: gap} follows from Lemma~\ref{lem: prep gap}.
\end{proof}

Theorem ~\ref{intro: finite generation} will follow from the following slightly more general result.

\begin{thm}\label{thm: finite generation} Given $n$ and $R$ there is a constant $C$ 
such that the following holds. 
Suppose $(M,g)$ is an $n$-manifold, $p\in M$,  $\overline{B_{2R}(p)}$ is  compact 
and $Ric>-(n-1)$ on $B_{2R}(p)$. Furthermore, we assume that $\pi_1(M,p)$
is generated by loops of
length $\le R$. 
Then $\pi_1(M,p)$ can be generated by $C$ loops of length $\le R$. 

Moreover, there is a point $q\in B_{R/2}(p)$ such that 
 any Gromov short 
generator system of $\pi_1(M,q)$ has at most $C$ elements.
\end{thm}

\begin{proof}
We first want to prove the second part of the theorem.
We consider a Gromov short generator system 
$\gamma_1,\ldots, \gamma_k$ of $\pi_1(M,q)$.
With $|\gamma_i|\le |\gamma_{i+1}|$ for some $q\in B_{R/2}(p)$. 
Clearly $|\gamma_i|\le 2R$
and 
it  easily follows from Bishop--Gromov relative volume comparison that there are effective a priori 
bounds (depending only on $n,R$ and $r$) for the number of short generators with length $\ge r$.

We will argue by contradiction. 
It will be convenient to modify the assumption on $\pi_1$ being boundedly generated.  
We call $(M_i,p_i)$ a contradicting sequence if the following holds 
\begin{enumerate}
\item[$\bullet$] $\overline{B_3(p_i)}$  is compact and $\Ric>-(n-1)$ on $B_3(p_i)$.
\item[$\bullet$] For all $q_i\in B_1(p_i)$ the 
number of short 
generators of $\pi_1(M_i,q_i)$ of length $\le 4$ is larger than $2^i$. 
\end{enumerate}

Clearly it suffices to rule out the existence of a contradicting sequence. 
We may assume that $(B_3(p_i),p_i)$ converges to some limit space $(X,p_\infty)$. 
We put
\[
\dim(X)= \max\bigl\{k\mid\mbox{ there is a regular $x\in B_{1/4}(p_\infty)$ 
with $C_xX\cong \R^k$}\bigr\}.
\]
We argue by reverse induction on $\dim(X)$.
We start our induction at $\dim(X)\ge n+1$. 
It is well known that this can not happen so there is nothing to prove.
The induction step is subdivided in two substeps.\\[2ex]
{\bf Step 1.} For any contradicting sequence $(M_i,p_i)$ 
converging to $(X,p_\infty)$ 
there is a new contradicting sequence converging to $(\R^{\dim(X)},0)$. \\[2ex]
Choose $q_i\in B_{1/4}(p_i)$ converging to  some point 
$q_\infty\in B_{1/4}(p_\infty)$ with $C_{q_\infty}X=\R^{\dim(X)}$.
After passing to a subsequence we can assume that 
for all $x_i\in B_{1/4}(q_i)$ the number of short generators 
of $\pi_1(M_i,x_i)$ of length $\le 4$ is at least $3^i$.

Since the number of short generators of length $\in [\eps,4]$ is bounded 
by some a priori constant, we can find a sequence $\lambda_i\to \infty$ very slowly 
such that 
for every $x\in B_{1/\lambda_i}(q_i)$ the number of short 
generators of $\pi_1(M_i,x)$ of length $\le 4/\lambda_i$ is at least $2^i$
and 
$(\lambda_iM_i,q_i)$ converges to $(C_{q_\infty}X,0)=(\R^{\dim(X)},0)$. 
Replacing $M_i$ by $\lambda_iM_i$ and $p_i$ by $q_i$ gives a new contradicting sequence, as claimed.\\[2ex] 
{\bf Step 2.} If there is a contradicting sequence converging to $(\R^k,0)$,
then we can find a contradicting sequence converging 
to a space whose dimension is larger than $k$. \\[2ex]
Let $(M_i,p_i)$ be a contradicting sequence converging to $(\R^k,0)$. 
We may assume without loss of generality that for some $r_i\to \infty $ and $\beps_i \to 0$
the Ricci curvature on $B_{r_i}(p_i)$ is bounded below by $-\beps_i$ 
and that  $\overline{B_{r_i}(p_i)}$ is compact. 
In fact, one can run through the arguments of the first 
step, to see that, after passing to a subsequence, a rescaling by $\lambda_i\to \infty$ 
(very slowly) is always possible.

By Theorem~\ref{hess-ineq} we can find 
a harmonic map
$(b_1^i,\ldots,b_k^i)\colon B_1(q_i)\to \R^k$
with

\[
\fint_{B_{1}(q_i)}\sum_{j,l=1}^k\bigl( |\ml\nabla b^i_l,\nabla b^i_j\mr-\delta_{lj}| +\|\Hess(b^i_l)\|^2\bigr)\,d\mu_i=\eps_i\to 0\hspace*{0.5em}\mbox{ and}
\]
\[
|\nabla b^i_j|\le C(n). \]

By the weak (1,1) inequality (Lemma~\ref{lem: weak11}) we can find $z_i\in B_{1/2}(q_i)$ 
with 

\[
\fint_{B_{r}(z_i)}\sum_{j,l=1}^k\bigl( |\ml\nabla b^i_l,\nabla b^i_j\mr-\delta_{lj}| +\|\Hess(b^i_l)\|^2\bigr)d\mu_i\le C\eps_i\to 0
\]
for all $r\le 1/4$.
By the Product Lemma~{\eqref{prodlem}}, for any sequence $\mu_i\to \infty$ 
the spaces $(\mu_iB_r(z_i),z_i)$ subconverge to a metric product $(\R^k\times Z,z_\infty)$ for some $Z$ 
depending on the rescaling. 

From Lemmas~\ref{lem: prep gap} and  ~\ref{lem:short}  we deduce that there is a sequence
$\delta_i\to 0$ such that for all $z_i\in B_2(p_i)$ 
the short generator system of 
$\pi_1(M_i,z_i)$ does not contain any elements with length in $[\delta_i, 4]$.
Choose $r_i\le 1$ maximal with the property that
there is $y_i\in B_{r_i}(z_i)$ 
such that the short generator system of $\pi_1(M_i,y_i)$ contains 
one generator of length $r_i$. 
We have seen above that $r_i\le \delta_i \to 0$.

Put $N_i=\tfrac{1}{r_i}M_i$.  By construction,
$\pi_1(N_i,y_i)$  still has $2^i$ short generators of length $\le 1$
for all $y_i\in B_1( z_i) \subset N_i$, 
and there is one with length 1 for a suitable $y_i$. 
By the Product Lemma~{\eqref{prodlem}}, $(N_i,z_i)$ subconverges to a product 
$(\R^k\times Z, z_\infty)$. Lemmas~\ref{lem: prep gap} and  ~\ref{lem:short} 
imply that $Z$ can not be a point
and the claim is proved.

In order to prove the first part of the theorem 
we consider the subgroup of $\pi_1(M,p)$ generated by 
loops of length $\le R/10$. By the second part 
this subgroup can be generated by $C(n,R)$ elements of length 
$\le 2R/5$. Since the number of short generators of $\pi_1(M,p)$ with length 
in $[R/10,R]$ is bounded by some a priori constant the theorem 
follows.
\end{proof}

Finally, let us  mention that there is a measured version of 
Theorem~\ref{thm: finite generation}.

\begin{thm}\label{localbound}
For any $n>1$ and any $\eps\in (0,1)$ there exists $C(n,\eps)$ such that if $\Ric(M^n)\ge-(n-1)$ 
on $B_3(p)$  with $\overline{B_3(p)}$ being compact, then 
there is  a subset $B_1(p)'\subset B_1(p)$ with
$\vol B_1(p)'\ge (1-\eps)\vol B_1(p)$ and  
any short basis of $\pi_1(M,q)$  has at most $C(n,\eps)$ elements of length $\le 1$
for any $q\in B_1(p)'$.
\end{thm}

Since we do not have any applications of this theorem,
we omit its proof.

%%%%%%%%%%%%%%%%%%%%%%%%%%%%%%%%%%%%%%%%%%%%%%%%%%%%%%%%%%%%%%%%%%%%%%%%%%%%%%%%%%%%%%%%
%%%%%%%%%%%%%%%%%%%%%%%%%%%%%%%%%%%%%%%%%%%%%%%%%%%%%%%%%%%%%%%%%%%%%%%%%%%%%%%%%%%%%%%%%
%%%%%%%%%%%%%%%%%%%%%%%%%%%%%%%%%%%%%%%%%%%%%%%%%%%%%%%%%%%%%%%%%%%%%%%%%%%%%%%%%%%%%%%%
%%%%%%%%%%%%%%%%%%%%%%%%%%%%%%%%%%%%%%%%%%%%%%%%%%%%%%%%%%%%%%%%%%%%%%%%%%%%%%%%%%%%%%%%

\section{Maps which are on all scales close to isometries.}\label{sec: zoom}

For a map $f\colon X\rightarrow Y$ between metric spaces we define the
distance distortion on scale $r$ by
\begin{equation}\label{def dtr}
\dt_r^f(p,q)=\min\{r, |d(p,q)-d(f(p),f(q))|\}\hspace*{1em}\mbox{ for $p,q\in X$}.
\end{equation}

\begin{defin}\label{def: zoom} Let $(M_i^n,p_i^1)$ and $(N_i^n,p_i^2)$ be two sequences 
of Riemannian manifolds. 
We say that a sequence of diffeomorphisms 
$f_i\colon M_i\rightarrow N_i$ has the zooming in property if the following holds: 
There exist $R_0>0$,  sequences $r_i\to \infty $, $\eps_i\to 0$
 and subsets
$B_{2r_i}(p_i^j)'\subset B_{2r_i}(p_i^j)$ ($j=1,2$)
satisfying 
\begin{enumerate}
\item[a)] $\overline{B_{4r_i}(p_i^{_j})}$  is compact 
and $\Ric_{B_{4r_i}(p_i^{j})}>-R_0$. 
\item[b)] $\vol\bigl(B_1(q)\cap B_{2r_i}(p_i^{_j})'\bigr)\ge (1-\eps_i)\vol(B_1(q))$ for all
$q\in B_{r_i}(p_i^{_j})$. 
\item[c)] For all $p\in B_{r_i}(p_i^1)'$, all $q \in B_{r_i}(p_i^2)'$ 
and all $r\in (0,1]$ we have 
\begin{eqnarray*}
\fint _{B_r(p)\times B_r(p)}\dt_r^{f_i}(x,y) d\mu_i^1(x)d\mu_i^1(y)&\le& r\eps_i\, \mbox{ and }\\
\fint _{B_r(q)\times B_r(q)}\dt_r^{f_i^{-1}}(x,y)
d\mu_i^2(x)d\mu_i^2(y)&\le& r\eps_i.\\
\end{eqnarray*}
\item[d)] There are subsets $S^j_i\subset B_1(p_i^j)$ 
with $\vol(S^j_i)\ge\tfrac{1}{2}\vol\bigl(B_1(p_i^j)\bigr)$  ($j=1,2$) 
and $f(S_i^1)\subset B_{R_0}(p_i^2)$ and $f^{-1}(S_i^2)\subset B_{R_0}(p_i^1)$.
\end{enumerate}
\end{defin}

We will call elements of $B_{2r_i}(p_i^1)'$ {\it good} points
and sometimes use the convention
$B_r(q)':=B_{2r_i}(p_i^j)'\cap B_r(q)$ for all $B_r(q)\subset B_{2r_i}(p_i^j)$.

In the applications we will always have $N_i=M_i$. However, in some instances
$d(p_i^{1},p_i^{2})\to \infty$. 
That is why it might be  helpful to also think  about 
maps between two unrelated pointed manifolds.
If the choice of the base points 
is not clear we will say that
$f_i\colon [M_i,p_i^1]\rightarrow [N_i,p_i^2]$
has the zooming in property. 
Notice that we do not require $f_i(p_i^1)=p_i^2$. 
However, property d) ensures that the base points are respected in a weaker sense.

\begin{lem}\label{lem: iso} Let $(M_i,p_i^1)$, $(N_i,p_i^2)$ and $f_i$ be as above. 
Then, after passing to a subsequence, 
$(M_i,p_i^1)\conv (X_1,p_{\infty}^1)$, $(N_i,p_i^2)\conv (X_2,p_{\infty}^2)$
and
$f_i$ converges in the weakly measured sense to a measure preserving 
isometry $f_{\infty}\colon X_1\rightarrow X_2$, that is 
for each $r>0$ there is a sequence $\delta_i\to 0$  
and subsets $S_i\subset B_r(p_i^1)$ satisfying
\begin{enumerate}
\item[$\bullet$] $\vol( S_i)\ge (1-\delta_i) \vol( B_r(p_i^1))$ and
\item[$\bullet$] $f_{i|S_i}$ is Gromov--Hausdorff close to 
$f_{\infty|B_r(p_{\infty}^1)}$. 
\end{enumerate}
Moreover, $f_i^{-1}$ converges in this sense to $f_{\infty}^{-1}$.
\end{lem}

For the proof we will need the following 

\begin{sublem}\label{sublem:2r}
There exists $C_1(n)$ such that
for any good point $x\in B_{r_i}(p_i^1)'$ and any $r\le 1$ there is
a subset $B_r(x)''\subset B_r(x)$  with 
\[
\vol B_r(x)''\ge (1-C_1\eps_i)\vol B_r(x)\mbox{ and } f_i(B_r(x)'')\subset B_{2r}(f_i(x)).
\]
\end{sublem}
\begin{proof}
{
The proof proceeds by induction on the size of $r$ as follows. It is clear that at good points  $x\in B_{r_i}(p_i^1)'$ the differential of $f_i$ 
has a bilipschitz constant  $e^{C\eps_i}$ with some universal $C$
provided  $\eps_i<1/2$. 
Therefore it is clear that the statement holds for all small $r$. 
Hence it suffices to prove that if it holds for some $r\le 1/10$, then 
it holds for $10r$. By assumption
\[
\fint _{B_{10r}(x)\times B_{10r}(x)}\dt^{f_i}_{10r}(p,q)\,d\mu_i^1(p)d\mu_i^1(q)\le 10r\eps_i.
\]
Furthermore, as long as $C_1\eps_i\le 1/2$ 
our induction assumption implies that there is a subset $S\subset B_r(x)$ 
with $\vol(S)\ge \tfrac{1}{2}\vol(B_r(x))$ and $f_i(S)\subset B_{2r}(f_i(x))$.

By Bishop--Gromov
\[
\fint _{S\times B_{10r}(x)}\dt^{f_i}_{10r}(p,q)d\mu_i^1(p)d\mu_i^1(q)\le C_2(n)r\eps_i
\]
with some universal constant $C_2(n)$. 
Therefore there is a subset $B_{10r}(x)''\subset B_{10r}(x)$
with $\vol(B_{10r}(x)'')\ge (1-C_2(n)\eps_i/2)\vol(B_{10r}(x))$
and
\[
\fint _{S}\dt^{f_i}_{10r}(p,q)d\mu_i^1(p)\le 2r\,\,\mbox{ for all $q\in B_{10r}(x)''$.}
\]
Using that $f_i(S)\subset B_{2r}(f_i(x))$ this clearly implies that
$f_i(B_{10r}(x)'')\subset B_{20r}(f_i(x))$. 
}
Thus the sublemma is valid if we put $C_1(n)=C_2(n)/2$ 
provided we know in addition that $C_1(n)\eps_i\le 1/2$. 
We can remove the upper bound on $\eps_i$ by just putting
$C_1(n)=C_2(n)$.
\end{proof}

Of course, a similar inequality holds for $f_i^{-1}$. 

\begin{proof}[Proof of Lemma~\ref{lem: iso}]
Note that as was observed in the proof of Sublemma~\ref{sublem:2r}, at good points in $B_{r_i}(p_i^1)'$ ($B_{r_i}(p_i^2)'$) the differential of $f_i$ ($f^{-1}_i$) has bilipschitz constant $\le e^{C\eps_i}$ for some universal $C$ provided $\eps_i<1/2$.
Using condition b) from the definition of the zooming in property  this will clearly ensure that $f_i$ and $f_i^{-1}$ converge to {\em measure-preserving } isometries once we establish the following\\[2ex] 
{\bf Claim.} Given $\delta\in (0,1/10)$ there is $i_0$ such that 
$|d(f_i(x_i),f_i(y_i))-d(x_i,y_i)|\le 10\delta$ holds
for all $i\ge i_0$ and
 all $x_i,y_i\in B_{r_i}(p_i^1)'$ with $d(x_i,y_i)<1/2$.\\[2ex]
Consider subsets $B_{\delta}(x_i)''$ and $B_{\delta}(y_i)''$
as in the sublemma. Then
$\vol(B_{\delta}(x_i)'')\ge\tfrac{1}{2} \vol(B_{\delta}(x_i))$
for large $i$ and
combining with Bishop--Gromov gives
\[
\vol(B_1(x_i))^2\le C_2(n,\delta)\vol(B_{\delta}(x_i)'')\vol(B_{\delta}(y_i)'').
\]
Thus,
\[
\fint_{B_{\delta}(x_i)''\times B_{\delta}(y_i)''}\dt_1^{f_i}(p,q)\le 
C_2 \fint_{B_1(x_i)^2}\dt_1^{f_i}(p,q)\le C_2\eps_i.
\]
Choose $i_0$ so large that $C_2\eps_i\le \delta$ for $i\ge i_0$. 
Then for such $i$
we can find $x_i'\in B_{\delta}(x_i)''$ and $y_i'\in B_{\delta}(y_i)''$ 
with $\dt_1^{f_i}(x_i',y_i')\le \delta$. 
Combining with $d(f_i(x_i'),f_i(x_i))\le 2\delta$ and $d(f_i(y_i'),f_i(y_i))\le 2\delta$
we deduce $\dt_1^{f_i}(x_i,y_i)\le 7\delta$ 
as claimed.
\end{proof}

\begin{lem} \label{lem:comp-zoom}\label{lem: composition}
Consider three pointed Riemannnian manifolds 
$(M_i,p_i^1),(N_i,p_i^2)$, $(P_i,p_i^3)$ 
and two sequences of 
diffeomorphisms $f_i\colon M_i\to N_i $ and $g_i\colon N_i\to P_i$ 
with the zooming in property. 
Then $g_i\circ f_i$ also has the zooming in property. 
\end{lem}
\begin{proof}
Let $R>10$ be  arbitrary. By assumption we can find a sequence 
$\eps_i\to 0$ and a subset $B_{2R}(p_i^j)'\subset B_{2R}(p_i^j)$ 
with $\vol(B_{2R}(p_i^j)')\ge (1-\eps_i)\vol(B_{2R}(p_i^j))$ ($j=1,2,3$)
such that the following holds
\begin{eqnarray*}
\fint_{B_r(q)^2}\dt_r^{f_i}(a,b)&\le& \eps_ir\hspace*{1em} \mbox{ for all 
$r\in (0,2]$ and $q\in B_{2R}(p_i^1)'$,}\\
\fint_{B_r(q)^2} \dt_r^{f_i^{-1}}(a,b) + \dt_r^{g_i}(a,b)&\le& \eps_ir\hspace*{1em} \mbox{ 
for all $r\in (0,2]$ and $q\in B_{2R}(p_i^2)'$ and }\\
\fint_{B_r(q)^2}\dt_r^{g_i^{-1}}(a,b)&\le& \eps_ir \hspace*{1em} \mbox{  for all 
$r\in (0,2]$ and $q\in B_{2R}(p_i^3)'$.}
\end{eqnarray*}
In order to get the above inequalities 
for $r\in [1,2]$ we used that $f_i$ and $g_i$ 
converge in the weakly measured sense to an isometry by Lemma~\ref{lem: iso}. Lemma~\ref{lem: iso}
also implies that after a possible adjustment 
of $\eps_i\to 0$ we have
\begin{eqnarray*}
S&:=& B_R(p_i^2)'\cap f_i(B_R(p_i^1)')\cap g_i^{-1}(B_R(p_i^3)')
\mbox{ satisfies }\\
\vol(S)&\ge& (1-\eps_i)\vol(B_R(p_i^2))
\end{eqnarray*}

Similarly without loss of generality
$\vol( f_i^{-1}(S))\ge (1-\eps_i)\vol(B_R(p_i^1))$
and $\vol( g_i(S))\ge (1-\eps_i)\vol(B_R(p_i^3))$. In other words, we may assume 
\[
f_i(B_R(p_i^1)')= B_R(p_i^2)'\mbox{ and }
g_i( B_R(p_i^2)')= B_R(p_i^3)'.
\] 

Next we consider the characteristic function $\chi$
of the set $B_{2R}(p_i^1)\setminus B_{2R}(p_i^1)'$. 
Using Lemma~\ref{lem: weak11} b) and Bishop--Gromov 
we can assume that there is a sequence 
$\delta_i\to 0$ such that $\delta_i>\eps_i$ and
\[
 G_R(p_i^1):=\{ q\in B_R(p_i^1)'\mid \Mx \chi(q)\le \delta_i\}\mbox{ fulfills } 
\]
\begin{eqnarray}\label{eq: volGr}
\vol\bigl(B_R(p_i^1)\setminus G_{R}(p_i^1)\bigr)\le \delta_i\min_{q\in B_{R}(p_i^1)}{\vol(B_1(q))}.
\end{eqnarray}

By Sublemma \ref{sublem:2r}  for all 
$q_i\in  G_{R}(p_i^1)$ and $r\le 1$ 
there is a subset $B_r(q_i)''\subset B_r(q_i)$ 
with 
{$\vol (B_r(q_i)'')\ge (1-C_1\delta_i) \vol B_r(q_i)$}
and $f_i(B_r(q_i)'')\subset B_{2r}(f(q_i))$. 
Using $q_i\in  G_{R}(p_i^1)$ we can actually assume 
$B_r(q_i)''\subset B_{2R}(q_i)'$ provided we replace $C_1$ by $C_2=C_1+1$.
Thus,
\begin{eqnarray*}
\fint_{B_r(q_i)^2}\!\!\dt_r^{g_i\circ f_i}(a,b)\!&\le& 2C_2\delta_ir+
\tfrac{1}{\vol(B_r(q))^2}\int_{(B_r(q_i)'')^2}\!\!\dt_r^{g_i }(f_i(a),f_i(b))+dt_r^{f_i}(a,b)\\
&\le& (2C_2+1) \delta_ir+\tfrac{e^{2nC\eps_i}}{\vol(B_r(q_i))^2}\int_{B_{2r}(f_i(q_i))^2}\dt_r^{g_i}(a,b)\\
&\le& (2C_2+1)\delta_ir+\tfrac{2e^{2nC\eps_i}\vol (B_{2r}(f_i(q_i))^2}{\vol(B_r(q_i))^2}\eps_ir
\le C_3\delta_ir,
\end{eqnarray*}
where we used that the differential of $f_i$ at $q\in B_r(q_i)''\subset B_{2R}(q_i)'$ has a bilipschitz 
constant $\le e^{C\eps_i}$
and that the ratio of $\vol(B_{2r}(f_i(q_i)))$ and $\vol(B_r(q_i))$ 
is for large $i$ bounded by a universal constant. 
The latter statement follows from Bishop--Gromov and
Sublemma~\ref{sublem:2r} applied to $f_i^{-1}$ and the ball
$B_{r}(f_i(q_i))$.

The last inequality holds for all $q_i\in G_R(p_i^1)$, where $R\ge 10$ 
was arbitrary. By the usual diagonal sequence 
argument one can deduce that there is a sequence $R_i\to \infty$ 
and an adjusted sequence $\delta_i\to 0$ such that the above inequality holds for all 
$q_i\in G_{R_i}(p_i^1)$ and in addition we can assume 
that \eqref{eq: volGr} remains valid.

Since everthing can be carried out for $f_i^{-1}\circ g_i^{-1}$ as well 
this finishes the proof.
\end{proof}

The next lemma explains the notion zooming in property. 

\begin{lem}\label{lem: zoom} Let $(M_i,p_i^1),(N_i,p_i^2)$ and $f_i$ be as above. 
Then there are $\rho_i\to \infty$, $\delta_i\to 0$ 
and $T_i^1\subset B_{\rho_i}(p_i^1)$ such that the following holds
\begin{enumerate}
\item[$\bullet$] $\vol(B_{1}(q)\cap T_i^1)\ge (1-\delta_i)\vol(B_1(q))$
for all $q\in B_{\rho_i/2}(p_i^1)$. 
\item[$\bullet$] For any sequence of real numbers $\lambda_i\to \infty$ 
and any sequence $q_i\in {T_i^1} $ 
\[f_i\colon (\lambda_i M_i,q_i)\rightarrow  (\lambda_i N_i, f_i(q_i))\]
has the zooming in property.
We say  that $f_i$  is good on all scales at $q_i$. 
\end{enumerate}
\end{lem}

\begin{proof}
Let $G_i^j=B_{2r_i}(p_i^j)'$ and $B_i^j=B_{2r_i}(p_i^j)\setminus B_{2r_i}(p_i^j)'$.
After adjusting $r_i\to \infty$ and $\eps_i\to 0$
we may assume $\vol(B_i^j)\le \eps_i\vol(B_1(q))$ 
for all $q\in B_{2r_i}(p_i^j)$.
Let $\chi_{i}^j$ be the characteristic function of $B_i^j$. 
By the weak 1-1 inequality (Lemma~\ref{lem: weak11}) there exists a universal $C$ 
such that 
the set
\[
H_i^j:=\bigl\{x\in B_{r_i/2}(p_i^j)\mid \Mx(\chi_{i}^j)(x)\ge \sqrt \eps_i\bigr\}
\]
satisfies
\[
\vol(H_i^j)\le C\sqrt{\eps_i}\vol(B_{1}(q))
\] 
for all $q\in B_{2r_i}(p_i^j)$. We
put $T_i^1:= \bigl(B_{r_i/2}(p_i^1)\backslash H_i^1\bigr) \cap f_i^{-1}\bigl( B_{r_i/2}(p_i^2)\backslash H_i^2\bigr)$ and $T_i^2:=f_i( T_i^1)$.
Using Lemma~\ref{lem: iso} we can find $\rho_i\to \infty$ and $\delta_i\to 0$ 
such
that 
$$
\vol\bigl(B_{\rho_i}(p_i^j)\setminus T_i^j\bigr)\le \delta_i\vol(B_{1}(q))\,\mbox{ for all $q\in B_{\rho_i}(p_i^j)$, $j=1,2$.}
$$
By definition of $T_i^j$ 
\[
\frac{\mathrm{vol}(B_r(q)\cap G_i^j)}{\mathrm{vol}(B_r(q))}\ge 1-\sqrt{\eps_i}\,\,\,\,\mbox{ for all $q\in T_i^j$ and all $r\le 1$, $j=1,2$.}
\]
Let
$\dt^{\lambda_if}_{r}$  denote the distortion
on scale $r$ of $f_i\colon \lambda_iM_i\rightarrow \lambda_iN_i$.
Clearly 
\[
\dt^{\lambda_if}_r(p,q)=\lambda_i\dt_{r/\lambda_i}^{f_i}(p,q).
\]
Thus for all $\lambda_i\to\infty$ and all 
$q_i\in T_i^j$ the map $f_i\colon (\lambda_iM_i,q_i)\rightarrow (\lambda_iN_i,f_i(q_i))$
has the zooming in property.

\end{proof}

\begin{prop}[First main example]\label{thm: ex zoom}\label{prop: first main exa} Let $\pe>1$.
Consider a sequence of $n$-manifolds 
$(M_i,p_i)$ with a fixed lower Ricci curvature bound 
and a sequence of  time dependent vector fields  $X_i^t$ (piecewise constant in time)
with compact support. 
Let $c_i\colon [0,1]\rightarrow B_{r_i}(p_i)$ be an integral curve  of $X_i^t$ 
with $c_i(0)=p_i$ 
Assume that $X_i^t$ is divergence free on $B_{r_i+100}(p_i)$ and that $\overline{B_{r_i+100}(p_i)}$
is compact.
Put 
\[
u_{s,i}(x):=\bigl(\Mx\|\nabla_\cdot X_i^s\|^\pe\bigr)^{1/\pe}(x)
\]
and suppose
\[
\int_0^1u_{t,i}(c_i(t))\,dt=\eps_i\to 0.
\]

Let $f_i=\phi_{i1}$ be the flow of $X^t_i$ evaluated at time $1$.
Then for all $\lambda_i\to \infty$ 
\[f_i\colon (\lambda_iM_i,c(0))\to (\lambda_i M_i,c(1))\] 
has the zooming 
in property. 
Moreover, for any lift $\tf_i\colon \tM_i\rightarrow \tM_i$ 
of $f_i$ to the universal  cover $\tM_i$ of $M_i$ and for any lift $\tp_i\in \tM$ 
of $c(0)=p_i$ the sequence $\tf_i\colon (\lambda_i\tM_i,\tp_i)\rightarrow (\lambda_i\tM_i,\tf_i(\tp_i))$
has the zooming in property as well.
\end{prop}

The proposition remains valid if the assumption on $X_i^t$ being divergence free
is removed. However, the proof is easier in this case and we do
not have any applications of the more general case. We will need the following

\begin{lem}\label{lem:dtr}
There exists (explicit) $C=C(n)$ such that the following holds. 
Suppose  $(M^n,g)$ has $\Ric\ge -1$ and $X^t$ is a vector field 
with compact support, which depends on time (piecewise constant). 
Let $c(t)$ be the integral curve of $X^t$ with $c(0)=p_0\in M$ 
and assume that $X^t$ is divergence free 
on $B_{10}(c(t))$ for all $t\in [0,1]$.

Let $\phi_t$ be the flow of $X^t$. Define the distortion function  $dt_r(t)(p,q)$ 
of the flow
on scale $r$ by the formula
\[
\dt_r(t)(p,q):=\min\Bigl\{r,\max_{0\le\tau\le t}\bigm|d(p,q)-d(\phi_\tau(p),\phi_\tau(q))|\Bigr\}.
\]
Put
$\eps:=\int_0^1\Mx_1(\|\nabla_\cdot X^t\|)(c(t))\,dt$. Then for any $r\le 1/10$ we have

\begin{eqnarray*}
\fint_{B_r(p_0)\times B_r(p_0)}\dt_r(1)(p,q)\,d\mu(p) d\mu(q)\le C r\cdot \eps
\end{eqnarray*}

and there exists $B_r(p_0)'\subset B_r(p_0)$ such that 

\begin{eqnarray*}
\tfrac{
\vol(B_r(p_0)')}{\vol(B_r(p_0))}\,\ge\, (1-C\eps)\,\,\mbox{ and }\,\,
\phi_t(B_r(p_0)')\subset B_{2r}(c(t)) \text{  for all } t\in [0,1].
\end{eqnarray*}

\end{lem}

\begin{proof} We prove the statement for a constant in time
vector field $X^t$. The general case is completely analogous 
except for additional notational problems.

Notice that all estimates are trivial if $\eps\ge \tfrac{2}{ C}$. 
Therefore it suffices to prove the statement with a universal constant $C(n)$ 
for all $\eps\le \eps_0$.   We put $\eps_0=1/2C$ and determine $C\ge 2$ in the process. 
We again proceed by induction on the size of $r$.

Notice that the differential of 
$\phi_s$ at $c(0)$  is bilipschitz 
with bilipschitz constant $e^{\int_0^s \|\nabla_{\cdot} X\|(c(t))dt}
\le 1+2\eps$.
Thus the Lemma holds 
for very small $r$.

Suppose the result holds for some $r/10\le 1/100$. 
It suffices to prove that it then holds for $r$.
By induction assumption we know that for any $t$ there exists 
$B_{r/10}(c(t))'\subset B_{r/10}(c(t))$ such that for any $s\in [-t,1-t]$ we have

\[
\vol (B_{r/10}(c(t))')\ge (1-C\eps)\vol (B_{r/10}(c(t)))\ge \frac{1}{2}\vol (B_{r/10}(c(t)))
\]
and 
\[
\phi_s(B_{r/10}(c(t))')\subset B_{r/5}(c(t+s)),
\]
where  we used $\eps \le \tfrac{1}{2C}$ in the inequality.
This easily implies that $\vol (B_{r/10}(c(t)))$ are comparable for all $t$.
More precisely, for  any $t_1, t_2\in [0,1]$ we have that 
\begin{equation}\label{eq:comp1}
\tfrac{1}{C_0}\vol B_{r/10}(c(t_1))\le \vol B_{r/10}(c(t_2))\le C_0\vol B_{r/10}(c(t_1))
\end{equation}
with a computable universal $C_0=C_0(n)$. 
Put
\[   
h(s)=\fint _{B_{r/10}(c(0))'\times B_r(c(0))}\dt_r(s)(p,q)\,\,d\mu(p)\,d\mu(q),
\]

\begin{eqnarray*}
U_s&:=&\{(p,q)\in B_{r/10}(c(0))'\times B_r(c(0))\mid \dt_r(s)(p,q)<r\},\\
\phi_s(U_s)&:=&\{(\phi_s(p),\phi_s(q))\mid (p,q)\in U_s\}\hspace*{1em}\mbox{and}\\
\dt_r'(s)(p,q)&:=&\limsup_{h\searrow 0} \tfrac{\dt_r(s+h)(p,q)-\dt_r(s)(p,q)}{h}.
\end{eqnarray*}
As $\dt_r(t)\le r$ is monotonously increasing, 
we deduce that if $\dt_r(s)(p,q)= r$, then $\dt_r'(s)(p,q)=0$.
Since $\dt_r(s+h)(p,q)\le 
\dt_r(s)(p,q)+\dt_r(h)(\phi_s(p),\phi_s(q))$ and $\phi_s$ is
 measure preserving, it follows 
\begin{eqnarray*}
h'(s)&\le& \fint_{\phi_s(U_s)}\dt_r'(0)(p,q)\\
&\le&\tfrac{4\vol B_{3r}(c(s))^2}{\vol B_{r/10}(c(0))^2}\fint_{B_{3r}(c(s))^2}
\dt_r'(0)(p,q),
\end{eqnarray*}
where we used that $\phi_s(B_{r/10}(p_0)')^2\subset \phi_s(U_s)\subset B_{3r}(c(s))^2$. 
If $p$ is not in the cut locus of $q$ 
and $\gamma_{pq}\colon [0,1]\rightarrow M$ 
is a minimal geodesic between $p$ and $q$,
then 
\[
dt'_r(0)(p,q)\le d(p,q)\int_0^1\|\nabla_\cdot X\|(\gamma_{pq}(t))\,dt.
\]

Combining the last two inequalities with the segment inequality 
we deduce 

\begin{eqnarray*}
h'(s)&\le& C_1(n)r\fint_{B_{6r}(c(s))}\|\nabla_{\cdot}X\|\\
&\le& C_1(n)r\Mx\!_1\|\nabla_{\cdot}X\|(c(s))
\end{eqnarray*} 
Note that the choice of the constant $C_1(n)$ can be made explicit and  {\it independent} 
of the induction assumption. We deduce $h(1)\le C_1(n)r\eps$ and thus
 the subset 
\[
B_r(p_0)':= \biggr\{ p\in B_r(p_0)\Bigm| \fint_{B_{r/10}(p_0)'}\dt_r(1)(p,q)\,d\mu(q)\le r/2\biggr\}
\]
satisfies 
\begin{eqnarray}\label{vol est1}
\vol(B_r(p_0)')\ge (1-2C_1(n)\eps)\vol(B_r(p_0)).
\end{eqnarray}
It is elementary to check 
that 
\[
\phi_t(B_r(p_0)')\subset B_{2r}(c(t)) \mbox{ for all $t\in [0,1]$.}
\]

Then arguing as before we estimate that 

\[
\fint_{B_r(p_0)'\times B_r(p_0)}dt_r(1)(p,q)d\mu(p) d\mu(q)\le C_2(n)\cdot r\cdot \eps.
\]
Using $\dt_r(1)\le r$ and the volume estimate (\ref{vol est1}) this gives
\[
\fint_{B_r(c(p_0))^2}\dt_r(1)(p,q)\,d\mu(p)\, d\mu(q)\le C_2(n)\cdot r\cdot \eps+2rC_1\eps=: C_3r\eps.
\]
This completes the induction step with $C(n)=C_3$ and $\eps_0=\tfrac{1}{2C_3}$.
In order to remove the restriction $\eps\le \eps_0$ 
one can just increase $C(n)$ by the factor $4$, as indicated at the beginning. 
\end{proof}

\begin{proof}[Proof of Proposition~\ref{prop: first main exa}]
Put $g_{t,i}(x)=\|\nabla_\cdot X_i^t\|(x)$. 
First notice that by \eqref{eq:mf-holder}
\[
\int_0^1\Mx(g_{t,i})(c_i(t))\,dt\le \int_0^1 u_{t,i}(c_i(t))\,dt\to 0.
\]

Let $\lambda_i\to \infty$ and put $r_i=\tfrac{R}{\lambda_i}$, 
where $R>1$ is arbitrary. 
By Lemma~\ref{lem:dtr}
there is $S_i\subset B_{r_i}(c_i(0))$ and  $\delta_i\to 0$
with \begin{enumerate}
\item[$\bullet$] $\vol(S_i)\ge (1-\delta_i)\vol(B_{r_i}(c_i(0)))$  and 
\item[$\bullet$] $\phi_{it}(S_i)\subset B_{2r_i}(c_i(t))$ for all $t$.
\end{enumerate}
In the following we assume that $i$ is so large that 
$\delta_i\le 1/2$, and $r_i\le 1/100$.
As in the proof of  Lemma~\ref{lem:dtr} this easily implies that there is a universal constant 
$C=C(n)$ with 
\[
\tfrac{\vol(B_{2r_i}(c_i(t)))}{\vol(B_{r_i}(c_i(0)))}\le \tfrac{C}{2} \mbox{ for all $t$ and all $i$.}
\]

Using that $\phi_{it|S_i}$ is measure preserving, we deduce

\begin{eqnarray*}
\fint_{S_i}\int_0^1 \Mx \!_1(g_{t,i})(\phi_{it}(p))\,dt\,d\mu_i(p)\!\!\!&\le&
C\int_0^1\fint_{B_{2r_i}(c_i(t))} \Mx\!_1(g_{t,i})(q)\,d\mu_i\,dt(q)\\
&\le& C \int_0^1\Mx\!_1(\Mx\!_1(g_{t,i}))(c_i(t))\,dt\\
&\overset{by \,\,(\ref{est: mxmx})}{\le}&C\cdot C_2\int_0^1\Mx(g_{t,i}^\pe)^{1/\pe}(c_i(t))\, dt\\
&=&C\cdot C_2\int_0^1 u_{t,i}(c_i(t))\,\dt\to 0.
\end{eqnarray*}

Thus we can find $\tdelta_i\to 0$ 
and a subset $B_{r_i}(c_i(0))''\subset  S_i$
with 
\[
\vol(B_{r_i}(c_i(0)))-\vol(B_{r_i}(c_i(0))'')\le \tdelta_i\min_{q\in B_{r_i}(c_i(0))} \vol(B_{r_i/R}(q))
\]
and 
\[
\int_0^1 \Mx\!_1(g_{t,i})(\phi_{it}(q))\, dt\le \tdelta_i\mbox{ for all $q\in B_{r_i}(c_i(0))''$}.
\]

Recall that $r_i=\tfrac{2R}{\lambda_i}$ with an arbitrary $R$. 
By a diagonal sequence argument it is easy to deduce 
that 
after replacing $\tdelta_i$ by  another sequence converging slowly 
to $0$ we can keep the above estimates for $r_i=\tfrac{R_i}{\lambda_i}$ 
with $R_i\to \infty$ sufficiently slowly. 
Combining this with Lemma~\ref{lem:dtr} shows that 
$f_i=\phi_{i1}\colon (\lambda_iM_i, c_i(0))\rightarrow (\lambda_iM_i,c_i(1))$
has the zooming in property.  

Let  $\tX^t_i$ be a lift of $X^t_i$ to the universal covering {
$\tM_i$ of $M_i$. }Consider the integral curve $\tc_i\colon [0,1]\rightarrow \tM_i$ 
of $\tX^t_i$ with $\tc(0)=\tp_i$. 
Clearly $\tc_i$ is a lift of $c_i$ and by the Covering Lemma~{\eqref{lem: covering}} we have
\[
\int_{0}^1 \Mx(\|\nabla_\cdot \tX^t_i\|^\pe)^{1/\pe}(\tc_i(t))\,dt\le C\int_{0}^1 \Mx(\|\nabla_\cdot X^t_i\|^\pe)^{1/\pe}(c_i(t))\,dt
\]
with some universal constant $C$. 
Thus $\tphi_{i1}\colon (\lambda_i\tM_i,\tp_i)\rightarrow (\lambda_i\tM_i,\phi_{i1}(\tp_i))$ has the zooming in property as well. 
Any other lift  $\tf_i$ of $f_i$ is obtained by composing $\tphi_{i1}$ 
with a deck transformation and thus the result carries over to any lift of $f_i$.
\end{proof}

\begin{prop}[Second main example]\label{prop: second main exa}\label{prop: translations}
Let  $(M_i,g_i)$ be a sequence of $n$-manifolds 
with $\Ric>-1/i$ on $B_{i}(p_i)$ and 
$\overline {B_i(p_i)}$ compact. 
Suppose that $(M_i,p_i)$ converges to  {$( \R^k\times Y,p_\infty)$}.

Then for each $v\in \R^k$ there is  
a sequence of diffeomorphisms $f_i\colon [M_i,p_i]\rightarrow [M_i,p_i]$ 
with the zooming in property 
which converges in the weakly measured sense 
to an isometry $f_\infty$ of $  \R^k\times Y$ that acts trivially on $Y$ 
and by $w\mapsto w+v$ on  
$\R^k$.
Moreover, $f_i$ is isotopic to the identity and 
there is a lift $\tf_i\colon [\tM_i,\tp_i]\rightarrow [\tM_i,\tp_i]$ of $f_i$ to the universal
cover which has the zooming in property as well.
\end{prop}

\begin{proof}
Using the splitting $\R^k= \R v\oplus (v)^\perp$ 
and replacing $Y$ by $Y\times (v)^\perp$ we see 
that it suffices to prove the statement for $k=1$.

By the work of Cheeger and Colding ~\cite{CC} we can find sequences  $\rho_i\to \infty$ 
and $\eps_i\to 0$
and harmonic functions $b_i\co B_{4\rho_i}(p_i)\to \R$ such that

\begin{eqnarray*}\label{e:lip2}
|\nabla b_i|&\le& L(n) \mbox{ for all $i$ and}\\
\fint_{B_{4R}(p_i)} \bigl(\bigl||\nabla b_i|-1\bigr|+ \|\Hess_{b_i}\|\bigr)^2&\le& \eps_i\,\,\,\,\,\,\,\mbox{
for any $R\in [1/4,\rho_i]$.} 
\end{eqnarray*}

Let $X_i$ be a vector field 
with compact support with 
$X_i=\nabla b_i$ on $B_{3\rho_i}(p_i)$, and
let $\phi_{it}$ denote the flow of $X_i$. 
Clearly for any $t$ we can find $r_i\to\infty$ 
such that $\phi_{it|B_{r_i}(p_i)}$ is measure preserving.

Put $\psi_i:= \bigl||\nabla b_i|-1\bigr|+ \|\Hess(b_i)\|$. 
We deduce from Lemma~\ref{lem: weak11} that 
\[
\fint_{B_{2R}(p_i)} \Mx(\psi_i)^2\le C(n,R)\eps_i
\]
and Cauchy--Schwarz gives 

\[
\fint_{B_R(p_i)} \Mx(\psi_i)\le \sqrt{C(n,R)\eps_i}.
\]

Suppose now that $t_0\le \tfrac{R}{4L(n)}$. 
Then $\phi_t(q)\in B_{\frac{3R}{4}}(p_i)$ for all $q\in B_{R/2}(p_i)$ 
and all $t\in [-t_0,t_0]$.
Combining that $\phi_{it}$ is measure preserving and 
$\vol(B_R(p_i))\le C_3(n,R)\vol(B_{R/2}(p_i))$, we get that
\begin{eqnarray*}
\fint_{B_{R/2}(p_i)}\int_{0}^{t_0} \Mx(\psi_i)(\phi_{it}(p))
&\le& C_3(n,R) t_0\fint_{B_{R}(p_i)}\Mx(\psi_i)\\
&\le& C_4(t_0,R,n)\sqrt{\eps_i}.
\end{eqnarray*}

It is now easy to find $R_i\to \infty$ and $\delta_i\to 0$ 
with

\begin{eqnarray*}
\fint_{B_{R}(p_i)}\int_{0}^{t_0} \Mx(\psi_i)(\phi_{it}(p))\,dt\,d\mu_i(p)
\le \delta_i \mbox{ for all $R\in [1,R_i]$.}
\end{eqnarray*}

After an adjustment of the sequences $\delta_i\to 0$ and $R_i\to \infty$
we can find $B_{R_i}(p_i)'\subset B_{R_i}(p_i)$ 
with 
\begin{eqnarray*}
 \vol(B_{R_i}(p_i))-\vol(B_{R_i}(p_i)')&\le&  \delta_i\vol(B_1(q))\,\mbox{ for 
all $q\in B_{R_i}(p_i)$ and}\\
\int_{0}^{t_0} \Mx(\psi_i)(\phi_t(q))&\le& \delta_i\,\,\hspace*{4.6em}\mbox{ 
for all $q\in B_{R_i}(p_i)'$. }
\end{eqnarray*}

Moreover, the  displacement of $\phi_{i t_0}$ is globally bounded by $L(n)t_0$
and thus, Lemma~\ref{lem:dtr} implies that $\phi_{i t_0}$ has the
 zooming in property. 
To be precise the lemma only implies that 
$\phi_{i t_0}\colon \bigl[10 M_i,p_i\bigr]\rightarrow 
\bigl[10 M_i,p_i\bigr]$ has the zooming in property
and Lemma~\ref{lem: iso} then allows to scale back down by a factor 
10 without losing the zooming in property.

 As in the proof of the Product Lemma~{\eqref{prodlem}} we see that $b_i$ converges to the projection $b_{\infty}\colon  \R\times Y\rightarrow \R$. 
In particular $b_\infty$ is a submetry. 

For all $q_i\in B_{R_i}(p_i)'$
the length of the integral curve $c_i(t)=\phi_{it}(q_i)$ ($t\in [0,t_0]$)
is bounded above by $t_0+\delta_i$.
Moreover, $b_i(c_i(t_0))-b_i(c_i(0))$ is bounded below by $t_0-\delta_i$. 
This implies that $\phi_{\infty t_0}$ satisfies 
$d(p,\phi_{\infty t_0}(p))\le t_0$ 
and $b_\infty(\phi_{\infty t_0}(p))-b_\infty(p)\ge t_0$. 
Clearly, equality must hold and $\phi_{\infty t}(q)$ 
is a horizontal geodesic with respect to the submetry 
$b_\infty$. This in turn implies 
that $\phi_{\infty t_0}$ respects the decomposition $Y\times \R$, it  
acts by the identity on the first factor and by translation on the second as claimed.

Let $\sigma\colon \tM_i\rightarrow M_i$ denote the universal cover, 
$\tp_i\in \tM_i$ a lift of $p_i$, and 
let $\tX_i$ be a lift of $X_i$. By Lemma~\ref{lem: cover} 
all estimates for $X_i$ on $B_R(p_i)$ give similar estimates 
for $\tX_i$ on $B_{R}(\tp_i)$. Thus the 
flow $\tphi_{it}$ of $\tX_i$ has the zooming in property as well.
\end{proof}

%%%%%%%%%%%%%%%%%%%%%%%%%%%%%%%%%%%%%%%%%%%%%%%%%%%%%%%%%%%%%%%%%%%%%%%%%%%%%%%%%%%%%%%%%%%
%%%%%%%%%%%%%%%%%%%%%%%%%%%%%%%%%%%%%%%%%%%%%%%%%%%%%%%%%%%%%%%%%%%%%%%%%%%%%%%%%%%%%%%%%%%
%%%%%%%%%%%%%%%%%%%%%%%%%%%%%%%%%%%%%%%%%%%%%%%%%%%%%%%%%%%%%%%%%%%%%%%%%%%%%%%%%%%%%%%%%%%%
%%%%%%%%%%%%%%%%%%%%%%%%%%%%%%%%%%%%%%%%%%%%%%%%%%%%%%%%%%%%%%%%%%%%%%%%%%%%%%%%%%%%%%%%%%%%

\section{A rough idea of the proof of the Margulis Lemma.}\label{sec: scetch}\label{sec: rough}

The proof of the Margulis Lemma is somewhat indirect 
and it might not be easy for everyone to grasp 
immediately how things interplay. 
In this section we want to give a rough idea of 
how the arguments would unwrap in a simple but nontrivial case.

Let $p_i\to \infty $ be a sequence of odd primes 
and let $\Gamma_i:=\Z\ltimes \Z_{p_i}$ be the semidirect 
product where the homomorphism $\Z\rightarrow \Aut(\Z_{p_i})$
maps $1\in \Z$ to the automorphism 
$\varphi_i$ given by $\varphi_i(z+p_i\Z)=2z+p_i\Z$.

Suppose, contrary to the Margulis Lemma, we have  a sequence of compact $n$-manifolds  $M_i$ with $\Ric>-1/i$ 
and $\diam(M_i)=1$ and fundamental group $\Gamma_i$. 

A typical problem situation would be that $M_i$ converges
to a circle and { its universal cover  $\tM_i$} converges to $\R$.

We then replace $M_i$ by $B_i=\tM_i/\Z_{p_i}$
and in order  not to lose information we endow $B_i$ with the deck transformation
$f_i\colon B_i\rightarrow B_i$ representing a generator of $\Gamma_i/\Z_{p_i}$. 
Then $B_i$ will converge to $\R$ as well.

It is part of the rescaling theorem that one can find $\lambda_i\to \infty$ 
such that the rescaled sequence $\lambda_i B_i$ converges to $\R\times K$ with $K$ being compact but not equal to  a
point. Suppose for illustration 
that $\lambda_iB_i$ converges to $\R\times S^1$ 
and $\lambda_i\tM_i$ converges to $\R^2$ and that the action of $\Z_{p_i}$ 
converges to a discrete action of $\Z$ on $\R^2$.

The maps $f_i\colon \lambda_i B_i\rightarrow \lambda_i B_i$ do not converge, 
because typically  $f_i$ would map a base point $x_i$ to some point $y_i=f_i(x_i)$ 
with $d(x_i,y_i)=\lambda_i\to \infty$ with respect to the rescaled distance.

But the second statement in the rescaling theorem
guarantees that we can find a sequence 
of diffeomorphisms $g_i\colon [\lambda_i B_i,y_i]\rightarrow [\lambda_i B_i,x_i]$ 
with the zooming in property. 
The composition $f_{new, i}:=g_i\circ f_i\colon [\lambda_i B_i,x_i]\rightarrow [\lambda_i B_i,x_i]$ 
also has the zooming in property and thus subconverges to 
an isometry of the limit.

Moreover, 
a lift  $\tf_{new, i}\colon \lambda_i\tM_i\rightarrow \lambda_i\tM_i$ of $f_{new, i}$ 
has the zooming in property, too. 
Since $g_i$ can be chosen isotopic 
to the identity, the action of $\tf_{new, i}$ on the deck transformation 
group $\Z_{p_i}=\pi_1(B_i)$ by conjugation remains unchanged.  

On the other hand, the $\Z_{p_i}$-action on $\tM_i$
converges to a discrete $\Z$-action on $\R^2$ 
and  $\tf_{new, i}$ converges to an isometry $\tf_{new,\infty}$ 
of $\R^2$ normalizing 
the $\Z$-action. 
This implies that  $\tf_{new,\infty}^2$ commutes with the $\Z$-action and 
it is then easy to get a contradiction.

\begin{problem} We suspect that for any given $n$ and $D$ 
only finitely many of the groups in the
above family of groups  should occur
as fundamental groups of compact $n$-manifolds with $\Ric>-(n-1)$ 
and diameter $\le D$. { Note that this is not even known  under the stronger assumption of 
$K>-1$ and $\diam\le D$.}
\end{problem}

%%%%%%%%%%%%%%%%%%%%%%%%%%%%%%%%%%%%%%%%%%%%%%%%%%%%%%%%%%%%%%%%%%%%%%%%%%%%%%%%%%%%%%%%%%%
%%%%%%%%%%%%%%%%%%%%%%%%%%%%%%%%%%%%%%%%%%%%%%%%%%%%%%%%%%%%%%%%%%%%%%%%%%%%%%%%%%%%%%%%%%%%
%%%%%%%%%%%%%%%%%%%%%%%%%%%%%%%%%%%%%%%%%%%%%%%%%%%%%%%%%%%%%%%%%%%%%%%%%%%%%%%%%%%%%%%%%%%%
%%%%%%%%%%%%%%%%%%%%%%%%%%%%%%%%%%%%%%%%%%%%%%%%%%%%%%%%%%%%%%%%%%%%%%%%%%%%%%%%%%%%%%%%%%%%

\section{The Rescaling Theorem.}\label{sec: rescaling}

\begin{thm}[Rescaling Theorem]\label{thm: rescale}
Let $(M_i,g_i,p_i)$ be a sequence of $n$-manifolds satisfying
$\Ric_{B_{r_i}(p_i)}>-\mu_i$ and $\overline{B_{r_i}(p_i)}$ is compact 
for some  $r_i\to\infty$, $\mu_i\to 0$. 
Suppose that $(M_i,g_i,p_i)\conv (\R^k,0)$ for some $k<n$. 
Then after passing to a subsequence
we can find a compact metric space $K$ with $\diam(K)=10^{-n^2}$,
a sequence of subsets $G_1(p_i)\subset B_1(p_i)$ 
with $\tfrac{\vol(G_1(p_i))}{\vol(B_1(p_i))}\to 1$
and a sequence $\lambda_i\to \infty$ such that the following holds
\begin{enumerate}
\item[$\bullet$]  For all $q_i\in G_1(p_i)$ the isometry type of the limit
of any convergent subsequence of $(\lambda_i M_i,q_i)$ is given 
by the metric product $\R^k\times K$.
\item[$\bullet$] For all $a_i,b_i\in G_1(p_i)$
we can find a sequence of diffeomorphisms
\[
f_i\colon [\lambda_iM_i,a_i]\rightarrow [\lambda_iM_i,b_i]
\]
with the zooming in property such that $f_i$ is
isotopic to the identity. 
Moreover, for any lift $\ta_i,\tb_i\in \tM_i$ of $a_i$ and $b_i$ 
to the universal cover $\tM_i$
we can find a lift $\tf_i$ of $f_i$ such that 
\[
\tf_i\colon [\lambda_i \tM_i,\ta_i]\rightarrow [\lambda_i \tM_i,\tb_i] 
\]
has the zooming in property as well.
\end{enumerate}
Finally, if $\pi_1(M_i,p_i)$ is generated by loops of length 
$\le R$ for all $i$, then we can find $i_0$ 
such that 
$\pi_1(M_i,q_i)$ is generated by loops of length $\le \tfrac{1}{\lambda_i}$ 
for all $q_i\in G_1(p_i)$ and $i\ge i_0$. 
\end{thm}

\begin{proof}
By Theorem~\ref{hess-ineq},
after passing to a subsequence 
we can find a harmonic map $b^i\co B_{i}(p_i)\to B_{i+1}(0)\subset\R^k$ 
that gives a $\frac{1}{2^i}$ {Gromov--Hausdorff approximation from }$B_i(p_i)$ to $B_i(0)\subset\R^k$.
Put
\begin{equation*}
h_i= \sum_{j,l=1}^k |<\nabla b^i_j, \nabla
b^i_l>-\delta_{j,l}|+\sum_j \|\Hess_{b^i_j}\|^2.
\end{equation*}
After passing to a subsequence we also have by Theorem~\ref{hess-ineq} that
\[
\fint_{B_R(p_i)}h_i\le \eps_i^4\, \mbox{ for all $R\in [1,i]$ with $\eps_i\to 0$.}
\]
Furthermore, we may assume $\Ric>-\eps_i$ on $B_i(p_i)$. Put
\begin{equation}\label{eq:good}
G_4(p_i):=\bigl\{x\in B_4(p_i)\mid \Mx\!_4 h_i(x)\le \eps_i^2 \bigr\}.
\end{equation}
We will call elements of $G_4(p_i)$ {\em "good points" } in $B_4(p_i)$.
We put $G_r(q)=B_r(q)\cap G_4(p_i)$ for $q\in B_2(p_i)$ and $r\le 2$. 
It is immediate from the weak 1-1 inequality (Lemma~\ref{lem:mf-est}) that
\[
\tfrac{\vol(G_1(p_i))}{\vol(B_1(p_i))}\to 1.
\]

By  the Product Lemma~{\eqref{prodlem}}, for any $q_i\in G_4(p_i)$ 
and any $h\ge 1$  we have that 
$hB_{1/h}(q_i)$ is $\teps_i$ close  to a ball in a product $ \R^k\times Y$
where $\teps_i\to 0$  can be chosen independently of $h$ and $q_i$.
Note that the space $Y$ may depend on $h$ and $q_i$.

For any point $q_i\in G_1(p_i)$ let $\rho_i(q_i)$ be the largest number $\rho< 1$ such that
the map $b^i\co B_{\rho}(q_i)\to \R^k$ has distance distortion 
exactly equal to $\rho\cdot 10^{-n^2}$. 
For any choice $q_i\in G_1(p_i)$ we know that $\rho_i(q_i)\to 0$  
and 
by the Product Lemma~{\eqref{prodlem}}
$ \bigl(\tfrac{1}{\rho_i(q_i)}M_i,q_i\bigr)$ subconverges to  $\R^k\times K$ where $K$ is compact  with $\diam(K)=10^{-n^2}$.

Again note, that a priori, the constants $\rho_i(q_i)$ and the  space $K$ depend on the choice of $q_i\in G_1(p_i)$.
We will show that, in fact,  both are essentially  independent of the choice of good points in $B_1(p_i)$.

For each $i$ we define the number $\rho_i$ as the supremum  of $\rho_i(q_i)$ with $q_i\in G_1(p_i)$ 
and put 
\[\lambda_i=\tfrac{1}{\rho_i}.\]

The statement of the theorem will follow easily from the following sublemma.

\begin{sublem}\label{lem:subl} 
There exist  constants $\bar C_1,\bar C_2$ independent of $i$  such that the following holds.
For any good point $p\in G_{1}(p_i)$ and  any $r\le 2$ there exists $B_r(p)'\subset B_r(p)$ such that
\begin{equation}\label{eq:c1}
\vol B_r(p)'\ge (1-\bar C_1r\eps_i)\vol B_r(p)
\end{equation}
and such that for any $q\in B_r(p)'$ there exists $p'\in B_{L\rho_i}(p)$ 
(with $L=9^n$) 
and a time dependent (piecewise constant in time) divergence free vector 
field $X^t$ (dependence on $q$  is suppressed)
and its integral curve $c_q(t)\co [0,1]\to B_2(p_i)$ such that $c(0)=p',c(1)=q$ and 
\[
\tfrac{1}{\vol(B_r(p))}\int_{B_r(p)'}\int_0^1\bigl(\Mx (\|\nabla_\cdot X^t\|^{3/2})\bigr)^{2/3}(c_q(t))\,dt\,d\mu_i(q)\le \bar C_2r\eps_i.
\]
\end{sublem}
\begin{proof}
Despite the fact that $L$ can be chosen 
as $L=9^n$ we treat it for now 
as a large constant $L\ge 10$ which needs to be determined.  
Throughout the proof we will denote by $C_j$ various constants depending on $n$.
Sometimes these constants will also depend on $L$ in which case that dependence will always be explicitly stated.

Although  we consider a fixed  $i$ for the major  part of the proof,
we take the liberty of assuming that $i$ is large.
This way  we can ensure that for any $r\ge \rho_i$ and any good point $p\in G_1(p_i)$ 
the ball  $\frac{1}{Lr}B_{2Lr}(p)$ is by the Product Lemma in the measured 
Gromov--Hausdorff sense arbitrarily close to a ball 
in a product $\R^k\times K$ where by our choice of $\rho_i$, $K$ has diameter $\le 10^{-n^2}$. 
This implies, for example, that we can assume that
\begin{equation}\label{eq: volume}
\tfrac{1}{2}\bigl(\tfrac{r_1}{r_2}\bigr)^k \le \tfrac{\vol(B_{r_1}(q_1))}{\vol(B_{r_2}(q_2))}\le 2\bigl(\tfrac{r_1}{r_2}\bigr)^k\,\mbox{ for all $q_1,q_2\in B_{Lr}(p),\, 
r_1,r_2\in [r/20,2Lr]$.}
\end{equation}

The sublemma is trivially true  for $r\le L\rho_i$ with $X\equiv 0$. 
By induction it suffices to show that if the sublemma holds for some 
$r\in (\rho_i,\tfrac{2}{L}]$ then it holds for $Lr$.

Let $q_1,\ldots , q_l$ be a maximal $r/2$-separated net in $B_{Lr}(p)$. Note that 
\[
B_{Lr}(p)\subset \bigcup_{m=1}^lB_{r/2}(q_m).
\]
By a standard volume comparison argument 
we can deduce from \eqref{eq: volume} that 
\begin{equation}\label{eq:relvol}
\frac{\sum_{m=1}^l\vol B_{r/2}(q_m)}{\vol B_{Lr}(p)}\le 3^{k+1}\le 3^n.
\end{equation}
Fix $q_m$ and consider the following vector field
\[
X(x)=\sum_{\alpha=1}^k(b_\alpha^i(p)-b_\alpha^i(q_m)) \nabla b_\alpha^i (x).
\]

Since $b_\a^i$ are  Lipschitz with a universal Lipschitz constant  by (\ref{e:lip}), $X$ satisfies 
\begin{equation}\label{eq:linear}
|X(x)|\le C(n)\cdot Lr.
\end{equation}
Also, by construction, $X$ satisfies
\begin{equation}\label{eq:avX}
\Mx\!_4\|\nabla_\cdot X\|^{2}(p)\le C_4\eps_i^2L^2r^2.
\end{equation}
Note that we get an extra $L^2r^2$ factor as compared to (\ref{eq:good}) because $|b^i(p)-b^i(q_m)|\le C(n)Lr$.
By applying (\ref{eq:mmf1}) we get 

\[
\Mx\Bigl([\Mx (\|\nabla_\cdot X\|^{3/2})]^{4/3}\Bigr)(p)
\le C(n)\Mx \!_4 (\|\nabla_\cdot X\|^2)(p)\le (C_5L\eps_ir)^2.
\]
In particular, for $R_1=2C(n)Lr$

\[
\fint_{B_{R_1}(p)}[\Mx (\|\nabla_\cdot X\|^{3/2})]^{4/3}\le (C_5L\eps_ir)^2
\]
provided $R_1\le 1$. However,
in the case of $R_1\in  [1,4C(n)]$ the same follows 
more directly from our initial assumptions combined 
with Lemma~\ref{lem: weak11} for all large $i$.
By the Cauchy--Schwarz  inequality the last estimate gives

\begin{equation}\label{eq:avX-d}
\fint_{B_{R_1}(p)}\bigl(\Mx (\|\nabla_\cdot X\|^{3/2})\bigr)^{2/3}\le C_5L\eps_ir.
\end{equation}

Consider the measure preserving flow $\phi$ of $X$ on $[0,1]$. 
Because of \eqref{eq:linear}
the flow lines $\phi_t(q)$ with $q\in B_{r}(p)$ 
stay in the ball $ B_{R_1}(p)$ for $t\in [0,1]$. 
We choose a universal constant $C_6(L)$ independent of $i$ with 
$C_5L\cdot \tfrac{\vol(B_{2C(n)Lr}(p))}{\vol(B_{r/10}(q_m))}\le C_6(L)$.

Combining that the flow is measure preserving, inequality~\eqref{eq:avX-d}
and our choice of $C_6$ we see that
\begin{eqnarray}\label{eq:avX-flow}
\hspace*{2.5em}\tfrac{1}{\vol(B_{r/10}(q_m))}\int_{B_{r}(q_m)}\int_0^1\bigl(\Mx(\|\nabla_\cdot X\|^{3/2})\bigr)^{2/3}(\phi_t(x))\,dt\,d\mu_i(x)\!\!&\le&
C_6(L)r\eps_i.
\end{eqnarray}

By the Product Lemma~{\eqref{prodlem}} applied to the rescaled balls $\frac{1}{Lr}B_{Lr}(p)$ we know that $\frac{1}{Lr}B_{Lr}(p)$ is measured Gromov--Hausdorff close to a unit ball in 
$\R^k\times K_3$ where $\diam K_3\le 10^{-n^2}$.
Moreover, $\phi_1$  is measured close to a translation 
by $b^i(p)-b^i(q_m)$ in $\R^k$, see proof
of Proposition~\ref{prop: second main exa}. Since we only need to argue for large $i$, 
we may assume that by volume  $3/4$ of the points in $B_{r/10}(q_m)$ 
are mapped by $\phi_1$ to points in $B_{ r/9}(p)$. 

We choose such a point $q\in B_{r/10}(q_m)$. In view of \eqref{eq:avX-flow} 
we may assume in addition that
\[
\int_0^1\Mx(\|\nabla_\cdot X\|^{3/2})^{2/3}(\phi_t(q))\,dt\le \tfrac{4}{3}C_6(L)r\eps_i.
\]

By Lemma~\ref{lem:dtr} this implies
\begin{equation}\label{dtr-est}
\fint_{B_{r}(q)\times B_{r}(q)}\dt_r(1)(x,y)\,d\mu_i(x)\, d\mu_i(y)\le C_7(L) r\cdot \eps_i.
\end{equation}
Note  that by definition, $\dt_r(1)(x,y)\ge \min\{r, |d(\phi_1(x),\phi_1(y))-d(x,y)|\}$.
Combining this inequality with  the knowledge 
that $3/4$ of the points in $B_{r/10}(q_m)$ end up in 
$B_{r/9}(p)$,  implies that 
we can find a subset $B_{r/2}(q_m)'\subset B_{r/2}(q_m)$ with 
\begin{eqnarray} \label{eq:c8}
\vol(B_{r/2}(q_m)')&\ge& (1-C_8(L)\eps_i r)\vol(B_{r/2}(q_m))\hspace*{0.5em}\mbox{ and }\\
\phi_1(B_{r/2}(q_m)')&\subset& B_r(p).
\end{eqnarray}

Set  $B_{r/2}(q_m)'' :=B_{r/2}(q_m)'\cap \phi_1^{-1}(B_r(p)')$. Then we get

\[
\phi_1(B_{r/2}(q_m)'')\subset B_r(p)'\,\,\,\,\mbox{ and}
\]

\begin{eqnarray}\nonumber
\vol (B_{r/2}(q_m)'')&\ge& \vol (B_{r/2}(q_m)')-(\vol B_r(p)-\vol B_r(p)')\\\nonumber
&\overset{\text{\hspace*{-1.8em}by \eqref{eq:c8} and \eqref{eq:c1}\hspace*{-1.3em}}}{\ge}& (1-C_8(L)r\eps_i)\vol B_{r/2}(q_m) -\bar C_1r\eps_i\vol B_r(p)\\\nonumber
&\overset{\text{by ~\eqref{eq: volume} }}{\ge}& (1-C_8(L)r\eps_i)\vol B_{r/2}(q_m) -2^{k+1}\bar C_1r\eps_i\vol B_{r/2}(q_m)\\ \label{eq:r/2}
&\ge& (1-(C_8(L)+2^{n}\bar C_1)r\eps_i)\vol B_{r/2}(q_m),
\end{eqnarray}
where we used that $\phi_1$ is volume-preserving in the first inequality.

Using the  induction assumption, we can prolong each integral curve from any  $q\in B_{r/2}(q_m)''$ to a point $x\in B_r(p)'$ by extending it by a previously constructed integral curve from $x$ to $p'\in B_{L\rho_i}(p)$ of a vector field $X^t_{old}$ (which depends on $p'$).
We set  $X^t_{new}=2X_{old}^{2t}$
for $0\le t\le 1/2$ and $X^t_{new}=-2X$ for $1/2\le t\le 1$.
Let $c_q(t)$ be the integral curve of $X^t_{new}$ with $c_q(1)=q$ 
and $c_q(0)=p'$.

By the induction assumption and (\ref{eq:avX-flow}), we get
\begin{equation}\label{eq:qm}
\begin{aligned}
&\int _{B_{r/2}(q_m)''}\int_0^1(\Mx \|\nabla_\cdot X_{new}^t\|^{3/2})^{2/3}(c_q(t))\,dt\,d\mu_i(q)\le \\
&\hspace*{7em}
\le C_6(L)(r\eps_i)\vol B_{r/2}(q_m)+\bC_2(r\eps_i)\vol B_r(p)
\\&\hspace*{6.7em}\stackrel{ \eqref{eq: volume}}{\le}  C_6(L)(r\eps_i)\vol  B_{r/2}(q_m)+2^{n}\bC_2(r\eps_i)\vol B_{r/2}(q_m)
\\&\hspace*{7em}= \bigl(C_9(L)+2^{n}\bC_2\bigr)(r\eps_i)\vol B_{r/2}(q_m).
\end{aligned}
\end{equation}

Recall that the  balls $B_{r/2}(q_m)$ cover $B_{Lr}(p)$. We put 
\[
B_{Lr}(p)':=B_{Lr}(p)\cap\bigcup_m B_{r/2}(q_m)''.
\]

By construction, for every point in $B_{Lr}(p)'$ there exists a vector field $X^t$ 
whose integral curve connects it to a point in $B_{L\rho_i}(p)$ 
satisfying (\ref{eq:qm}). For points covered by several sets $B_{r/2}(q_m)''$
we  pick any one.
Then we have
\begin{eqnarray*}
\int_{B_{Lr}(p)'}\int_0^1\Bigl(\Mx \bigl(\|\nabla_\cdot X_{new}^t\|^{\frac{3}{2}}\bigr)\Bigr)^{\frac{2}{3}}(c_q(t))\,dt\,d\mu_i(q)\,\le\hspace*{-10em}&&\\
&\le& \sum_{m=1}^l (C_9(L)+2^{n}\bar C_2)(r\eps_i)\vol B_{r/2}(q_m)\\
&\stackrel{ \eqref{eq:relvol}}{\le}& 3^n (C_9(L)+2^{n}\bar C_2)(r\eps_i)\vol B_{rL}(p)\\
&\le& \bar C_2\cdot Lr\eps_i\cdot\vol B_{rL}(p).
\end{eqnarray*}
The last inequality holds if
$L=9^n\ge 2\cdot 2^{n}\cdot 3^{n}$ and 
$ \bar C_2\ge 3^n\cdot C_9(L)$.

Recall that by \eqref{eq:r/2}
\[
\vol B_{r/2}(q_m)-\vol B_{r/2}(q_m)''\ge (C_8(L)+2^{n}\bar C_1)r\eps_i\vol B_{r/2}(q_m)
\]
for any $m$ and thus 
\begin{eqnarray*}
\vol B_{Lr}(p)-\vol B_{Lr}(p)'&\le& \sum_{m=1}^l (C_8(L)+2^{n}\bar C_1)r\eps_i\vol B_{r/2}(q_m)\\[-1ex]
&\stackrel{ \eqref{eq:relvol}}{\le}& 
3^n(C_8(L)+2^{n}\bar C_1)r\eps_i\vol(B_{Lr}(p))\\
&\le&  \bar C_1Lr\eps_i\vol B_{Lr}(p)
\end{eqnarray*}
provided that $L=9^n\ge 2 \cdot 6^n$ and $\bar C_1\ge 3^n C_8(L)$.
This finishes the proof of Sublemma~\ref{lem:subl}.
\end{proof}

Observe that for any two good points $x_i,y_i\in G_1(p_i)$ 
we can deduce from  Sublemma \ref{lem:subl} that 
$\vol(B_2(x_i)'\cap B_2(y_i)')\ge \vol(B_{1/2}(p_i))$ 
for all large $i$. 

Then it follows, also by Sublemma \ref{lem:subl},  that there is $q\in B_2(x_i)'\cap B_2(y_i)'$, 
$x_i'\in B_{L\rho_i}(x_i)$ and $y_i'\in B_{L\rho_i}(y_i)$ 
such that for the integral curve $c$ connecting 
$x_i'$ with $q$ on the first half of the interval 
and $q$ with $y_i'$ on the second half we 
have for the corresponding time dependent vector field $X^t_i$ 
\[
\int_0^1\bigl(\Mx(\|\nabla_\cdot X^t_i\|^{3/2})\bigr)^{2/3}(c(t))\,dt\le \bC_3 \eps_i 
\]
with a constant $\bC_3$ independent of $i$. 

Let $f_i:=\phi_{i1}$ be the flow of $X^t_i$ 
evaluated at time $1$.
Since $\lambda_i=\tfrac{1}{\rho_i}\to \infty$, we can employ  Proposition~\ref{prop: first main exa}
to see that
\[
f_i\colon [\lambda_i M_i, x_i]\rightarrow [\lambda_i M_i,y_i]
\]
has the zooming in property.
Thus, the Gromov--Hausdorff limits of the two sequences are isometric. 

For any lifts $\tx_i$ and $\ty_i$
of $x_i$ and $y_i$ to the universal 
cover we can find lifts  $ \tx_i'$ and $\ty_i'$ of $ x_i'$ and $y_i'$ 
in the $L\rho_i$ neighborhoods of $\tx_i$ and $\ty_i$.
Let
$\tf_i \colon \tM_i\rightarrow \tM_i$ be the lift $f_i$ with $\tf_i(\tx_i')=\ty_i'$. 
Proposition~\ref{prop: first main exa} ensures that 
$\tf_i\colon [\lambda_i \tM_i,\tx_i]\rightarrow [\lambda_i\tM_i,\ty_i]$ 
has the zooming in property as well. 

It remains to check the last part  of the Rescaling Theorem concerning the fundamental group.
Assume  that $\pi_1(M_i,p_i)$ is generated by loops of length $\le R$. 
For every good point $q\in B_1(p_i)$ let $r_i(q)$ 
denote the minimal number such that $\pi_1(M_i,q)$ can be generated by loops 
of length $\le r_i(q)$.
Let $r_i$ denote the supremum of $r_i(q)$ over all good points $q$
and choose a good point $q_i$ with $r_i(q_i)\ge \tfrac{1}{2}r_i$. 
It suffices to check that $\limsup_{i\to \infty}\lambda_ir_i< 1$. 
Suppose we can find a subsequence with 
$\lambda_ir_i(q_i)\ge 1/4$ for all $i$. 
By the Product Lemma~{\eqref{prodlem}}, the sequence $(\tfrac{1}{r_i}M_i,q_i)$ subconverges to { $(\R^k\times K', q_\infty)$ }
with $\diam(K')\le 4\cdot 10^{-n^2}$. 
Combining Lemma~\ref{lem: prep gap} with Lemma~\ref{lem:short}
we see that $\pi_1(\tfrac{1}{r_i}M_i,q_i)$ 
can be generated by loops of length $\le 2\diam(K')+c_i$ 
with $c_i\to 0$ -- a contradiction.
\end{proof}

%%%%%%%%%%%%%%%%%%%%%%%%%%%%%%%%%%%%%%%%%%%%%%%%%%%%%%%%%%%%%%%%%%%%%%%%%%%%%%%%%%%%%%%%%%%%%%%
%%%%%%%%%%%%%%%%%%%%%%%%%%%%%%%%%%%%%%%%%%%%%%%%%%%%%%%%%%%%%%%%%%%%%%%%%%%%%%%%%%%%%%%%%%%%%%%
%%%%%%%%%%%%%%%%%%%%%%%%%%%%%%%%%%%%%%%%%%%%%%%%%%%%%%%%%%%%%%%%%%%%%%%%%%%%%%%%%%%%%%%%%%%%%%%
%%%%%%%%%%%%%%%%%%%%%%%%%%%%%%%%%%%%%%%%%%%%%%%%%%%%%%%%%%%%%%%%%%%%%%%%%%%%%%%%%%%%%%%%%%%%%%%%

\section{The Induction Theorem for C-Nilpotency}\label{sec:C-nilp}\label{sec: induction}

$C$-nilpotency of fundamental groups of manifolds with almost nonnegative Ricci 
curvature (Corollary~\ref{intro: almost nonneg}) will follow from the following technical result.

\begin{thm}[Induction Theorem]\label{thm:pi1}\label{thm: induction}
Suppose $(M^n_i,p_i)$ is a sequence of pointed n-dimensional Riemannian manifolds (not necessarily complete) satisfying 
\begin{enumerate}
\item $\Ric_{M_i}\ge -1/i$; 
\item $\overline{B_i(p_i)}$ is compact for any $i$;
\item There is some $R>0$ such that $\pi_1(B_R(p_i))\to\pi_1(M_i)$ is surjective for all $i$;
\item $(M_i,p_i)\conv (\R^k\times K, (0,p_\infty))$ where $K$ is compact.
\end{enumerate}

Suppose in addition that we have
$k$ sequences $f_i^{j}\co [\tM_i,\tp_i]\to [\tM_i,\tp_i]$ 
of diffeomorphisms of the universal covers $\tM_i$ of $M_i$ 
which have the zooming in property
and which 
normalize the deck transformation group acting on $\tM_i$, $j=1,\ldots k$.
Here $\tp_i$ is a lift of $p_i$.

Then there exists a positive integer $C$ such that 
for all sufficiently large $i$, $\pi_1(M_i)$  
contains a nilpotent  subgroup $\gN\lhd \pi_1(M_i)$ of index 
at most $C$ 
such that $\gN$ has an $(f_i^j)^{C!}$-invariant ($j=1,\ldots,k$) cyclic 
nilpotent chain of length $\le {n-k}$, that is: 

We can find 
$\{e\}=\gN_0\lhd\,\cdots\,\lhd\gN_{n-k}=\gN $ with
$[\gN,\gN_h]\subset \gN_{h-1}$ 
and cyclic
factor groups $\gN_{h+1}/\gN_h$. 
Furthermore, each  $\gN_h$ is invariant under the 
action  of $(f_i^j)^{C!}$ by conjugation
and the induced automorphism of 
$\gN_h/\gN_{h+1}$ is the identity. 

\end{thm}

Although we will have $f_i^j=\id$ in the applications, it is crucial for the proof 
of the theorem by induction to establish it in the stated generality.

\begin{proof}
The proof proceeds by reverse induction on $k$. For $k=n$ the result  follows immediately from Theorem~\ref{noncol}. We assume 
that the statement holds for all $k'>k$ and we plan to prove it for $k$.

We argue by contradiction. After passing to a subsequence we can
assume that 
any subgroup 
$\gN\subset \pi_1(M_i)$ of index $\le i$ does not have a
nilpotent cyclic chain of length $\le n-k$ 
which is invariant under $(f_i^j)^{i!}$ ($j=1,\ldots,k$).

After passing to a subsequence, $(\tM_i, \tp_i)$ converges 
to $(\R^{\bar k}\times \tK, (0,\tp_\infty))$, where $\tK$ contains no lines. 
Moreover, we can assume that the action of $\pi_1(M_i)$ converges 
with respect to the pointed equivariant Gromov--Hausdorff topology 
to an isometric action of a closed subgroup $\lG\subset \Iso(\R^{\bar k}\times \tK)$.

It is clear that the metric quotient $(\R^{\bar k}\times \tK)/\lG$ 
is isometric to $\R^k\times K$. 
Using that lines in $\R^k\times K$ can be lifted to lines in $\R^{\bar k}\times \tK$,
we deduce that $\R^{\bar k}=\R^k\times \R^l$ and the action of $\lG$ 
on the first Euclidean factor is trivial,
(cf. proof of  Lemma~\ref{lem:short}). 
Since the action of $\lG$ on $\R^l\times \tK$ and hence on $\tK$ is cocompact,
we can use an observation of Cheeger and Gromoll ~\cite{CG2} to deduce that $\tK$ is 
compact for there are no lines in $\tK$.

By passing once more to a subsequence, we can assume that $f_i^{j}$ 
converges in the weakly measured sense to an isometry $f_\infty^j$ 
on the limit space, $j=1,\ldots ,k$, see Lemma~\ref{lem: iso}. 
The overall most difficult step is essentially a consequence of the Rescaling Theorem:

\begin{step} 
Without loss of generality we can assume that $K$ is not a point.
\end{step}

We assume $K=pt$. The strategy is to find a new contradicting sequence converging to $\R^{k}\times K'$ where
$K'\neq pt$.

After rescaling  
every manifold down by a fixed factor we may assume that the limit 
isometry $f_\infty^j$ displaces $(0,p_\infty)$ by less than $1/100$, $j=1,\ldots,k$.
We choose the set of "good" points $G_1(p_i)$  as in Rescaling Theorem~\ref{thm: rescale}.
We let $G_1(\tp_i)$ denote those points in $ B_1(\tp_i)$
projecting to $G_1(p_i)$.
By Lemma~\ref{lem: covering}, $\tfrac{\vol(G_1(\tp_i))}{ \vol(B_1(\tp_i))}\to 1$.

By Lemma~\ref{lem: zoom}, we can remove small subsets $H_i\subset G_1(p_i)$ (and the corresponding subsets from $G_1(\tp_i)$)  such that $\frac{\vol H_i}{\vol G_1(p_i)}\le \delta_i\to 0$
and for any choice of points $\tq_i\in G_1(\tp_i)\backslash H_i$ 
the sequence $f_i^j$ and $(f_i^j)^{-1}$ ($i\in \fN$) is good on all scales 
at $\tq_i$,   $j=1,\ldots,k$. To simplify notations we will assume that it is already true for all choices of  $\tq_i\in G_1(\tp_i).$

For all large $i$ we
can choose a point $\tq_i\in G_{1/2}(\tp_i)$ with 
$f_i^j(\tq_i)\in G_1(\tp_i)$, $j=1,\ldots,k$.
Let $q_i$ be the image of $\tq_i$ in $M_i$ 
and $\bar{f}_i^j\colon M_i\rightarrow M_i$ be
the induced diffeomorphism.

By the Rescaling Theorem \eqref{thm: rescale},
there is a sequence $\lambda_i\to \infty$
and a sequence of diffeomorphisms $g_i^j$ of $M_i$ which are isotopic to identity 
such that
\[
g_i^j\colon [\lambda_i M_i, \bar{f}_i^j(q_i)]\rightarrow   [\lambda_i M_i, q_i]
\]
has the zooming in property
and we can find  lifts $\tg_i^j$ of $g_i^j$ 
such that 
\[
\tg_i^j\colon [\lambda_i \tM_i, f_i^j(\tq_i)]\rightarrow   [\lambda_i \tM_i, \tq_i]
\]
has the zooming in property.
Using Lemma~\ref{lem: composition} we see that

\[
f_{i,new}^j:=f_i^j\circ \tg_i^j\colon [\lambda_i \tM_i, \tq_i]\rightarrow [\lambda_i \tM_i, \tq_i]
\]
has the zooming in property as well for $j=1,\ldots,k$.
Since $g_i^j$ is isotopic to the identity, it follows that conjugation by
$\tg_i^j$ induces an inner automorphism of $\pi_1(M_i)$. 
Therefore $f_{i,new}^j$ produces the same element
$\alpha_i^j\in \Out(\pi_1(M_i))$ as $f_i^j$.

The Rescaling Theorem also ensures that 
$\pi_1(\lambda_iM_i,q_i)$ remains boundedly  generated. 
Finally, it states that $(\lambda_iM_i,q_i)$ converges
to $(\R^k\times K ,(0,q_\infty))$ with $K$ being a compact space with $\diam(K)=10^{-n^2}$.

Thus, $(\lambda_iM_i, q_i)$ and the maps $f_{i,new}^j$ on  the universal covers  give a new contradicting sequence with the limit satisfying $K\ne pt$.

From now on we will assume that it is true for the original contradicting sequence.

\begin{step} Without loss of generality we can assume that 
$f_i^j$ converges in the weakly measured sense to the identity map
of the limit space $\R^k\times \R^l\times \tK$, $j=1,\ldots,k$.
\end{step}

We prove this by finite induction on $j$. 
Suppose we already found a contradicting sequence 
where $f_i^{1},\ldots,f_i^{j-1}$ converge to the identity. 
We have to construct one where in addition 
$f_i:=f_i^j$ converges to the identity.

We first consider the induced
diffeomorphisms $\bar f_i\colon M_i\rightarrow M_i$ 
which converge to an isometry ${\bar f}_\infty$
of $\R^k\times K$ in the weakly measured sense. 
Thus, there is $A\in \Or(k)$ and $v\in \R^k$
such that the induced isometry of the Euclidean factor
$\pr({\bar f}_\infty)\colon \R^k\to \R^k$ is given by $(w\mapsto Aw+ v)$. 

We claim that we can assume without loss of generality that $v=0$.
In fact, by Proposition~\ref{prop: translations},
we can find a sequence of diffeomorphisms
$g_i\colon M_i\to M_i$ with the zooming in 
property such that $g_i$ converges 
to an isometry $g_\infty$ which induces translation by $-v$
on the Euclidean factor. 
Moreover, there is a lift $\tg_i\colon \tM_i\rightarrow \tM_i$
which also has the zooming in property 
if we endow $\tM_i$ with the base point $\tp_i$. 
By Lemma~\ref{lem: composition}, we are free to replace $f_i$ 
by $\tg_i\circ f_i$ and hence $v=0$ without loss of generality.

Since  $\bar f_{\infty}$ fixes the origin of the Euclidean factor, 
we obtain that for the limit $f_\infty$ of $f_i$ 
the following property holds.
Given any $m>0$ we can find 
$g_m\in \lG$   
such that $d( f_{\infty}^m (g_m (0, \tp_\infty)), (0,\tp_{\infty}))\le \diam(K)$ 
in $\R^k\times \R^l\times \tK$. 
It is now an easy exercise to find a  sequence  $\nu_b$ of natural numbers and a sequence of $g_b\in \lG$ 
for which $f_\infty^{\nu_b}\circ g_b$ converges to the identity:

We consider a finite $\eps$-dense
set  $\{a_1,\ldots,a_N\}$ in the ball of radius $\tfrac{1}{\eps}$ around 
$(0,\tp_\infty)\in \R^{k+l}\times \tK$. 
For each integer $m$ choose a $g_m\in \lG$ 
such that  
\[
d\bigl( f_{\infty}^m \bigl(g_m (0, \tp_\infty)\bigr), (0,\tp_{\infty})\bigr)\le \diam(K).
\] 

The elements $(f_{\infty}^m(g_m a_1),\cdots,f_{\infty}^m(g_m a_N))$ 
are contained in $B_{1/\eps+\diam(K)}(0,\tp_\infty)$. 
Thus, there are $m_1\neq m_2$ such that 
$d(f_{\infty}^{m_1}(g_{m_1} a_j), f_{\infty}^{m_2}(g_{m_2} a_j))<\eps$ for $j=1,\ldots, N$. 
Consequently, $d(f_{\infty}^{m_1-m_2}( g a_j),a_j)<\eps$ for $j=1,\ldots, N$ with  
\[
g=f_{\infty}^{m_2-m_1}g_{m_1}^{-1}f_\infty^{m_1-m_2}g_{m_2}\in\lG.
\]
In summary, $f_\infty^{m_1-m_2}\circ g$ displaces any point in the ball of radius $\tfrac{1}{\eps}$ around $(0,p_\infty)$
by at most $4\eps$. Since $\eps$ was arbitrary, this clearly proves our claim.

Let $(\nu_b,g_b)$ be as above. 
For each $b$ we choose a large $i=i(b)\ge 2\nu_b$  and $g_b \in \pi_1(M_i)$ such that 
$f_i^{\nu_b}\circ g_b$ is in the weakly measured 
sense close to $f_{\infty}^{\nu_b}\circ g_b$. 
We also choose $i$ so large that   
$(f_{i(b)}^{\nu_b}\circ g_b)_{b\in \N}$ still has the zooming in property 
and converges to the identity.
This finishes the proof of Step 2. \\[2ex]
{\bf Further Notations.}
We may replace $p_i$ with any other point $p\in B_{1/2}(p_i)$.
Thus we may assume that $p_\infty$ is a regular point in 
$K$. 
Let $\pi_1(M_i,p_i,r)=\pi_1(M_i,p_i)(r)$ denote the subgroup 
of $\pi_1(M,p_i)$ generated by loops of length $\le r$. 
Similarly, recall that $\lG(r)$ is the group generated by elements which displace
$(0,\tp_\infty)$ by at most $r$.
By the Gap Lemma~{\eqref{lem: gap}} there is an $\eps>0$ and $\eps_i\to 0$ such  that $\pi_1(M_i,p_i,\eps)=\pi_1(M_i,p_i,\eps_i)$. 
{Moreover, $\lG(r)=\lG(\eps)$ for all $r\in (0,2\eps]$.
We put $\Gamma_i:=\pi_1(M_i,p_i)$ 
and $\Gamma_{i\eps}=\pi_1(M_i,p_i,\eps)$. }

\begin{step}\label{step: finite index case}
The desired contradiction arises if  the index 
$[\Gamma_i:\Gamma_{i\eps}]$ is bounded by some constant $C$ for all large $i$.
\end{step}

If $C$ denotes a bound on the index and $d=C!$, then 
$(f_i^j)^d$ leaves the subgroup $\Gamma_{i\eps}$ invariant
and we can find a new contradicting sequence for the manifolds 
$\tM_i/\Gamma_{i\eps}$.  Thus we may assume that $\Gamma_i=\Gamma_{i\eps}$.

As we have observed above, $\pi_1(M_i,p_i)=\pi_1(M_i,p_i,\eps_i)$ for some $\eps_i\to 0$.
We now choose a sequence $\lambda_i\to \infty$  very slowly 
such that 
\begin{enumerate}
     \item[$\bullet$] {$(\lambda_i M_i,p_i)\conv (\R^k\times C_{p_\infty}K,o)=(\R^{k'},0)$} with $k'>k$.
\item[$\bullet$] $\pi_1(\lambda_iM_i,p_i)$ is generated by loops of length $\le 1$. 
\item[$\bullet$]  $f_i^j\colon \lambda_i \tM_i\rightarrow \lambda_i \tM_i$ 
still has the zooming in property and still converges to the identity of the limit space. 
    \end{enumerate}

We put $f_i^{k+1}=\cdots=f_i^{k'}=\id$ and 
obtain a new sequence contradicting our induction assumption.

\begin{step}\label{step: normal}
There is an $H>0$ and a sequence of uniformly open 
 subgroups (see Definition~\ref{def: open sub}) $\Upsilon_i\subset  \Gamma_{i\eps}$ such that the index $[\Gamma_{i\eps}:\Upsilon_i]\le H$
and such that $ \Upsilon_i$ is normalized by a subgroup of $\Gamma_i$ with 
index $\le H$  for all large $i$.
\end{step}

After passing to a subsequence we may assume 
that $\Gamma_{i\eps}$ converges to the limit action 
of a closed subgroup $\lG_\eps\subset \lG$. 
Clearly $\lG_{\eps}$ contains the open subgroup $\lG(\eps)\subset \lG$. 
Since the kernel of $\pr\colon \lG\rightarrow \Iso(\R^l)$ is compact (as a closed subgroup of $\Iso(\tK)$), the images $\pr(\lG)$ and $\pr(\lG_{\eps})$ 
are closed subgroups. 
Moreover, $ \pr(\lG_{\eps})$ is an open subgroup of $\pr(\lG)$  since it contains $\pr(\lG(\eps))\supset \pr(\lG)_0$.
Using Fukaya and Yamaguchi ~\cite[Theorem 4.1]{FY} or that the component group of $\pr(\lG)$ 
is the fundamental group 
of the nonnegatively curved manifold $\pr(\lG)\backslash \Iso(\R^l)$, 
we see that $\pr(\lG)/\pr(\lG)_0$ is virtually abelian. 

Choose a subgroup $\lG'\lhd\lG$ of finite index such 
that 
$\pr(\lG')/\pr(\lG)_0$ is free abelian and
$\lG'=\pr^{-1}(\pr(\lG'))$.
Let  $\lG_{\eps}'=\lG'\cap \lG_{\eps}$. Then $\pr(\lG)_0\subset \pr(\lG_{\eps}')\subset \pr(\lG')$.
Moreover,
$\pr(\lG_{\eps}')$ is normal in $\pr(\lG')$
and 
$\pr(\lG')/ \pr(\lG_{\eps}') $ is a finitely generated abelian group.

The inverse image $\hat{\lG}_{\eps}:= \pr^{-1}(\pr(\lG_{\eps}))\cap \lG'$ 
is  a normal subgroup of $\lG'$ with the same quotient, i.e. $\lG'/\hat{\lG}_{\eps}=\pr(\lG')/ \pr(\lG_{\eps}')$.
Since the kernel of $\pr$ is compact, the open subgroup $\lG_{\eps}'$ 
is cocompact  in
$\hat{\lG}_{\eps}$ and thus its index $[\hat{\lG}_{\eps}:\lG_{\eps}']$ is finite.

In particular, we can say that there is some integer $H_1>0$ and a subgroup $\lG'$ 
with $[\lG:\lG']\le H_1$ such that  $[\lG_\eps: (g \lG_{\eps} g^{-1}\cap \lG_\eps)]<H_1$
for all $g\in \lG'$.
We want to check that there are only finitely many possibilities 
for $g \lG_{\eps} g^{-1}\cap \lG_\eps$ where $g\in \lG'$.

Let  $g\in \lG'$.
Notice that $\Gamma_{i\eps}$ is a uniformly open sequence 
of subgroups. 
Choose $g_i\in \Gamma_i$ with $g_i\to g$.
By Lemma~\ref{lem: open sub},
$\Upsilon_i:=g_i\Gamma_{i\eps}g_i^{-1}\cap \Gamma_{i\eps} $ 
converges to $g \lG_{\eps} g^{-1}\cap \lG_\eps$ and
the index of $\Upsilon_i $ in $\Gamma_{i\eps}$ 
is $\le H_1$ for all large $i$.
By Theorem~\ref{thm: finite generation}, the number of generators of 
$\Gamma_{i\eps}$ is bounded by some a priori constant. 
Therefore, $\Gamma_{i\eps}$ contains at most 
$h_0=h_0(n,H_1)$ subgroups of index $\le H_1$. 

Combining these statements we see that  
\[
\{g \lG_{\eps} g^{-1}\cap \lG_\eps\mid g\in \lG'\}=\{g_h\lG_{\eps}g_h^{-1}\cap \lG_\eps  \mid h=1,\ldots,h_0\}
\]
for suitably chosen elements $g_1,\ldots,g_{h_0}\in \lG'$.
We choose $g_{hi}\in \Gamma_i$ converging to $g_h\in \lG$. 
By Lemma~\ref{lem: open sub}, the sequence of subgroups
\[ 
\Upsilon_i:= \bigcap_{h=1}^{h_0} g_{hi} \Gamma_{i\eps} g_{hi}^{-1}
\]
is uniformly open and converges to
\[
\Upsilon_{\infty}:=\bigcap_{i=1}^{h_0}g_h\lG_{\eps}g_h^{-1}=\bigcap_{g\in \lG'}g\lG_{\eps} g^{-1}. 
\]
Clearly, $\lG'$ normalizes $\Upsilon_{\infty}$.
It is not hard to see that we can find 
elements $c^1_i,\ldots,c^\tau_i\in \Gamma_i$ 
that generate a subgroup of index $\le [\lG:\lG']$ 
such that  each $c^j_i$ converges  to an element $c^j_\infty\in \lG'$.
Since $c^j_i\Upsilon_i (c^j_i)^{-1}$ is uniformly open 
and converges to $c^j_\infty\Upsilon_\infty (c^j_\infty)^{-1}=\Upsilon_{\infty}$, 
one can now apply Lemma~\ref{lem: open sub} c) to see that 
$c^j_i$ normalizes $\Upsilon_i$ for large $i$.

Thus, the normalizer of $\Upsilon_i$ has finite index $\le [\lG:\lG']$ in $\Gamma_i$ 
for all large $i$ and Step~\ref{step: normal} is established.

\begin{step} The desired contradiction arises if, after passing 
to a subsequence, the indices $[\Gamma_i:\Gamma_{i\eps}]\in \N\cup \{\infty\}$
converge to $\infty$. 
\end{step}

By Step~\ref{step: normal} there is a subgroup 
$\Upsilon_{i}\subset \Gamma_{i\eps}$ of index $\le H$ 
which is normalized by a subgroup of $\Gamma_i$ of 
index $\le H$. 
Similarly to the beginning of the proof of Step~\ref{step: finite index case}, 
one can reduce the situation to the case of 
$\Upsilon_{i}\lhd \Gamma_i$.

Moreover, $\Upsilon_{i}\lhd \Gamma_i$ is uniformly open and hence the action of 
$\Gamma_{i}/\Upsilon_{i}$ on $M_i/\Upsilon_{i}$ 
is uniformly discrete 
and converges  to a properly discontinuous
action of the virtually abelian group  $\lG/\Upsilon_{\infty}$
on the space $\R^k\times (\R^l\times \tK)/\Upsilon_{\infty}$. 
It is now easy to see that $\Gamma_{i\eps}/\Upsilon_{i}$ 
contains an abelian subgroup of controlled finite index for large 
$i$.

After replacing $\Gamma_{i}$ once more by a subgroup of controlled finite index we may assume 
that $\Gamma_{i}/\Upsilon_{i}$ is abelian.
This also shows that without loss of generality 
$\Upsilon_{i}= \Gamma_{i\eps}$.

Recall that by Step 2 we have assumed that $f_i^j$ converges to the identity in the weakly measured sense for every $j=1,\ldots,k$. Therefore, the same is true for 
the commutator
$[f_i^j,\gamma_i]\in \Gamma_i$ if $\gamma_i\in \Gamma_i$ has bounded displacement.
This in turn implies that  $[f_i^j,\gamma_i]\in \Gamma_{i\eps}$ for all large $i$ 
and $j=1,\ldots,k$.

By Theorem~\ref{thm: finite generation}, we can find 
for some large $\sigma$ 
elements $d_1^i,\ldots,d_\sigma^i\in \Gamma_{i}$ with bounded displacement 
generating $\Gamma_{i}$. We can also assume that the first $\tau$ elements 
generate $\Gamma_{i\eps}$ for some $\tau<\sigma$.

Let $\hGamma_i:=\ml d_1^i,\ldots,d_{\sigma-1}^i\mr $. 
Note that $\hGamma_i\lhd \Gamma_i$ since  $\hGamma_i$ contains  $\Gamma_{i\eps}= \Upsilon_{i}\lhd \Gamma_i$ and $\Gamma_i/\Gamma_{i\eps}$ is abelian.

If for some subsequence and some integer  $H$ 
the group  $\hGamma_i $ has index $\le H$ 
in $\Gamma_i$ then we may replace $\Gamma_i$ by  $\hGamma_i$. Thus, without loss of generality, the order of the cyclic group
$\Gamma_i/\hGamma_i$ tends to infinity. 
The element $f_i^{k+1}:=d_\sigma^i\in \Gamma_i$ represents a generator of this 
factor group which has bounded displacement. 

We have seen above that $f_i^j$ normalizes $\hat{\Gamma}_i$ and 
$[f_i^{k+1}, f_i^j]\in \Gamma_{i\eps}\subset \hat{\Gamma}_i$, $j=1,\ldots,k$. 
Clearly $f_i^{k+1}$, being an isometry with bounded displacement, has the zooming in property.

We replace $M_i$ by $B_i:=\tM_i/\hat{\Gamma}_i$. 
Notice that  $(B_i,\hp_i)$ converges to $(\R^k\times (\R^l\times \tK)/\hat{\lG},\hp_\infty)$ 
where $\hat{\lG}$ is the limit group of $\hat{\Gamma_i}$
and $\hp_i$ is the image of $\tp_i$ under the projection $\tM_i\to B_i$. 
Since $\lG/\hat{\lG}$ is a noncompact group acting co-compactly on 
$(\R^l\times \tK)/\hat{\lG}$, we see that $(\R^l\times \tK)/\hat{\lG}$ 
splits as $\R^p\times K'$ with $K'$ compact and $p>0$. 
If $p>1$, we put $f_i^j=\id$ for $j=k+2,\ldots,k+p$.

Thus the induction assumption applies: 
There is a 
positive integer $C$ such that
the following holds for all large $i$:

We can find a subgroup $\gN^i$ of $\pi_1(B_i)$ of index $\le C$ 
and a nilpotent chain of normal subgroups $\{e\}=\gN_0\lhd\ldots\lhd\gN_{n-k-1}=\gN^i$
such that the quotients are cyclic. 
Moreover, the groups are invariant under the automorphism 
induced by conjugation of $(f_i^j)^{C!}$ 
and the induced automorphism of $\gN_{h+1}/\gN_h$ is the identity, $j=1,\ldots,k+1$.

We put $d=C!$, 
and consider the group $\bar{\gN}^i\subset\pi_1(M_i)$ generated by $(f_i^{k+1})^{d}$ and $\gN^i$. 
Clearly, the index of $\bar{\gN}^i$  in $\pi_1(M_i)$ is bounded by  $C\cdot d$. 

Moreover, the chain $\{e\}=\gN_0\lhd\ldots\lhd\gN_{n-k-1}\lhd \gN_{n-k}:=\bar{\gN}^i$ 
is normalized by the elements $(f_i^j)^d$ ($j=1,\ldots,k$) 
and the action on the cyclic quotients is trivial. 

In other words, the sequence fulfills the conclusion  of our theorem -- 
a contradiction. 
\end{proof}

\begin{rem}\label{rem: checol case}
If one was only interested in proving that 
the fundamental group contains a polycyclic subgroup 
of controlled index, the proof would simplify considerably. 
First of all, one would not need 
any diffeomorphisms $f_i^j$ to make the induction work. 
In the proof of Step 1 the use of the rescaling theorem 
could be replaced with the more elementary 
Product Lemma~\ref{prodlem}. Step 2 would become 
unnecessary. The core of the remaining arguments would 
be the same although they would
simplify somewhat.  
\end{rem}

%%%%%%%%%%%%%%%%%%%%%%%%%%%%%%%%%%%%%%%%%%%%%%%%%%%%%%%%%%%%%%%%%%%%%%%%%%%%%%%%%%%%%%%%%%%%%%%
%%%%%%%%%%%%%%%%%%%%%%%%%%%%%%%%%%%%%%%%%%%%%%%%%%%%%%%%%%%%%%%%%%%%%%%%%%%%%%%%%%%%%%%%%%%%%%%
% %%%%%%%%%%%%%%%%%%%%%%%%%%%%%%%%%%%%%%%%%%%%%%%%%%%%%%%%%%%%%%%%%%%%%%%%%%%%%%%%%%%%%%%%%%%%%%%%
% %%%%%%%%%%%%%%%%%%%%%%%%%%%%%%%%%%%%%%%%%%%%%%%%%%%%%%%%%%%%%%%%%%%%%%%%%%%%%%%%%%%%%%%%%%%%%%%%
%%%%%%%%%%%%%%%%%%%%%%%%%%%%%%%%%%%%%%%%%%%%%%%%%%%%%%%%%%%%%%%%%%%%%%%%%%%%%%%%%%%%%%%%%%%%%%%%%

\section{Margulis Lemma}\label{sec: margulis}

\begin{proof}[Proof of Theorem~\ref{intro: margulis}]
Let $B_1(p)$ be a metric ball in complete $n$-manifold with $\Ric>-(n-1)$ 
on $B_1(p)$,
$\tN$ the universal cover of $B_1(p)$ 
and let $\tp\in \tN$ be a lift of $p$.

\begin{stp} There are universal positive constants $\eps_1(n)$ and $C_1(n)$ 
such that  the  group
\[
\Gamma:=\Bigl\ml\bigl\{g\in \pi_1(B_1(p),p)\,\bigm|\, d(q,gq)\le \eps_1(n)\mbox{ for all $q\in B_{1/2}(\tp)$}\bigr\}\Bigr\mr
\]
has a subgroup of index $\le C_1(n)$ 
which has a nilpotent basis of length at most $n$.\\[-1ex]
\end{stp}

Put $N=\tN/\Gamma$ and let $\hat{p}$ denote the image 
of $\tp$. By the definition of $\Gamma$, 
for each point $q\in B_{1/2}(\hp)$, the fundamental 
group 
$\pi_1(N,q)$ is generated by loops of length $\le \eps_1(n)$.   
Moreover, $\overline{B_{3/4}(\hat p)}$ is compact.

Assume, on the contrary, that the statement is false. 
Then we can find a sequence of 
pointed $n$-dimensional  manifolds $(N_i,p_i)$ satisfying 
\begin{enumerate}
\item[$\bullet$]
$\Ric_{N_i}\ge -(n-1)$, 
\item[$\bullet$] $\overline{B_{3/4}(p_i)}$ is compact,
\item[$\bullet$] for each point $q\in  B_{1/2}(p_i)$ the fundamental group
$\pi_1(N_i,q)$ is generated by loops of length $\le 2^{-i}$, and 
\item[$\bullet$] $\pi_1(N_i,p_i)$ does not contain a subgroup 
of index $\le 2^i$ which has a nilpotent basis of length $\le n$.
\end{enumerate}

After passing to a subsequence we can assume that $(N_i,p_i)$ converges 
to $(X, p_\infty)$. 
We  choose $q_i\in B_{1/4}(p_i)$ such that $q_i$ converges to a regular 
point $q\in X$ and
$\lambda_i\to \infty$ very slowly such 
that $\pi_1(\lambda_i N_i,q_i)$ is still generated by loops of length $\le 1$
and such that $ (\lambda_i N_i,q_i)$ converges to $C_qX\cong \R^k$ for some $k\ge 0$. 

Notice that the rescaled sequence $M_i=\lambda_iN_i$ satisfies  
$\Ric_{M_i}\ge -\tfrac{n-1}{\lambda_i^2}\to 0$ and 
$\overline{B_{r_i}(q_i)}$ is compact with $r_i=\tfrac{\lambda_i}{2}\to \infty$. 
Thus the existence of this sequence contradicts the Induction Theorem ~\ref{thm: induction}
 with $f_i^1=\cdots=f_i^k=\id$.

We can now finish the proof of the theorem by establishing.

\begin{stp} Consider $\eps_1(n)>0$ and $\Gamma$  from Step 1. 
Then there are $\eps_2(n), C_2(n)>0$ such that the group
\[
\gH:=\Bigl\ml\bigl\{g\in \pi_1(B_1(p),p)\mid d(\tp,g\tp)\le \eps_2(n)\bigr\}\Bigr\mr
\]
satisfies: $\Gamma\cap \gH$ has index at most $C_2(n)$ in $\gH$. \\[-2ex]
\end{stp}

We will provide (in principle) effective bounds on $\eps_2$ and $C_2$
depending on $n$ and the (ineffective) bound $\eps_1(n)$. Put
\[
\Gamma':=\Bigl\ml\bigl\{g\in \pi_1(B_1(p),p)\mid d(q,gq)\le \eps_1(n)\mbox{ for all $q\in B_{2/3}(\tp)$}\bigr\}\Bigr\mr\subset \Gamma.
\]

By Theorem~\ref{thm: finite generation},  there exists $h=h(n)$ 
such that $\gH$ can be generated by some
$b_1,\ldots,b_h\in \gH$ satisfying $d(\tp,b_i\tp)\le 4\eps_2(n)$ for any $i=1,\ldots h$.
We can obviously assume that the generating set $\{b_1,\ldots,b_h\}$ contains inverses of all its elements.
We proceed in three substeps.\\[2ex]
{\bf Claim 1.} There is a positive integer $L(n,\eps_1(n))$ such that 
for all $\eps_2(n)< \tfrac{1}{100L}$ 
and any choice of $g_i \in \{b_1,\ldots,b_h\}$, $i=1,\ldots, L$ 
we can find $l,k\in \{1,\ldots, L\}$ with $l\le k$ and  
$g_l\cdot g_{l+1}\cdots g_k\in \Gamma'$. \\[1ex]
We assume  $\eps_2(n)< \tfrac{1}{100L}$ and we will show that there 
is an a priori estimate for $L$. 
Notice that 
$d(\tp,g_1\cdots g_l \tp)< 4\eps_2(n)l\le \tfrac{1}{25}$ holds for $l=1,\ldots,L$. 
Thus $g_1\cdots g_l$ maps the ball $B_{2/3}(\tp)$ into the ball 
$B_{3/4}(\tp)$.

We choose a maximal collection of points 
$\{a_1,\ldots, a_m\}$
in $B_{2/3}(\tp)$ with pairwise distances $\ge \eps_1(n)/4$. 
It is immediate from Bishop--Gromov that $m$ can be estimated just in terms of 
$\eps_1$ and $n$.

Assume we can choose
$g_i\in \{b_1,\ldots,b_h\}$, $i=1,\ldots,L$, such that
$g_l\cdots g_k\not\in \Gamma'$ 
for all $1\le l\le k\le L$. 
This implies that  $d(g_l\cdots g_k a_u,a_u)\ge \eps_1(n)/4$ 
for some $u=u(l,k)\in \{1,\ldots,m\}$. 
Hence,
\[d\bigl(g_1\cdots g_k a_{u(l,k)},\, g_1\cdots g_{l-1} a_{u(l,k)}\bigr)\ge \eps_1(n)/4,\,\,\mbox{ for $1<l\le k<L$.}\]

We consider the diagonal action of $\Gamma$ on the $m$-fold product
$\tN_i^m$ and the point $a=(a_1,\ldots, a_m)$. 
We just showed $d(g_1\cdots g_k a,g_1\cdots g_l a)\ge \eps_1(n)/4$
for all $1\le k<l\le L$. 
Thus there are $L$ points in the $m$-fold product $B_{3/4}(\tp)^m$
which are $\eps_1(n)/4$ separated.
Now the Bishop--Gromov inequality provides an a priori bound for $L$. \\[2ex]
{\bf Claim 2.} Choose $L$ as in Claim 1 and assume $\eps_2(n)\le \tfrac{1}{100L}$.  
Then for every word $w=b_{\nu_1}\ldots b_{\nu_l}$  of length
$l\le L$ we have $w\Gamma'w^{-1}\subset \Gamma$. \\[2ex]
Let $\gamma'\in \pi_1(B_1(p))$ be an element that displaces every 
point in $B_{2/3}(\tp)$ by at most $\eps_1(n)$.
By the definition of $\Gamma'$ it suffices to show that
$w\gamma'w^{-1}\in \Gamma$. 
Let $q\in B_{1/2}(\tp)$. Since $\eps_2(n)<\tfrac{1}{100L}$, it follows that 
$w^{-1}\tp\in B_{1/25}(\tp)$ and hence 
$w^{-1}q\in B_{2/3}(\tp)$.
Therefore
\[d(q,w\gamma'w^{-1}q)=d(w^{-1}q,\gamma'w^{-1}q)<\eps_1(n) \mbox{ for all $ q\in B_{1/2}(\tp)$.}\] 
Hence, $w\gamma'w^{-1}q\in \Gamma$ by the definition of $\Gamma$
and Claim 2 is proven. 
Step 2 now follows from\\[2ex]
{\bf Claim 3.} Every element $h\in \gH$ is of the form 
$h=w\gamma$ with $\gamma\in \Gamma$ 
and $w=b_{\nu_1}\cdots b_{\nu_l}$ with $l\le L$. \\[2ex]
We write $h=w\cdot \gamma$, where $\gamma\in \Gamma$, $w$ is a word of length $l$
and $l$ is chosen minimal. 
It suffices to prove that $l\le L$. 
Suppose $l>L$. Then we can apply Claim~1 to the tail of $w$  and obtain
$w=w_1\gamma'w_2$ 
with $\gamma'\in \Gamma'$, $w_2$ a word of length $< L$  
and the word-length of $w_1\cdot w_2$ is smaller than the length of $w$. 

By Claim 2, $w_2^{-1}\gamma'w_2\in \Gamma$. 
Putting $\gamma_2=w_2^{-1}\gamma'w_2\gamma$ gives
$h=w_1\cdot w_2\gamma_2$ 
-- a contradiction, as the length of $w_1w_2$ is smaller than $l$.
\end{proof}

\begin{cor}\label{cor: rank n} In each dimension $n$ there exists $\eps>0$ 
such that the following holds for any complete $n$-manifold $M$ and $p\in M$ with
$\Ric>-(n-1)$ on $B_1(p)$.
 
If the image of $\pi_1(B_\eps(p))\rightarrow \pi_1(B_1(p))$ 
contains a nilpotent group of rank $n$, then $M$ 
is homeomorphic to a compact infranilmanifold.
\end{cor}

Actually, we will only show that $M$ is homotopically equivalent to 
an infranilmanifold. By work of Farrell and Hsiang ~\cite{FH}
this determines the homeomorphism type in dimensions above $4$. 
The $4$ -dimensional case follows from work of Freedman--Quinn \cite{FQ}. 
Lastly, the 3-dimensional case follows from Perelman's solution of the geometrization conjecture.
% ~\cite{PerG1,PerG2,PerG3}, see \cite{KlLo1, klLo2} or ~\cite{BBGM} 
% for full details of the proof.
\begin{comment}
In fact,  one does not need the full strength of the geometrization conjecture and the 3-dimensional case follows from the 3-dimensional Poincare conjecture as follows.  A nil 3-manifold is obviously Seifert and hence so is any infranil 3-manifold by~\cite{Scott1} as it's finitely covered by a nilmanifold. By Poincare conjecture (see the above references or ~\cite{MoTi1} for the proof), a homotopy infranil 3-manifold is irreducible and by above it's homotopy equivalent to a Seifert manifold. By~\cite{Scott2}, an irreducible 3-manifold homotopy equivalent to a Seifert manifold is actually homeomorphic to it and the result follows.
\end{comment}

\begin{proof} [Proof of Corollary~\ref{cor: rank n}]
Notice that it suffices to prove that $M$ is aspherical: 
{In fact, since the cohomological dimension
of a rank $n$ nilpotent group is  $n$, the group $\pi_1(M)$}  must then be a torsion free virtually 
nilpotent group of rank $n$ and thus, by a result of  Lee and Raymond \cite{LR},
it is isomorphic to the fundamental group of an infranilmanifold. 
Therefore $M$ is homotopically equivalent to an infranilmanifold 
and as explained above this then gives the result.

We argue by contradiction and assume that 
we can find $\eps_i\to 0$ and  a sequence of complete manifolds $M_i$
with $\Ric>-(n-1)$ on $B_1(p_i)$ 
such that  the image of $\pi_1(B_{\eps_i}(p_i),p_i)\rightarrow \pi_1(B_1(p_i),p_i)$ 
contains a nilpotent group of rank $n$ and $M_i$ is not aspherical.

By  the Margulis Lemma (Theorem~\ref{intro: margulis}),
 the nilpotent group 
can be chosen to have index $\le C(n)$ in the image.
Arguing on the universal cover of $B_1(p_i)$ 
it is not hard to deduce that there is $\delta_i\to 0$ 
such that for all $q_i\in B_{1/10}(p_i)$ the image 
of $\pi_1(B_{\delta_i}(q_i),q_i)\rightarrow \pi_1(B_1(p_i),p_i)$
also contains a nilpotent subgroup of rank $n$.

Next we observe that $\diam(M_i,p_i)\to 0$. 
In fact, otherwise we could, after passing to a subsequence, 
assume that $(M_i,p_i)$ converges in the Gromov--Hausdorff sense 
to $(Y,p_\infty)$ with $Y\ne pt$. Choose $q_i\in B_{1/10}(p_i)$ 
converging to a regular point $q_\infty$ in the limit $Y$. 
Similarly, choose $\lambda_i\le \tfrac{1}{\sqrt{\delta_i}}$ { slowly converging to infinity}
such that  {$(\lambda_iM_i,q_i)\conv (C_{q_\infty}Y,o)=(\R^k,0)$ }
for some $k>0$. 
Let $\tN_i$ denote the universal cover of $B_1(p_i)$ 
and $\tq_i$ a lift of $q_i$. 
Let $\Gamma_i$ be the subgroup of the deck transformation group 
generated by those elements which displace $q_i$ by at most $\delta_i$. 
By construction, $\Gamma_i$ contains a nilpotent group 
of rank $n$. Put $N_i=\tN_i/\Gamma_i$.
Since $(\lambda_iN_i,q_i)\conv (\R^k,0)$  with $k>0$, 
we get a contradiction to the Induction Theorem~{\eqref{thm: induction}}.

Thus, $\diam(M_i)\to 0$ and in particular, $M_i=B_1(p_i)$
is a closed manifold for all large $i$.
After a slow rescaling we may assume in addition that $\Ric_{M_i}\ge -h_i\to 0$. 
Next we plan to show 
that the universal cover of $M_i$ converges to $\R^n$. This will follow from\\[2ex]
{\bf Claim 1.} Suppose a sequence of complete $n$-manifolds $(N_i,p_i)$
with lower Ricci curvature bound $>-h_i\to 0$
converges to $(\R^k\times K,p_\infty)$ with $K$ compact. Also assume that  $\pi_1(N_i)$ is generated by 
loops of length $\le R$ and contains a nilpotent subgroup of rank $n-k$. 
Then the universal cover of $(\tN_i,\tp_i)$ converges to $(\R^n, 0)$.\\[2ex]
We argue by reverse induction on $k$. 
The base of induction $k=n$ is obvious.

By the Induction Theorem,
after passing to a bounded cover, we may assume
that $\pi_1(N_i)$ itself is a torsion free nilpotent group of rank $n-k$. 

Choose $\gN_i\lhd \pi_1(N_i)$ with $\pi_1(N_i)/\gN_i\cong \Z$. 
Then $(\tN_i/\gN_i,\hp_i)$ converges to $(\R^{l}\times K',(0,p_\infty))$ for some $l>k$ and $K'$ 
compact 
and the claim follows from the induction assumption. \\[2ex] 
Therefore $(\tM_i,\tp_i)$ converges to $(\R^n,0)$. 
From the Cheeger--Colding Stability Theorem (see Theorem~\ref{noncol})
it follows that for any $R>0$
the ball $B_R(p)$ is contractible in $B_{R+1}(p)$ for all $p\in B_R(\tp_i)$ 
and $i\ge i_0(R)$. 
Since we have a cocompact 
deck transformation group with nearly dense orbits, the result actually holds for all 
$p\in \tM_i$. 

In order to show that $M_i$ is aspherical we may replace $M_i$ by a bounded cover and 
thus, by Theorem~\ref{intro: margulis} without loss of generality,
 $\pi_1(M_i)$ has a nilpotent basis of length $\le n$. 
Because $\rank(\pi_1(M_i))\ge n$ it follows that $\pi_1(M_i)$ is torsion free.
Therefore we
can choose subgroups $\{e\}=\gN_0^i\lhd \cdots\lhd \gN_n^i=\pi_1(M_i)$ 
with $\gN_j^i/\gN_{j-1}^i\cong \Z$.

{In the rest of the proof we will slightly abuse notations and sometimes drop basepoints when talking about pointed Gromov--Hausdorff convergence when the base points are clear.

Note that  the above claim easily implies that $\tM_i/\gN_j^i\to \R^{n-j}$ for all $j=0,\ldots n$.}
\\[2ex]
{\bf Claim 2.} 
$\gN_j^i\star B_{r}(\tp_i)$ is contractible in $\gN_j^i\star B_{4^{4^{j}} r}(\tp_i)$ 
for $j=0,\cdots,n$ and $r\in \bigl[1,4^{4^{(2n-j)^2}}\bigr]$ and all large i.\\[2ex]
We want to prove the statement by induction on $j$. 
For $j=0$ it holds as was pointed out above. 
Suppose it holds for $j<n$ and we need to prove it for $j+1$. Choose $g\in \gN_{j+1}^i$ 
representing a generator of $\gN_{j+1}^i/\gN_j^i$.

Notice that $\gN_{j+1}^i\star B_R(\tp_i)/\gN_j^i$ converges  to $\R\times B_R(0)\subset \R^{n-j}$ as $i\to\infty$ { where $B_R(0)$ is the ball in $\R^{n-j-1}$. }
Moreover, the action of $\gN_{j+1}^i/\gN_j^i$ on the set converges to 
the $\R$ action on $\R\times B_R(0)$ given by translations. 

We can also find 
finite index subgroups $\bar{\gN}^i_{j+1}\subset \gN^i_{j+1}$ with $\gN_j^i\subset \bar\gN_{j+1}^i$ 
such that $\gN_{j+1}^i\star B_R(\tp_i)/\bar \gN_{j+1}^i$ 
converges to  $\mathbb{S}^1\times B_R(0)\subset \Sph^1\times \R^{n-j}$
where $\mathbb{S}^1$ has diameter $10^{-n}$.
It is easy to construct a smooth map 
$\bsigma\colon \gN_{j+1}^i\star B_R(\tp_i)/\bar\gN_{j+1}^i\rightarrow \Sph^1$ 
that is arbitrary close to the projection 
map $\mathbb{S}^1\times B_R(0)\rightarrow \Sph^1$
in the Gromov--Hausdorff sense.

We can lift $\bsigma$ to 
a map $\sigma\colon \gN_{j+1}^i\star B_R(\tp_i)\rightarrow \R$.
Notice that $\sigma$ commutes 
with the action of $\bar\gN_{j+1}^i$ where
$\bar\gN_{j+1}^i/\gN_j^i$ can be thought of as the deck transformation 
group of the covering $\R\rightarrow \mathbb{S}^1$.

It suffices to 
show that the inclusion $\gN_{j+1}^i\star B_r(\tp_i)\rightarrow \gN_{j+1}^i\star B_{4^{4^{j+1}}r}(\tp_i)$
induces the trivial map on the level 
of homotopy groups. % for  $1\le r\le 4^{4^{(n-j-1)^2}}$. 
Thus, we have to show that for any $k>0$  any map
$\iota\colon \Sph^k\rightarrow \gN_{j+1}^i\star B_r(\tp_i)$ 
is null homotopic in $\gN_{j+1}^i\star B_{4^{4^{j+1}}r}(\tp_i)$.

We can assume that $\iota$ is smooth. 
The image of $\sigma\circ \iota$ is given by 
an interval $[a,b]$ in $\R$. 
We choose a $10^{-n}$-fine finite 
subdivision $a<t_1<\cdots<t_h<b$ of the interval 
such that the $t_\alpha$ are regular values.
Thus $\iota^{-1}(\sigma^{-1}(t_\alpha))=H_\alpha$
is a smooth hypersurface in $\Sph^k$ for every $\alpha$.

Notice that by construction, the image  
$\iota(H_\alpha)$ is contained in $g\gN_j\star B_{2r}(\tp_i)$ 
for some $g\in \bar\gN_{j+1}^i$. 
By induction assumption, $\iota_{|H_{\alpha}}$  is homotopic to 
a point map in {$g\gN_j^i\star B_{4^{4^j}2r}(\tp_i)$.}
We homotope $\iota$ into $\tiota$ such that 
$\tiota(H_{\alpha})$ is a point for all $\alpha$ 
and $\tiota$ is $4^{4^j}4r$ close to $\iota$. 

Consider now all components of $H_\alpha$ for all $\alpha$. 
They divide the sphere  into connected regions such that 
each boundary component of a region is mapped by $\tiota$ to a point 
and the whole region is mapped to a set 
$g\gN_j^i\star B_{4^{4^j}8r}(\tp_i)$ for some $g$. 

Thus the map $\tiota$ restricted to  a region with crushed boundary components 
is { null homotopic in $g\star B_{4^{2\cdot 4^j}8r}(\tp_i)$ by the induction assumption.}
Clearly, this implies that $\iota$ is null homotopic in
$\gN_{j+1}^i\star B_{4^{4^{j+1}}r}(\tp_i)$.

This finishes the proof of  Claim 2. Notice that $\gN_{n}^i\star B_{1}(\tp_i)=\tM_i$ for large $i$ since $\diam(M_i)\to 0$. Thus $\tilde M_i$ is contractible by  Claim 2. 
\end{proof}

%%%%%%%%%%%%%%%%%%%%%%%%%%%%%%%%%%%%%%%%%%%%%%%%%%%%%%%%%%%%%%%%%%%%%%%%%%%%%%%%%%%%%%%%%%%%%%%%%%%%%%%%
%%%%%%%%%%%%%%%%%%%%%%%%%%%%%%%%%%%%%%%%%%%%%%%%%%%%%%%%%%%%%%%%%%%%%%%%%%%%%%%%%%%%%%%%%%%%%%%%%%%%%%%%
%%%%%%%%%%%%%%%%%%%%%%%%%%%%%%%%%%%%%%%%%%%%%%%%%%%%%%%%%%%%%%%%%%%%%%%%%%%%%%%%%%%%%%%%%%%%%%%%%%%%%%%%
%%%%%%%%%%%%%%%%%%%%%%%%%%%%%%%%%%%%%%%%%%%%%%%%%%%%%%%%%%%%%%%%%%%%%%%%%%%%%%%%%%%%%%%%%%%%%%%%%%%%%%%%

\section{Almost nonnegatively curved manifolds with maximal first Betti number}\label{sec: correction}

This is the only section where we assume lower sectional curvature bounds.
The main purpose is to prove

\begin{cor}\label{cor: correction} In each dimension there 
are  positive constants $\eps(n)$ and $p_0(n)$ 
such that for all primes $p>p_0(n)$ the following holds.

Any manifold with $\diam(M,g)^2K_{sec}\ge -\eps(n)$ and $b_1(M,\Z_p)\ge n$ 
is {diffeomorphic} to a nilmanifold. 
Conversely, for every $p$, every compact $n$-dimensional
nilmanifold covers another (almost flat) nilmanifold $M$ 
with $b_1(M,\Z_p)=n$. 
\end{cor}
The second part of the Corollary is fairly elementary 
but it provides counterexamples to a theorem of Fukaya and Yamaguchi ~\cite[Corollary 0.9]{FY}
which asserted that only tori should show up. 
The second part  follows from the following Lemma {and the fact that every nilmanifold is almost flat, i.e.  for any $i>0$ it admits a metric with $\diam(M,g)^2| K_{sec}|\le \frac{1}{i}$ .}

\begin{lem}\label{lem:zp-nilp} Let $\Gamma$ be a torsion free nilpotent group 
of rank $n$,  and let $p$ be  a prime number. 
\begin{enumerate}
      \item[a)] There is a subgroup 
of finite index $\Gamma'\subset \Gamma$ for which we can find a surjective homomorphism 
$\Gamma'\rightarrow (\Z/p\Z)^n$.
\item[b)] We can find a torsion free nilpotent group $\hat\Gamma$ 
containing $\Gamma$ as finite index subgroup such 
that there  a surjective homomorphism 
$\hat\Gamma\rightarrow (\Z/p\Z)^n$.
\end{enumerate}
\end{lem}

\begin{proof}
{\em a)}  We argue by induction on $n$. 
The statement is trivial for $n=1$. Assume it holds for $n-1$. 
Choose a subgroup $\Lambda\lhd \Gamma$ with $\Gamma/\Lambda\cong \Z$. 
By the induction assumption, we can find 
a subgroup $\Lambda'\subset \Lambda$ of finite index 
and $\Lambda''\lhd \Lambda'$ with $\Lambda'/\Lambda''\cong (\Z/p\Z)^{n-1}$.
Let $g\in \Gamma$ represent a generator of $\Gamma/\Lambda$. 
Since $\Lambda$ contains only finitely many subgroups
with of index $\le [\Lambda:\Lambda'']$, it follows that 
$g^l$ is in the normalizer of $\Lambda'$ and $\Lambda''$ 
for a suitable $l>0$. After increasing $l$ further 
we may assume that the automorphism of $\Lambda'/\Lambda''\cong (\Z/p\Z)^{n-1}$ 
induced by the conjugation by $g^l$ is the identity.  

Define  $\Gamma'$ as the group generated by $\Lambda'$ and $g^l$. 
It is now easy to see that $\Gamma'/\Lambda''$ is isomorphic 
to $(\Z/p\Z)^{n-1}\times \Z$ and thus $\Gamma'$ has a surjective homomorphism 
to $(\Z/p\Z)^n$.

{\em b).} An analysis of the proof of a) shows that 
in a) the index of $\Gamma'$ in $\Gamma$ can be bounded by a constant
$C$ only depending on $n$ and $p$. 
Let $\gL$ be the Malcev completion of $\Gamma$,
i.e. $\gL$ is the unique $n$-dimensional nilpotent 
Lie group containing $\Gamma$ as a lattice. Let 
$\exp\colon \Ll\rightarrow \gL$ denote the exponential map of the group. 
It is well known that the group $\bGamma$ generated by 
$\exp(\tfrac{1}{C!}\exp^{-1}(\Gamma))$ contains $\Gamma$ 
as a finite index subgroup. 
By part a) we can assume now that $\bGamma$ contains a subgroup $\hat\Gamma$
of index $\le C$ such that there is a surjective homomorphism 
to $\hat\Gamma\rightarrow (\Z/p\Z)^n$.  {By construction, any subgroup of $\bGamma$ of index $\le C$ must contain $\Gamma$ and hence
we are done.}
\end{proof}

%We will also need the following general result.
\begin{lem}\label{lem:nilg}
For any $n$ there exists $p(n)$ such that if $\Gamma$ is a torsion free virtually nilpotent of rank $n$  which admits an epimorphism onto  $(\Z/p\Z)^n$ for some $p\ge p(n)$ then $\Gamma$ is nilpotent.
\end{lem}

\begin{proof}
By \cite{LR}, $\Gamma$ is an almost crystallographic 
group, that is, there is an $n$-dimensional nilpotent Lie group 
$\gL$, a compact subgroup $\gK\subset \Aut(\gL)$ 
such that $\Gamma$ is isomorphic to a lattice in 
$ \gK\ltimes \gL$. By a result of Auslander 
(see \cite{LR} for a proof), the projection 
of $\Gamma$ to $\gK$ is finite and we may assume 
that the projection is surjective. 
Thus $\gK$ is a finite group and it is easy to see 
that the action of $\gK$ on $\gL/[\gL,\gL]$ is effective.
It is known that $\gN':=\Gamma\cap [\gL,\gL]$ 
is a lattice in $[\gL,\gL]$.
Thus, the image $\bGamma:=\Gamma/\gN'$ of $\Gamma$ 
in $\gK\ltimes\gL/[\gL,\gL]$ is discrete and cocompact.
Since $\gK$ acts effectively on $\gL/[\gL,\gL]$, 
we can view $\gK\ltimes\gL/[\gL,\gL]$ as a cocompact subgroup
of $\Iso(\R^k)$ where $k=\dim(\gL/[\gL,\gL])$.

In particular, $\bGamma$ is a crystallographic group. 
As the projection of $\Gamma$ to $\gK$ is surjective, 
$\bGamma$ is abelian if and only if $\gK$ is trivial.
If it is not abelian, then
\[\rank(\bar\Gamma/[\bar\Gamma,\bar\Gamma])<\rank(\bar\Gamma )=k.\]

By the third Bieberbach theorem 
there are only finitely many crystallographic groups of any given rank, 
and consequently there are only finitely many possibilities for the isomorphism type of $\bar\Gamma/[\bar\Gamma,\bar\Gamma]$. Obviously,  any homomorphism $\bar\Gamma\to  (\Z/p\Z)^k$ factors through $\bar\Gamma\to \bar\Gamma/[\bar\Gamma,\bar\Gamma]\to (\Z/p\Z)^k$; this easily implies that
there is $p_0$ such 
that there is no surjective homomorphism $\bar\Gamma\to (\Z/p\Z)^k$ 
for $p\ge p_0$.

Therefore, there is no surjective homomorphism
$\Gamma\to (\Z/p\Z)^n$ for $p>p_0.$
If we put $p(n)=p_0$, then $\bGamma $ is abelian, $\gK=\{e\}$ and hence $\Gamma$ is nilpotent. 
\end{proof}

Corollary~\ref{cor: correction} follows from  Lemma~\ref{lem:zp-nilp}, Lemma~\ref{lem:nilg} and the following result
\begin{cor}\label{cor: n-1 basis}
In each dimension $n$
there {are positive constants $C(n)$ and $\eps(n)$} 
such that for any complete manifold 
with $K_{sec}>-1$ on  $B_1(p)$ for some $p\in M$ one of the following holds 
\begin{enumerate}
\item[a)] The image $\pi_1(B_{\eps(n)}(p))\rightarrow B_1(p)$ 
contains a subgroup $\gN$ of index less than $C(n)$ such that
$\gN$ has a nilpotent basis of length $\le n-1$ or 
\item[b)] $M$ is compact and homeomorphic to an infranilmanifold. {Moreover,
any finite cover with a nilpotent fundamental group is diffeomorphic to 
a nilmanifold.}
\end{enumerate}
\end{cor}

\begin{proof}
Consider a sequence of manifolds $(M_i,p_i)$ 
with $K_{sec}>-1$ such that $M_i$ does not satisfy a) 
with $C=i$ and $\eps(n)=\tfrac{1}{i}$. 
As in proof of Corollary~\ref{cor: rank n}, it is clear 
that $\diam(M_i)\to 0$. 

By the Margulis Lemma or by work of \cite{KPT},
we may assume that $\pi_1(M_i)$ is nilpotent. We plan to show that $\pi_1(M_i)$ has rank $n$ for large $i$.

The fact that $\pi_1(M_i)$ does not contain a subgroup of bounded index with a nilpotent 
basis of length $\le n-1$, guarantees a certain extremal 
behavior if we run through the proof of the Induction Theorem.
It will be clear that the blow up limits 
$ \R^k\times K$, that occur if we go through the procedure provided in the proof of 
the Induction Theorem, are flat -- more precisely each $K$ has to be a torus. 
This follows from the fact that we can never run into 
a situation where Step~\ref{step: finite index case} of the proof
of the Induction Theorem applies.

In particular,  if we rescale the manifolds to have diameter 
$1$ they must converge to a torus $K=T^h$. 
We can now use Yamaguchi's  fibration theorem ~\cite{Yam}
to get a fibration $F_i\to M_i\rightarrow T^h$. 

Let $\Gamma_i$ denote the kernel of $\pi_1(M_i)\rightarrow \pi_1(T^h)$.
In the next $h$ steps of the procedure in the proof of the Induction Theorem we replace $M_i$ by 
$\tM_i/\Gamma_i$ endowed with $h$ deck transformations 
 $f_i^1,\ldots,f_i^h$ projecting to generators of $\pi_1(T^h)$. 

$(\hM_i:=\tM_i/\Gamma_i,\hp_i)$ converges to $(\R^h,0)$.
The next step in the Induction Theorem would be 
to rescale $\hM_i$  so that $(\lambda_i\hM_i,\hp_i)\conv (\R^h\times K',\hp_\infty)$ 
with  $K'$ not a point and it will be clear again that $K'\cong T^{h'}$ is a torus. 

It is easy to see that $\tfrac{1}{\lambda_i}$ is necessarily comparable to the size 
of the fiber $F_i$ of the Yamaguchi fibration. 

From the exact homotopy sequence we can deduce that 
$\Gamma_i=\pi_1(\hM_i)\cong \pi_1(F_i)$.
The fibration theorem ensures that $F_i$ fibers over a new torus $T^{h'}$.
Let $\Gamma_i'$ denote the kernel of $\pi_1(F_i)\rightarrow \pi_1(T^{h'})$. 
Clearly $\pi_1(M_i)$ has rank $n$  if and only if 
$\Gamma_i'$ has rank $n-h-h'$. 
After finitely many similar steps this shows that 
$\pi_1(M_i)$ has rank $n$. 
By Corollary~\ref{cor: rank n} the first part of b) follows.
 To get the second part of b) observe that if  $\pi_1(M_i)$ is nilpotent and torsion free then by the proof of
~\cite[Theorem 5.1.1]{KPT}, for all large $i$ we have that $M_i$ smoothly fibers over a nilmanifold with simply connected factors. The fibers obviously have to be trivial as $M_i$ is aspherical which means that $M_i$ is diffeomorphic to a nilmanifold.
\end{proof}

\begin{problem}\begin{enumerate}\item[a)] We suspect that in the case of almost nonnegatively curved manifolds 
Corollary~\ref{cor: n-1 basis}  
remains valid if one replaces the $n-1$  in a)  by $n-2$.
\item[b)] The result of this section remain valid 
if one just has a uniform lower bound on sectional curvature
and lower Ricci curvature bounds close enough to $0$.
However, it would be nice to know whether or not they remain valid without sectional curvature assumptions.
\end{enumerate}
\end{problem}

%%%%%%%%%%%%%%%%%%%%%%%%%%%%%%%%%%%%%%%%%%%%%%%%%%%%%%%%%%%%%%%%%%%%%%%%%%%%%%%%%%%%%%%%%%%%%%%%%
%%%%%%%%%%%%%%%%%%%%%%%%%%%%%%%%%%%%%%%%%%%%%%%%%%%%%%%%%%%%%%%%%%%%%%%%%%%%%%%%%%%%%%%%%%%%%%%%%
%%%%%%%%%%%%%%%%%%%%%%%%%%%%%%%%%%%%%%%%%%%%%%%%%%%%%%%%%%%%%%%%%%%%%%%%%%%%%%%%%%%%%%%%%%%%%%%%%%%%
%%%%%%%%%%%%%%%%%%%%%%%%%%%%%%%%%%%%%%%%%%%%%%%%%%%%%%%%%%%%%%%%%%%%%%%%%%%%%%%%%%%%%%%%%%%%%%%%%%%%

\section{Finiteness results}\label{sec: finiteness}

In the proof of the following theorem 
the work of Colding and Naber~\cite{CoNa} is key.

\begin{thm}[Normal Subgroup Theorem]\label{thm: normal sub} Given $n$, $D$, $\eps_1>0$ 
there are positive constants 
$\eps_0<\eps_1$  and $C$ such that the following holds: 

If $(M,g)$ is a compact  $n$-manifold $M$ with $\Ric>-(n-1)
$ and $\diam(M)\le D$, then  there is $\eps\in [\eps_0,\eps_1]$
and a normal subgroup $\gN\lhd \pi_1(M)$ 
such that for all $p\in M$ we have:
\begin{enumerate}
\item[$\bullet$] the image of 
$\pi_1(B_{\eps/1000}(p),p)\rightarrow \pi_1(M,p)$ contains $\gN$ and
\item[$\bullet$] the index of $\gN$ in the image of $\pi_1(B_{\eps}(p),p)\rightarrow \pi_1(M,p)$
is $\le C$.
\end{enumerate}
\end{thm}
To avoid confusion,
we should recall that a normal subgroup $\gN\lhd \pi_1(M,p)$ 
naturally induces  a normal subgroup $\gN\lhd \pi_1(M,q)$ for all $p,q \in M$.

If we put $\eps_1= \eps_{Marg}(n)$, the constant in the Margulis Lemma,
then by Theorem~\ref{intro: margulis},
the group $\gN$ in the above theorem 
contains a nilpotent subgroup of index $\le C_{Marg}$ which 
has a nilpotent basis of length $\le n$. After replacing $\gN$ 
by a characteristic subgroup of controlled index 
Theorem~\ref{normal margulis} follows.

\begin{proof}[Proof of Theorem~\ref{thm: normal sub}]
We will argue by contradiction. It suffices to rule out the existence 
of a sequence $(M_i,g_i)$ with the following properties. 

\begin{enumerate}
\item $\Ric_{M_i}>-(n-1)$, $\diam(M_i)\le D$;
\item for all $\eps\in (2^{-i},\eps_{1})$ 
and any normal subgroup $\gN\lhd \pi_1(M_i)$ 
one of the following holds
\begin{enumerate}
\item[(i)] 
$\gN$ is not contained in the image of $\pi_1(B_{\eps/1000}(p))$ for some $p\in M_i$ or 
\item[(ii)] the index of $\gN$ in the image of  
$\pi_1(B_{\eps}(q_i),q_i)\rightarrow \pi_1(M_i,q_i)$ 
is $\ge 2^i$ for some $q_i\in M_i$.
\end{enumerate}
\end{enumerate}

Since any 
subsequence of $(M_i,g_i)$ also satisfies this condition, we can assume
that the universal covers 
$(\tM_i,\pi_1(M_i),\tp_i)$ converge to $(Y,\G_\infty,\tp_\infty)$.
Put $\G_i:=\pi_1(M_i)$. 
For $\delta>0$ define 
\[
S_i(\delta):=\bigl\{g\in\G_i\mid d(q,gq)\le \delta\mbox{ for all $q\in B_{1/\delta}(\tp_i)$}\bigr\}\, \mbox{ and }\Gamma_{i\delta}:=\ml S_i(\delta)\mr\subset \Gamma_i 
\]
for $i\in \N\cup \{\infty\}$.

By Colding and Naber \cite{CoNa}, $\G_\infty$ is a Lie group. 
In particular, the component group $\G_{\infty}/\G_{\infty0}$ 
is discrete and hence there is some positive $\delta_0\le \min\{\eps_{1},\tfrac{1}{2D}\}$
such that  $\G_{\infty \delta}$ 
is given by the identity component of $\G_{\infty}$
for all $\delta \in (0, 2\delta_0)$.

It is easy to see that this property carries over 
to the sequence in the following sense.  We fix a large constant $R\ge 10$, to be determined later,
and put $\eps_0:=\delta_0/R$.
Then 
$\G_{i\delta}=\G_{i\eps_0}$ for all $\delta\in [\eps_0/1000,R\eps_0]$
and all $i\ge i_0(R)$.

We claim that $\G_{i\eps_0}$ is normal in $\Gamma_i$.
In fact,
$\Gamma_i$ can be generated by elements  which 
displace $\tp_i$ by at most $2D\le \tfrac{1}{\eps_0}$. 
{It is straightforward to check that for any $g_i\in \Gamma_i$ satisfying $d(g_i\tp_i,\tp_i)\le \frac{1}{\eps_0}$ we have }
$g_i S(\eps_0/2)g_i^{-1}\subset S(\eps_0)$. 
Thus, \[
g_i\Gamma_{i\eps_0/2}g_i^{-1}\subset \Gamma_{i\eps_0}=\Gamma_{i\eps_0/2}
 \]
and clearly the claim follows.

In summary, $\Gamma_{i\eps_0/1000}=\Gamma_{iR\eps_0}\lhd \Gamma_i$ 
with $\eps_0\le \eps_{1}$ for all $i\ge i_0(R)$. 
By the choice of our sequence this implies 
that for some $q_i\in M$ the index of $\Gamma_{i\eps_0}$ 
in  the image of $\pi_1(B_{\eps_0}(q_i),q_i))\rightarrow \pi_1(M_i)$   
is larger than $2^i$. 

Similarly to  Step 2 in the proof of the Margulis 
Lemma in section~\ref{sec: margulis}, this gives a contradiction provided we choose $R$ 
sufficiently large. In fact, the present situation is quite a bit 
easier as $\Gamma_{i R\eps_0}$ is normal in $\pi_1(M_i)$. 
\end{proof}

\begin{lem}\label{lem: cw}
\begin{enumerate}
\item[a)] Let
$\eps_0, D, H, n>0$. 
Then there is a finite collection of groups such that 
the following holds:

Let $(M,g)$
be a compact $n$-dimensional  manifold with $\diam(M)< D$, 
$\Ric(M)> -(n-1)$, $\gN\lhd \pi_1(M)$ 
and $\eps\ge\eps_0$. Suppose that 
\begin{enumerate}
\item[$\bullet$] 
The image $I$ of $\pi_1(B_{\eps}(q))\rightarrow \pi_1(M,q)$ satisfies
that the index $[I:I\cap \gN]$ is bounded by $H$, for all $q\in M$.
\item[$\bullet$] The group $\gN$ is in the image of 
$\pi_1(B_{\eps/1000}(q_0))\rightarrow \pi_1(M,q_0)$ for
some 
$q_0\in M$.
\end{enumerate}
Then $\pi_1(M)/\gN$ is isomorphic
to a group in the a priori chosen finite collection.
\item[b)](Bounded presentation)
Given a manifold as in a) one can find 
finitely many loops in $M$ based at $p\in M$ whose length and number
is bounded by an a priori constant such that 
the following holds: the loops 
project to a generator system of 
$\pi_1(M,p)/\gN$ and there are finitely many 
relations (with an a priori bound on the length of the words involved)
such that these words 
provide a finite presentation of the
group $\pi_1(M,p)/\gN$.
\end{enumerate}
\end{lem}

Instead of compact manifolds one could also consider manifolds 
satisfying the condition that the image of $\pi_1(B_{r_0}(p_1))\rightarrow \pi_1(B_{R_0}(p_1))$ 
is via the natural homomorphism isomorphic to $\pi_1(M_i)$, where 
$r_0$ and $R_0$ are given constants. In that case one would only need the Ricci  curvature 
bound in the ball $B_{R_0+1}(p_1)$.
In fact, the only purpose of the  Ricci curvature bound is 
to guarantee the existence of an $\eps/100$--dense finite set in $M$
with an a priori bound on the order of the set.

Of course, b) implies a). However, we prove a) first and then see that b) follows from the proof.
In the context of equivariant Gromov--Hausdorff convergence 
a bounded presentation  can often  be carried over to the limit.
That is why 
part b) has some advantages over a).

\begin{proof} 
{\em a).} The idea is to define for each such manifold 
a 2-dimensional finite CW-complex $C$ whose combinatorics 
is bounded by some a priori constants and which has 
the fundamental group $\pi_1(M)/\gN$. 

Choose  a maximal collection of points $p_1,\ldots,p_k$ in $M$
with pairwise distances $\ge \eps/100$.  Clearly, there is an a priori bound on $k$ depending only on $n,D$ and $\eps_0$. 
We also may assume that $p_1=q_0$.

The points $p_1,\ldots,p_k$ will represent the $0$-skeleton $C_0$ of our $CW$-complex.

For each point $p_i$ 
define $\gF_i$ as the finite group given by 
$\im_{\eps/5}(p_i)/(\gN\cap \im_{\eps/5}(p_i))$
where $\im_\delta(p_i)$ denotes the image of 
$\pi_1(B_{\delta}(p_i),p_i)\rightarrow \pi_1(M,p_i)$. 
For each of the at most $H$ elements in $\gF_i$ we 
choose a loop in $B_{\eps/5}(p_i)$ 
representing the class modulo $\gN$. 
We attach to $p_i$ an oriented $1$-cell for each loop
and call the loop the model path of the cell. 

For any two points $p_i,p_j$ with $d(p_i,p_j)< \eps/20$ 
we choose in the manifold a shortest path from $p_i$ to $p_j$. 
In the $CW$-complex we attach an edge (1-cell) between these 
two points. We call the chosen path the model path of the edge. 

This finishes the definition of the $1$-skeleton $C_1$ of our $CW$-complex.

For each point $p_i\in C_0$  we 
consider the part $A_i$ of the $1$-skeleton $C_1$
containing 
all points $p_j$ with $d(p_i,p_j)<\eps/2$ and 
also all $1$-cells connecting them including the cells attached initially.
If $A_i$ has  more than one connected component we replace $A_i$
by the connected component containing $p_i$.

We choose loops in $A_i$ based at $p_i$ representing free generators of the free group
$\pi_1(A_i,p_i)$. 
We now consider the homomorphism  $\pi_1(A_i,p_i)\rightarrow \pi_1(M,p_i)$
induced by mapping each loop to its model path 
in $M$.  
The induced homomorphism 
$\alpha_i\colon\pi_1(A_i,p_i)\rightarrow \pi_1(M,p_i)/\gN$
has finite image of order $\le H$ since the model paths are contained in $B_{\eps}(p_i)$. 

The kernel of $\alpha_i$ is a normal subgroup of finite index $\le H$
in $\pi_1(A_i,p_i)$. 
Thus, it is finitely generated and there 
is an a priori bound on the number of possibilities for the kernel. 
We choose a free finite generator system of the kernel 
and attach the corresponding  2-cells to the $CW$-complex 
for each generator of the kernel.

This finishes the definition of the $CW$-complex $C$. 

{Note that by construction, the total number of cells in our complex is bounded by  a constant depending only on $n,D, H$ and $\eps_0$ and that the attaching maps of the two cells are under control as well.}
Hence  there are only finitely many possibilities for the homotopy type of $C$.
Thus, we can finish the proof by establishing  
that $\pi_1(M)/\gN$  is isomorphic to 
$\pi_1(C)$. 

First notice that there is a natural 
surjective map $\pi_1(C)\to \pi_1(M)/\gN$: 
Consider the 1-skeleton $C_1$ of $C$. Its fundamental group is 
free and 
there is a homomorphism $\pi_1(C_1)\to \pi_1(M)\to \pi_1(M)/\gN$
induced by mapping a path onto its model path.
The two cells that where attached to the 1-skeleton
only added relations to the fundamental group contained in the kernel 
of this homomorphism.
Thus, the homomorphism induces a homomorphism 
$\pi_1(C,p_1)\rightarrow \pi_1(M,p_1)/\gN$. 
Moreover, it is easy to see that the homomorphism is surjective.

In order to show that the homomorphism is injective 
we need to show that homotopies 
in $M$ can be lifted in some sense to $C$. 

Suppose we have a closed curve $\gamma$  in the 1-skeleton $C_1$ which is based at $p_1$
and whose model curve $c_0$ in $M$ is homotopic to a curve in $\gN\subset \pi_1(M)$. 
We parametrize $\gamma$ on $[0,j_0]$ for some large integer $j_0$. 
We will assume that $\gamma(t)=p_j$ for some $j$ implies  $t\in [0,j_0]\cap \Z$ --
we do not assume that the converse holds.

Using that $p_1=q_0$  we can find a homotopy 
$H\colon [0,1]\times [0,j_0]\rightarrow M$, $(s,t)\mapsto H(s,t)=c_s(t)$
such that each $c_s$ is a closed loop at $p_1$, 
$c_0$ is the original curve and $c_1$ is a curve in the ball $B_{\eps/1000}(p_1)$.
If we choose $j_0$  large enough, we can 
arrange for 
$L(c_{s[i,i+1]})<\eps/1000$ for all $i=0,\ldots, j_0-1$, $s\in[0,1]$.

After an arbitrary small change of the homotopy 
we can assume  that 
$c_s(i)$ intersects the cut locus of the  $0$-dimensional submanifold
$\{p_1,\ldots,p_k\}$  only finitely 
many times for every $i=1,\ldots, j_0-1$. 
We can furthermore  assume that at any given parameter value 
$s$ there is at most  one $i$ such that  $c_s(i)$ is in the cut locus of 
$ \{p_1,\ldots,p_k\}$.

For each $c_s(i)$ we choose a point $\tc_s(i):=p_{l_i}(s)$ with 
minimal distance to $c_s(i)$. By construction 
$\tc_s(i)$  is piecewise constant in $s$ 
and at each parameter value $s_0$ at most one of the chosen points is changed. 
Moreover,
$\tc_s(0)=\tc_s(1)=p_1$ and $\tc_1(i)= p_1$.

We now define $\tc_{s|[i,i+1]}$ 
in three steps.
In the middle third of the interval we run through $c_{s|[i,i+1]}$ 
with triple speed. In the last third of the interval 
we run through a shortest path from $c_s(i+1)$ to $p_{l_{i+1}}(s)$.
In the first third of the interval 
we run through a shortest path from $p_{l_i}(s)$ to $c_s(i)$
and  
if there is a choice we choose the same path that was chosen in the last third 
of the previous 
interval.

Each curve $\tc_s$ is homotopic to $c_s$.
Moreover, $\tc_s(i) $  is piecewise constant equal to a vertex. 
Clearly, 
$s\mapsto \tc_s$ is only piecewise continuous. 

The 'jumps' for $\tc$ occur exactly at those parameter values at which 
$c_s(i)$  is the cut locus of the finite set $\{p_1,\ldots,p_k\}$ 
for some $i$.

We introduce the following notation: 
for two continuous curves $c_1,c_2$  in $M$ from $p$ to $q\in M$ 
we say that $c_1$ and $c_2$ are homotopic modulo $\gN$ 
if the homotopy class of the loop obtained by prolonging $c_1$ with the 
inverse parametrized curve $c_2$ is in $\gN$.

We define a 'lift' $\gamma_s$ in $C$ of the curve $\tc_s$ in $M$  as follows:
$\gamma_{s|[i,i+1]}$ should run in the one skeleton of 
our $CW$ complex from $\tc_s(i)$ to $\tc_s(i+1) \in C_0$ and 
the model path of $\gamma_{s|[i,i+1]}$ should be 
modulo $\gN$  
homotopic to
$\tc_{s|[i,i+1]}$. 
We choose $\gamma_{s[i,i+1]}$ in such a way that on the first half of the interval 
it runs through the edge from $\tc_s(i)$ to $\tc_s(i+1) $ 
and in order to get the right homotopy type for the model path, 
in the second half of the interval it runs 
through one  of the 1-cells that were attached initially to the vertex $\tc_s(i+1)\in C_0$.

Thus,  the model curve of $\gamma_{s|[i,i+1]}$ 
is  homotopic to $\tc_{s|[i,i+1]}$ modulo $\gN$. 
In particular, the model curve of  $\gamma_{s}$ and the loop $c_s$ 
are homotopic modulo $\gN$. 

The map  $s\mapsto \gamma_s$  is piecewise constant in $s$. 
We have to check that at the parameter where $\gamma_s$ jumps
the homotopy type of $\gamma_s$ does not change. 

However, first we want to check that the original curve $\gamma$ is  homotopic to 
$\gamma_0$ in $C$. 
By our initial assumption, there are integers 
$0=k_0<\ldots<k_h=j_0$ such that  
$\gamma(t)$ is in the $0$-skeleton if and only if $t\in \{k_0,\ldots,k_h\}$.
Going through the definitions 
above it is easy to deduce that $\gamma(k_i)=c_0(k_i)=\tc_0(k_i)=\gamma_0(k_i)$.
Thus, it suffices to check that $\gamma_{|[k_i,k_{i+1}]}$ 
is homotopic to $\gamma_{0|[k_i,k_{i+1}]}$.
We have just seen that the model curve of $\gamma_{0|[k_i,k_{i+1}]}$ is homotopic 
to $\tc_{0|[k_i,k_{i+1}]}$ modulo $\gN$.
This curve is homotopic to $c_{0|[k_i,k_{i+1}]}$, which is the model curve of 
$\gamma_{|[k_i,k_{i+1}]}$.

Moreover, the model curve of $\gamma_{|[k_i,k_{i+1}]}$ 
is in an $\eps/5$-neighborhood of $\gamma(k_i)$. 
By the definition of $\tc_0$ this implies that 
$\tc_{0|[k_i,k_{i+1}]}$ is in an $\eps/4$-neighborhood
of $\gamma(k_i)$. Therefore, $\gamma_{|[k_i,k_{i+1}]}$ 
as well as $\gamma_{0|[k_i,k_{i+1}]}$ only meet points 
of the zero skeleton with distance $< \eps/4$ 
to $\gamma(k_i)$. Thus, they are contained in $A_j$, where $j$ is 
defined by $p_j=\gamma(k_i)$.
Since their model curves 
are homotopic modulo $\gN$,
our rules for attaching 2-cells 
imply that $\gamma_{0|[k_i,k_i+1]}$ is  homotopic to 
$\gamma_{[k_i,k_{i+1}]}$  in $C$.

It remains to check that the homotopy type 
of $\gamma_s$ does not change at a parameter $s_0$ 
at which $s\mapsto \gamma_s$ is not continuous. 
By construction there is an $i$ such that 
$\gamma_{s|[0,i-1]}$ and $\gamma_{s|[i+1,j_0]}$ are independent of
$s\in [s_0-\delta,s_0+\delta]$
for some $\delta>0$.

The model curve of $\gamma_{s_0-\delta|[i-1,i+1]}$ 
is modulo $\gN$ homotopic to $\tc_{s_0-\delta|[i-1,i+1]}$
which in turn is homotopic to $\tc_{s_0+\delta|[i-1,i+1]}$
and, finally, this curve is modulo $\gN$ homotopic to 
the model curve of $\gamma_{s_0+\delta|[i-1,i+1]}$. 

Furthermore, the model curves of 
$\gamma_{s_0\pm \delta|[i-1,i+1]}$ are in 
an $\eps/2$-neighborhood of $\tc_{s_0}(i-1)$.
Since they are  homotopic modulo $\gN$,
this implies by definition of our complex 
that $\gamma_{s_0-\delta|[i-1,i+1]}$ is
homotopic to $\gamma_{s_0+ \delta|[i-1,i+1]}$  in $C$ .

Thus, the curve $\gamma$ is homotopic to 
$\gamma_1$, which is the point curve by construction. 

{\em b)} 
Since the number of $CW$-complexes constructed in a) is finite, 
we can think of $C$ as fixed. 
We choose loops in $C_1$  representing a free generator system of the free group $\pi_1(C_1,p_1)$. 
The model curves of these loops 
represent generators of $\pi_1(M,p_1)/\gN$ and the lengths of the loops are bounded.

For each of the attached 2-cells 
we consider the loop based at some vertex $p_i$ in the one skeleton 
given by the attaching map. 
We choose a path in $C_1$ from $p_1$ to $p_i$ 
and conjugate the loop back to $ \pi_1(C_1,p_1)$. 

We can express this loop as word in our free generators 
and if collect all these words for all 
2-cells, we get a finite presentation of $\pi_1(C,p_1)\cong \pi_1(M)/\gN$.
\end{proof}

It will be convenient for the proof of Theorem~\ref{intro: pi1 structure} 
to restate it slightly differently.

\begin{thm}\label{cor: pi1 structure}\label{thm: pi1 structure}
\begin{enumerate}
\item[a)] Given $D$ and $n$ there are finitely many
groups $\gF_1,\ldots, \gF_k$
such that the following holds: For any compact $n$-manifold $M$ with
$\Ric>-(n-1)$ and $\diam(M)\le D$ we can find 
a nilpotent normal subgroup $\gN\lhd \pi_1(M)$ 
which has a nilpotent basis of length $\le n-1$ such that 
$\pi_1(M)/\gN$ is isomorphic to one of the groups in our collection.
\item[b)] In addition to a) one can choose 
a finite collection of irreducible rational  representations 
$\rho_i^j\colon \gF_i\rightarrow \GL(n_i^j,\Q)$ 
($j=1,\ldots,\mu_i, i=1,\ldots,k$) such that 
for a suitable choice of the isomorphism $\pi_1(M)/\gN\cong \gF_i$ the following holds: There is a chain
of subgroups $\mathrm{Tor}(\gN)=\gN_0\lhd\cdots \lhd \gN_{h_0}=\gN$ 
which are all normal in $\pi_1(M)$ such that $[\gN,\gN_h]\subset \gN_{h-1}$
and $\gN_h/\gN_{h-1}$ is free abelian.  
Moroever, the action of $\pi_1(M)$ on $\gN$ by conjugation 
induces an action of $\gF_i$ on $ \gN_h/\gN_{h-1}$ 
and the induced  representation 
$\rho\colon \gF_i\rightarrow \GL\bigl((\gN_h/\gN_{h-1})\otimes_\Z\Q\bigr)$ 
is isomorphic to $\rho_i^j$ for a suitable $j=j(h)$, $h=1,\ldots,h_0$.
\end{enumerate}
Addendum: In addition one can assume in a) that $\rank(\gN)\le n-2$.
\end{thm}

\begin{proof}[Proof of Theorem~\ref{cor: pi1 structure}.]
We consider a contradicting sequence $(M_i,g_i)$, 
that is, we have $\diam(M_i,g_i)\le D$ and $\Ric_{M_i}\ge -(n-1)$ 
but the theorem is not true for any subsequence 
of $(M_i,g_i)$. 

We may assume that $\diam(M_i,g_i)=D$ 
and $(M_i,g_i)$ converges to some space  $X$. 
Choose $p_i\in M_i$ converging to a regular point $p\in X$. 
By the Gap Lemma~{\eqref{lem: gap}} there is 
some $\delta>0$ and $\eps_i\to 0$ 
with $\pi_1(M_i,p_i,\eps_i)=\pi_1(M_i,p_i,\delta)$.\\[1ex] 
{\bf Claim.} There is $C$ and $i_0$ such that $\pi_1(M_i,p_i,\delta)$ 
contains a subgroup of index $\le C$ which has a nilpotent basis 
of length $\le n-\dim(X)$ for all large $i\ge i_0$.\\[2ex]
We choose $\lambda_i<\tfrac{1}{\eps_i}$ with $\lambda_i\to \infty$
such that 
$(\lambda_iM_i,p_i)$ converges to $(\R^{u},0)$ with $u=\dim(X)$. 
It is easy to see 
that $(\lambda_i\tM_i/\pi_1(M_i,p_i,\eps_i),\tp_i)$ converges to
 $(\R^{u},0)$ as well. 
The claim now follows from the
Induction Theorem~{\eqref{thm: induction}} with $f_i^j=\id$.\\[-1.5ex]

We now apply the Normal Subgroup Theorem~{\eqref{thm: normal sub}} with $\eps_1=\delta$. 
There is a normal subgroup $\gN_i\lhd \pi_1(M_i,p_i)$, 
some positive constants $\eps$ and $C$ such that $\gN_i$ is contained in 
$\pi_1(B_{\eps_i}(p_i))\rightarrow \pi_1(M_i)$ 
and $\gN_i$ is contained in the image of 
$\pi_1(B_{\eps}(q_i),q_i)\rightarrow \pi_1(M_i,q_i)$ 
with index $\le C$ for all $q_i\in M_i$.

{ By the Claim above, } $\gN_i$ has a subgroup of bounded 
index with a nilpotent basis of length $\le n-\dim(X)\le n-1$. 
We are free to replace $\gN_i$ by a characteristic subgroup of bounded 
index and thus, we may assume that $\gN_i$ itself 
has a nilpotent basis of length $\le n-\dim(X)$.
By Lemma~\ref{lem: cw} a), the number of possibilities 
for $\pi_1(M_i)/\gN_i$ is finite. 
In particular, part {\em a)} of Theorem~\ref{thm: pi1 structure}
 holds for the sequence 
and thus either $b)$ or the addendum is the problem.

By Lemma~\ref{lem: cw} b), 
we can assume that $\pi_1(M_i,p_i)/\gN_i$ 
is boundedly presented. 
By passing to a subsequence  we can assume that 
we have the same presentation for all $i$.
Let $g_{i1}, \ldots, g_{i\beta}\in \pi_1(M_i)$ 
denote elements with bounded displacement  
projecting to our chosen generator system of 
$\pi_1(M_i)/\gN_i$. Moreover, 
there are finitely many words
in $g_{i1},\ldots,g_{i\beta}$  (independent of $i$) 
such that these words give a finite presentation of 
the group $\pi_1(M_i)/\gN_i$.

We can assume that $g_{ij}$ converges to $g_{\infty j}\in \lG\subset \Iso(Y)$
where  $(Y, \tp_\infty)$ is the limit of $(\tM_i,\tp_i)$. 
Let $\gN_\infty\lhd \lG$ be the limit of $\gN_i$.
By construction, $\lG/\gN_\infty$ is discrete.
Since $\pi_1(M_i)/\gN_i$ is boundedly presented, it follows that
there is an epimorphism
$\pi_1(M_i)/\gN_i\rightarrow \lG/\gN_{\infty}$ 
induced by mapping $g_{ij}$ to $g_{\infty j}$, for all large $i$. \\[-1ex]

{\em b).} We plan to show that a subsequence satisfies b).
We may assume that $\lG\star \tp_\infty$ is not connected
and by the Gap Lemma~{\eqref{lem: gap}} 
\begin{eqnarray}\label{rho0}
2\rho_0&:=&\min\{d(\tp_{\infty},g\tp_{\infty})\mid g\tp_{\infty}\notin \lG_{0}\star \tp_{\infty}\}>0.
\end{eqnarray}

Let $r$ be the maximal nonnegative integer such that the following holds.
After passing to a subsequence there is 
 a subgroup $\gH_i\subseteq \gN_i$ 
of rank $r$  satisfying 
\begin{enumerate}
 \item[$\bullet$] $\gH_i\lhd \pi_1(M_i)$, $\gN_i/\gH_i$ 
is torsion free and
\item[$\bullet$]  there is a chain $\Tor(\gN_i)=\gN_{i0}\lhd\cdots\lhd\gN_{ih_0}=\gH_i$ 
such that each group $\gN_{ih}$ is normal in $\pi_1(M_i)$, 
$[\gN_i,\gN_{ih}]\subset \gN_{ih-1}$, $\gN_{ih}/\gN_{ih-1}$ is free 
abelian and for each  $h=1,\ldots,h_0$ the induced representations 
of $\gF=\pi_1(M_i)/\gN_i$ in $ (\gN_{ih}/\gN_{ih-1})\otimes_{\Z}\Q$ 
and in $ (\gN_{jh}/\gN_{jh-1})\otimes_{\Z}\Q$ are equivalent 
for all $i,j$. 
\end{enumerate}
Here we used implicitly that we have a natural isomorphism between 
$\pi_1(M_i)/\gN_i$ and $\pi_1(M_j)/\gN_j$ in order to talk about
equivalent representations.

Notice that these statements hold for $\gH_i=\Tor(\gN_i)$.
We need to prove that $\gH_i=\gN_i$.
Suppose on the contrary that $\rank(\gH_i)<\rank(\gN_i)$.
Consider $\hM_i:= \tM_i/\gH_i$ endowed with the action of 
$\Gamma_i:=\pi_1(M_i)/\gH_i$, and let $\hp_i$ be a lift of $p_i$ to $\hM_i$.
By construction $\hgN_i:=\gN_i/\gH_i$ is a torsion free 
normal subgroup of $\Gamma_i$. 

We claim that there is a central subgroup $\gA_i\subset \hgN_i$ 
of positive rank which is normal in $\Gamma_i$ and is 
generated by $\{a\in \gA_i\mid d(\hp_i,a\hp_i)\le d_i\}$ 
for a sequence $d_i\to 0$:
Recall that $\hgN_i$ is generated by 
$\{a\in \hgN_i\mid d(\hp_i,a\hp_i)\le \eps_i\}$. 
If $\hgN_i$ is not abelian this implies 
that $[\hgN_i,\hgN_i]$ is generated by
 $\{a\in [\hgN_i,\hgN_i]\mid d(\hp_i,a\hp_i)\le 2n\eps_i\}$. 
In fact an arbitrary commutator in $\hgN_i$ can be expressed as
a product of iterated commutators of a generator system 
and since $\hgN_i$ has a nilpotent basis of length $\le n-1$ 
one only needs to iterate at most $n-1$ times.
If $[\hgN_i,\hgN_i]$ is not central in $\hgN_i$ one replaces
it by $[\hgN_i,[\hgN_i,\hgN_i]]$. After finitely many similar 
steps this proves the claim.

For each positive integer $l$ 
put $l\cdot \gA_i=\{g^l\mid g\in\gA_i\}\lhd \Gamma_i$.
We define $l_i=2^{u_i}$ as the  maximal power of $2$ 
such that there is an element in  
$l_i\cdot \gA_i$ which displaces 
$\hp_i$ by at most $\rho_0$.
Thus, any element in $\gL_i:=l_i\cdot \gA_i$ displaces $\hp_i$ by at least
$\rho_0/2$.

After passing to a subsequence we may assume that 
$(\hM_i,\Gamma_i,\hp_i)$ converges to $(\hY,\hG,\hp_\infty)$ 
and  the action of $\gL_i$ 
converges to an action of some discrete abelian subgroup $\gL_{\infty}\lhd  \hG$. 
Finally we let $\hgN_\infty\lhd \hG$ denote the limit group 
of $\hgN_i$.
Let $g_i\in \gL_i$ be an element which displaces 
$\hp_i$ by at most $\rho_0$. 
Combining $d(g_i^k\hp_i,g_i^{k+1}\hp_i)\le \rho_0$  
with our choice of $\rho_0$, see \eqref{rho0}, 
gives that the sets $\{g_i^k\hp_i \mid k\in \Z\}$ converge 
to a discrete subset in the identity component of the limit orbit 
$\hG_{0}\star\hp_\infty$. Therefore,
\[
\{g\in \gL_\infty\mid g\star \hp_\infty\in \hG_{0}\star \hp_{\infty}\}
\]
is discrete and infinite.
Let $\gK\subset \hG$ denote the isotropy group of $\hp_{\infty}$.
By  Colding and Naber $\gK$ is a Lie group and thus 
it only has finitely many connected components.
Hence $\gL'_\infty:=\gL_\infty\cap \hG_0\lhd \hG$ 
is infinite as well. 
Since the abelian group $\gL_\infty'$ is a discrete subgroup of
 a connected Lie group, it is finitely generated.

We choose a free abelian subgroup $\hgL\subset \gL_{\infty}'$ 
of positive rank which is normalized by 
$\hG$
such that the induced representation of $\hG$ in $\hgL\otimes_\Z\Q$  
is irreducible. Notice that $\hgL\subset \gL_{\infty}$ commutes with 
$\hgN_\infty$. Hence we can view this 
as a representation of $\hG/\hgN_\infty$.
Recall that $\pi_1(M_i)/\gN_i\cong \Gamma_i/\hgN_i$ 
is boundedly represented. 
Let $\hg_{i1},\ldots,\hg_{i\beta}\in \Gamma_i/\hgN_i$ denote the images 
of $ g_{i1},\ldots,g_{i\beta}$. 
There is an epimorphism
$ \Gamma_i/\hgN_i\rightarrow \hG/\hgN_\infty$
induced by sending $\hg_{im}$ to its limit element 
$\hg_{\infty m}\in \hG$ for all large $i$. 
Thus, $\hgL$ is also naturally endowed with a representation 
of $\Gamma_i/\hgN_i$. 

Let $b_1,\ldots,b_l\in \hgL\cong \Z^l$ be a basis. 
For large $i$ there are unique elements $h_i(b_j)\in \gL_i$ 
which are close to $b_j$, $j=1,\ldots,k$. 
We extend $h_i$ to a $\Z$-linear map $h_i\colon \hgL\rightarrow \gL_i$. 

We plan to prove that $h_i\colon \hgL\rightarrow \gL_i$ 
is equivariant for large $i$. 
For any given linear combination $\sum_{\alpha=1}^lz_\alpha b_\alpha$ ($z_j\in \Z$)
we know that $\sum_{\alpha=1}^l z_\alpha h_i(b_\alpha)$ is the unique element
in $\gL_i$ which is close to $\sum_{\alpha=1}^lz_\alpha b_\alpha$
for all large $i$.
For each $\hg_{\infty m}$ and each $b_j$ 
we have $\hg_{\infty m} b_j\hg_{\infty m}^{-1}=\sum_{\alpha =1}^lz_\alpha b_\alpha$
for $z_\alpha\in \Z$ (we suppress the dependence on $m$ and $j$). 
We have just seen that 
$h_i\bigl(\sum_{\alpha =1}^lz_\alpha b_\alpha\bigr)$ is close to 
$\hg_{\infty m} b_j\hg_{\infty m}^{-1}$ for all large $i$.
On the other hand, $\hg_{i m} h_i(b_j)\hg_{i m}^{-1}\in \gL_i$ is the unique element 
in $\gL_i$  close to $\hg_{\infty m} b_j\hg_{\infty m}^{-1}$ for all large $i$. 

In summary, $h_i(\hg_{\infty m} b_j\hg_{\infty m}^{-1})=\hg_{i m}h_i(b_j)\hg_{im}^{-1}$ 
for $m=1,\ldots,\beta$, $j=1,\ldots,k$ and all large $i$. 
This shows that $h_i$ is equivariant with respect to the representation. 

Thus, $h_i(\hgL)$ is a normal subgroup of $\Gamma_i$
and the induced representation of $\gF=\pi_1(M_i)/\gN_i=\Gamma/\hgN_i$ 
in  $h_i(\hgL)\otimes_\Z \Q$ and is isomorphic to the one 
in $h_j(\hgL)\otimes_\Z \Q$ for all large $i, j$.

There is a unique subgroup  $\gA_i'$ 
of $\hgN_i$ such that  $h_i(\hgL)$ has finite index in 
$ \gA_i'$ and $\hgN_i/\gA_i'$ is torsion free. 
Let $\gN_{ih_0+1}\lhd \gN_i$ denote the inverse image of 
$\gA_i'\subset \hgN_i=\gN_i/\gH_i$. 

Clearly, the representation of $\pi_1(M_i)/\gN_i$ 
in $(\gN_{ih_0+1}/\gH_i)\otimes_{\Z}\Q$ is isomorphic to the 
one in $(\gN_{jh_0+1}/\gH_j)\otimes_{\Z}\Q$ for all large $i,j$ -- 
a contradiction to our choice of $\gH_i$.\\[-1ex]

{\em Proof of the addendum.} Thus,
the sequence satisfies a) and b)   but not the addendum.
Recall that $\gN_i$ has a nilpotent basis of length $\le n-\dim(X)$. 
Therefore $\dim(X)=1$ and 
$\gN_i$ is a torsion free group of rank $n-1$.
Since $X$ is one-dimensional we deduce 
that $\lG/\gN_{\infty}$ is virtually cyclic, that is, it contains  a cyclic subgroup 
of finite index. 
This in turn implies that  $\pi_1(M_i)/\gN_i$ 
is virtually cyclic and by a) we can assume that it is a fixed group $\gF$.

The result (contradiction) will now follow algebraically from b):
Let  $\rho_j\colon \gF\rightarrow \GL(n_j,\Q)$ 
be a finite collection of irreducible rational  representations 
$j=1,\ldots,j_0$. 
Consider, for all nilpotent 
torsion free groups $\gN$ of rank $n-1$, all
short exact sequences 
\[
\gN\rightarrow \Gamma\rightarrow \gF
\]
with the property that there is a 
chain $\{e\}=\gN_0\lhd  \cdots \lhd \gN_{h_0}=\gN$ 
such that $[\gN,\gN_h]\subset \gN_{h-1}$, each $\gN_h$ 
is normal in $\Gamma$,
$\gN_{h}/\gN_{h-1}$ is free abelian group of positive rank
and the induced representation of $\gF$ in $(\gN_{h}/\gN_{h-1})\otimes_\Z\Q$
is in the finite collection.

Using that $\gF$ is virtually cyclic, 
it is now easy to see that this leaves 
only finitely many possibilities for the isomorphism type 
of $\Gamma/\gN_{h_0-1}$. 
This shows that if we replace $\gN$ by $\gN_{h_0-1}$
 we are still left with finitely many possibilities  in a). 
 -- a contradiction.
\end{proof}

\begin{cor}\label{cor: finite pi1} Given $n$, $D$ there is a constant 
$C$ such that any {\em finite} fundamental 
group of an $n$-manifold $(M,g)$ with $\Ric>-(n-1)$, $\diam(M)\le D$ 
contains a nilpotent subgroup of index $\le C$ 
which has a nilpotent basis of length $\le (n-1)$. 
\end{cor}

\begin{example} Our theorems  rule out some rather innocuous families of groups as fundamental groups  of manifolds with lower Ricci curvature bounds.
\begin{enumerate}
           \item[a)] Consider a homomorphism $h\colon \Z\rightarrow \GL(2,\Z)$ 
whose image does not contain a unipotent subgroup of finite index. 
Put $h_d(x)=h(dx)$ and consider the group 
$\Z\ltimes_{h_d}\Z^2$, $d\in \fN$. 
By part b) of Theorem~\ref{cor: pi1 structure} the following holds:

In a given dimension $n$ and for a fixed diameter bound $D$ only
finitely many of these groups can be realized as fundamental groups 
of manifolds with $\Ric>-(n-1)$ and $\diam(M)\le D$.
\item[b)] Consider the action of $\Z_{p^{k-1}}$ on $\Z_{p^k}$ 
induced by $\bigl(1+p^k\Z\bigr)\mapsto \bigl((1+p)+p^k\Z\bigr)$. 
In a given dimension $n$ we have that for $k\ge n$ 
and $p>C_{Marg}$ (constant in the Margulis Lemma) the group $\Z_{p^{k-1}}\ltimes \Z_{p^k}$
can not be a subgroup of a compact manifold  with almost nonnegative Ricci curvature,
since it does not have a nilpotent basis of length $\le n$. 

Similarly,
by Corollary~\ref{cor: finite pi1}, for a given $D$ and $n$ there are only finitely many primes $p$
such that $\Z_{p^{n-1}}\ltimes \Z_{p^n}$ is isomorphic to the fundamental group 
of an $n$-manifold $M$ with $\Ric>-(n-1)$ and $\diam(M)\le D$. 
\end{enumerate}

\end{example}

\begin{problem} Let $D>0$. Can one find a finite collection
of $3$-manifolds  such that 
for any $3$-manifold $M$ with $\Ric>-1$ and $\diam(M)\le D$,
there is a manifold $T$ in the finite collection
and a finite normal covering $T\rightarrow M$ 
for which the covering group contains a cyclic subgroup 
of index $\le 2$?
\end{problem}

%%%%%%%%%%%%%%%%%%%%%%%%%%%%%%%%%%%%%%%%%%%%%%%%%%%%%%%%%%%%%%%%%%%%%%%%%%%%%%%%%%%%%%%%%%%%%
% %%%%%%%%%%%%%%%%%%%%%%%%%%%%%%%%%%%%%%%%%%%%%%%%%%%%%%%%%%%%%%%%%%%%%%%%%%%%%%%%%%%%%%%%%%%%%
%%%%%%%%%%%%%%%%%%%%%%%%%%%%%%%%%%%%%%%%%%%%%%%%%%%%%%%%%%%%%%%%%%%%%%%%%%%%%%%%%%%%%%%%%%%%%%%%%
%%%%%%%%%%%%%%%%%%%%%%%%%%%%%%%%%%%%%%%%%%%%%%%%%%%%%%%%%%%%%%%%%%%%%%%%%%%%%%%%%%%%%%%%%%%%%%%%
%%%%%%%%%%%%%%%%%%%%%%%%%%%%%%%%%%%%%%%%%%%%%%%%%%%%%%%%%%%%%%%%%%%%%%%%%%%%%%%%%%%%%%%%%%%%

\section{The Diameter Ratio Theorem}\label{sec: diam}
The aim of this section is to prove Theorem~\ref{thm: diamratio}.

\begin{lem}\label{lem: conv to torus}
Let $X_i$ be a sequence of compact inner metric
spaces Gromov--Hausdorff converging to a torus $\gT$.
Then $\pi_1(X_i)$ is infinite for large $i$. 
\end{lem}

The proof of the lemma is an easy exercise.

\begin{proof}[Proof of Theorem~\ref{thm: diamratio}]
Suppose, on the contrary, that 
$(M_i,g_i)$ is a sequence of compact  manifolds 
with $\diam(M_i)=D$,  $\Ric_{M_i}>-(n-1)$, 
$\#\pi_1(M_i)<\infty$ and the diameter of the universal 
cover $\tM_i$ tends to infinity. 

By Corollary~\ref{cor: finite pi1}, we know that 
$\pi_1(M_i)$ contains a subgroup of index $\le C(n,D)$ 
which has a nilpotent basis of length $\le (n-1)$.

Thus, we may assume that $\pi_1(M_i)$ itself 
has a nilpotent basis of length $u<n$.  
We also may assume that $u$ is minimal with 
the property that a contradicting sequence exist.
Put 
\[
\hM_i:=\tM_i/[\pi_1(M_i),\pi_1(M_i)].
\]
Clearly $ [\pi_1(M_i),\pi_1(M_i)]$ has a nilpotent basis 
of length $\le u-1$. 
By construction $\hM_i$ can not be a contradicting sequence 
and therefore $\diam(\hM_i)\to \infty$. 

Let $\gA_i:=\pi_1(M_i)/[\pi_1(M_i),\pi_1(M_i)]$ 
denote the deck transformation group 
of the normal covering $\hM_i\rightarrow M_i$.\\[2ex]
{\bf Claim.} The rescaled sequence $\tfrac{1}{\diam(\hM_i)}\hM_i$ 
is precompact in the Gromov--Hausdorff topology and any limit 
space has finite Hausdorff dimension.\\[2ex]
The problem is, of course, that no lower curvature
bound is available after rescaling. The claim will follow from 
a similar precompactness result for certain Cayley graphs of $\gA_i$. 

Recall that $\diam(M_i)=D$. 
Choose a base point $\hp_i\in \hM_i$. 
Let $f_1,\ldots,f_{k_i}\in \gA_i$ be an enumeration of 
all the elements $a\in \gA_i$ 
with $d(\hp_i,a\hp_i)<10D$. 

There is no bound for $k_i$, but
clearly $f_1,\ldots,f_{k_i}$ is a generator system of 
$\gA_i$. 
We define a weighted metric on the abelian group $\gA_i$ as follows: 
\[
d(e,a):=\min\left\{\sum_{j=1}^{k_i}|\nu_j|\cdot d(\hp_i,f_j\hp_i)\,\,\Bigm|\,\, \prod_{j=1}^{k_i}f_j^{\nu_j}=a\right\}
\,\,\mbox{ for all $a\in \gA_i$}
\]
and $d(a,b)=d(ab^{-1},e)$ for all $a,b\in\gA_i$.
{Note that this metric coincides with the restriction to $\gA_i$ of the inner metric on its Cayley graph where each edge corresponding to $f_j$ is given the length $d(\hp_i,f_j\hp_i)$.}
It is easy to see that the map 
\[
\iota_i\colon \gA_i\rightarrow \hM_i,\,\, a_i\mapsto a_i\hp_i
\]
is a quasi isometry  with uniform control on the constants involved. 
In fact, there is some $L$  independent of $i$ such that 
\[
\tfrac{1}{L}d(a,b)\le d(\iota_i(a),\iota_i(b))\le Ld(a,b)\,\,\mbox{ for 
all $a,b\in \gA_i$.}
\]
and the image is of $\iota_i$ is $D$-dense.

Therefore, it suffices to show 
that $\tfrac{1}{\diam(\gA_i)}\gA_i$ 
is precompact in the Gromov--Hausdorff topology and 
all limit spaces of convergent subsequences are finite dimensional. 
For the proof  we need\\[2ex]
{\bf Subclaim.} There is an $R_0>D$ (independent of $i$) such that 
the homomorphism
$h\colon \gA_i\rightarrow \gA_i,\,\,\, x\mapsto x^2$
satisfies that  the $R_0$ neighborhood of $h(B_R(e))$ 
contains $B_{\frac{3R}{2}}(e)$, for all $R$ and all $i$.\\[2ex]
Using that $B_{10D}(e)\subset \gA_i$ is isometric to a
subset of $B_{10D}(\hp_i)$, we can employ the Bishop--Gromov inequality
in order to find a universal constant $k$ such that 
$B_{10D}(e)$ does not contain $k$ points with pairwise distance 
$\ge D$.
Put $R_0:= 10D\cdot k$. 

There is nothing to prove if $R\le 2R_0/3$. 
Suppose the statement holds for $R'\le R-D$. We claim it holds for $R$. 

Let $a\in B_{3R/2}(e)\setminus B_{R_0}(e)$. 
By the definition of the metric on $\gA_i$, there are
$g_1,\ldots,g_l\in \gA_i$ with $d(e,g_j)\le 10 D$,
$a=\prod_{j=1}^{l}g_j$ and
$d(e,a)=\sum_{j=1}^{l}d(\hp_i,g_j\hp_i)$. If there is any choice 
we assume in addition that $l$ is minimal with these 
properties. 
By assumption $l\ge \tfrac{R_0}{10D}=k$. 
By the choice of $k$, after a renumbering,
we may assume that $d(g_1,g_2)\le D$. 
Our assumption on  $l$ being minimal implies that $d(e,g_1g_2)> 10D$
and we may assume $4D\le d(e,g_1)\le d(e,g_2)$.
Thus, \begin{eqnarray*}
d\bigl(e, g_{1}^{-2}a\bigr)&\le& d\bigl(e, (g_{1}g_2)^{-1}a\bigr)+D\,\,\,\,
=\,\,d(e, a)-d(e,g_1)-d(e,g_2)+ D\\ &\le& d(e, a)-2d(e,g_1)+ D\,\le\, d(e,a)-7D.
\end{eqnarray*}
By assumption, this implies that $g_{1}^{-2}a$ has distance $\le R_0$ 
to some $b^2\in \gA_i$ with $d(e,b)\le \tfrac{2}{3}\bigl( d(e, a)-2d(e,g_1) +D\bigr)$.
Consequently, $a$ has distance $\le R_0$ 
to $g_{1}^{2}b^2\in \gA_i$
with $d(e,g_1b)< \tfrac{2}{3} d(e, a)$. This finishes the proof of the subclaim.\\[-1ex]

Since the Ricci curvature of $\hM_i$ is bounded below
and the $L$-bilipschitz embedding $\iota_i$ maps the ball 
$B_{50R_0}(e)\subset \gA_i$ to a subset of $B_{50R_0}(\hp_i)$,
we can use Bishop--Gromov once more to see that 
there is a number $Q>0$ (independent of $i$)
such that  the ball $B_{50R_0}(e)\subset \gA_i$ 
can be covered by $Q$ balls of radius $R_0$ for all $i$.

We now claim
that the ball $B_{2R}(e)\subset \gA_i$
can be covered by $Q$ balls of radius $R$ for all $R\ge 20R_0$ 
and all $i$. This will clearly imply that
$\tfrac{1}{\diam(\gA_i)}\gA_i$  is precompact in the Gromov--Hausdorff topology 
and that the limits have finite Hausdorff dimension.

Consider the homomorphism
\[
h^{8}\colon \gA_i\rightarrow \gA_i,\,\,\, x\mapsto x^{16}.
\]
{It is obviously $16$-Lipschitz since $\gA_i$ is abelian.}
Choose a maximal collection of points $p_1,\ldots,p_{l_i}\in 
B_{\frac{2R}{5}}(e)\subset \gA_i$ with pairwise distances $\ge 2R_0$. 

From the subclaim
it easily follows that $B_{3R}(e)\subset \bigcup_{j=1}^{l_i}B_{50R_0}(h^8(p_j))$.
Hence we can cover $B_{3R}(e)$ by $l_i\cdot Q$ 
balls of radius $R_0$. 
Consider now a maximal collection of points 
$q_1,\ldots,q_h\in B_{2R}(e)$ with pairwise distances $\ge R$. 
In each of the balls $B_{\frac{2R}{5}}(q_j)$ we can 
choose $l_i$ points with pairwise distances $\ge 2R_0$. 
Thus, $B_{3R}(e)$ contains $h\cdot l_i$ points with pairwise 
distances $\ge 2R_0$.
Since we have seen before that $B_{3R}(e)$ can be covered by 
$l_i\cdot Q$ 
balls of radius $R_0$, this implies $h\le Q$  as claimed.\\[-1ex]

Thus, $\tfrac{1}{\diam(\hM_i)}\hM_i$ is precompact in the
Gromov--Hausdorff topology. After passing to a subsequence
we may assume that $\tfrac{1}{\diam(\hM_i)}\hM_i\to \gT$. 
Notice that $\gT$ comes with a transitive action of an abelian group. 
Therefore $\gT$ itself has a natural group structure.
Moreover, $\gT$ is an inner metric space and the Hausdorff dimension of 
$\gT$ is finite. 
Like Gromov in \cite{G6} we can now deduce from a theorem 
of Montgomery Zippin~\cite[Section 6.3]{MoZi} that $\gT$ is a Lie group and thus, a torus.

By Lemma~\ref{lem: conv to torus}, this shows $\pi_1(\hM_i)$ 
is infinite for large $i$ -- a contradiction.

\end{proof}

\section*{Final Remarks}
We would like to mention that in \cite{Wi11} the following partial 
converse of the Margulis Lemma (Theorem~\ref{intro: margulis}) 
is proved.

\begin{thm*} Given $C$ and $n$ there exists $m$ 
such that the following holds: Let $\eps>0$, 
and let $\Gamma$ be a group  containing a nilpotent subgroup $\gN$ of index 
$\le C$ which has a nilpotent basis of length $\le n$. 
Then there is a compact $m$-dimensional manifold $M$ with 
sectional curvature $K>-1$
and a point $p\in M$ such that 
$\Gamma$ is isomorphic to the image of the homomorphism 
\[
 \pi_1\bigl(B_{\eps}(p),p\bigr)\rightarrow \pi_1(M,p).
\]
\end{thm*}

Apart from the issue of finding the optimal dimension 
another difference to Theorem~\ref{intro: margulis}
is that this theorem uses the homomorphism 
to $\pi_1(M)$ rather than to $\pi_1(B_1(p),p)$. 
This actually allows for more flexibility (by adding relations 
to the fundamental group at large distances to $p$). 
This is  
the reason why the following problem remains open.

\begin{problem} The most important problem 
in the context of the Margulis Lemma for manifolds with lower Ricci curvature bound 
that remains open is whether or whether not
one can arrange in Theorem~\ref{intro: margulis} for the torsion of $\gN$ to be abelian. 
We refer the reader to \cite{KPT} for some related conjectures  
for manifolds with almost nonnegative sectional curvature.  
\end{problem}

\small
\bibliographystyle{alpha}
%\bibliography{master}

\hspace*{1em}\\[-1ex]
\begin{footnotesize}
\hspace*{0.3em}{\sc 
Department of Mathematics,
University of Toronto,
Toronto, Ontario, Canada}\\
\hspace*{0.3em}{\em E-mail address: }{\sf vtk@math.toronto.edu}\\
\hspace*{0.3em}{\sc University of M\"unster,
Einsteinstrasse 62, 48149 M\"unster, Germany}\\
\hspace*{0.3em}{\em E-mail address: }{\sf wilking@math.uni-muenster.de}
\end{footnotesize}

\end{document}